\documentclass[10pt]{article}
\usepackage{amssymb,amsmath,textcomp}
\usepackage{enumerate}
\usepackage{hyperref}
\usepackage{mathrsfs}
\usepackage{amsthm}
\usepackage{xcolor}
\title{The mean-field limit of non-exchangeable particle systems with non-conservative dynamics and adaptive weights}

\author{Katarzyna Ryszewska\footnote{Institute of Mathematics, Polish Academy of Sciences,
\'{S}niadeckich 8, 00-656 Warsaw, Poland \\
Katarzyna.Ryszewska@pw.edu.pl \\
The author was partially supported by National Science Center, Poland through 
UMO-2024/54/A/ST1/00159 Grant.
}}

\newtheorem{theo}{Theorem}
\newtheorem{defi}{Definition}
\newtheorem{prop}{Proposition}
\newtheorem{lemma}{Lemma}
\newtheorem{coro}{Corollary}
\newtheorem{remark}{Remark}

\def\divv{\operatorname {div}}

\DeclareMathOperator*{\esssup}{ess~sup}

\newcommand{\eqq}[2]{\begin{equation}  #1  \label{#2}\end{equation}    }
\newcommand{\hd}{\hspace{0.2cm}}

\newcommand{\m}[1]{\mbox{#1}}
\newcommand*{\norm}[1]{\left\Vert{#1}\right\Vert}
\newcommand*{\abs}[1]{\left\vert{#1}\right\vert}
\newcommand{\vf}{\varphi}
\newcommand{\ve}{\varepsilon}
\newcommand{\R}{\mathbb{R}}
\newcommand{\nic}[1]{ }

\newcommand{\izj}{\int_{0}^{1}}
\newcommand{\ird}{\int_{\mathbb{R}^d}}
\newcommand{\al}{\alpha}

\newcommand{\sn}{\frac{1}{N}\sum_{i=1}^N}
\newcommand{\xb}{\bar{X}}

\newcommand{\p}{\mathbb{P}}
\newcommand{\E}{\mathbb{E}}
\newcommand{\Tree}{\mathcal{T}}
\newcommand{\wij}{w_{ij}}
\newcommand{\wijb}{\bar{w}_{ij}}
\newcommand{\wji}{w_{ji}}
\newcommand{\wjib}{\bar{w}_{ji}}
\newcommand{\fji}{f^j_i}
\newcommand{\fik}{f^i_k}
\newcommand{\W}{\mathcal{W}}
\newcommand{\mI}{\mathbb{I}}

\begin{document}
\date{}
\maketitle
\begin{abstract}
In the paper we introduce vector-valued dynamic extended graphons in order to obtain the mean-field limit for a certain class of problems with adaptive weights. We impose fairly general assumptions on the matrix of initial connections, which allows us to cover the case of sparse graphs.  
\end{abstract}
\vspace{0.7em}
\begin{center}
{\bf AMS subject classification:} 35Q83, 46G10, 35Q70, 35R02, 35Q49, 35R06
\end{center}

\noindent{\bf Keywords:} mean-field limit, adaptive weights, vector-valued measures, interacting particles, non-exchangeable systems, extended graphons, non-conservative problems
\section{Introduction and main result}

The aim of this paper is to extend the ideas introduced in \cite{JPS2025} to obtain the mean-field limit for the non-conservative, non-exchangeable multi-agent system with adaptive weights for possibly sparse graphs of initial connections. To begin with, it is worth noting that establishing mean-field limits in the presence of adaptive weights remains largely underdeveloped, given the severe mathematical challenges posed by the co-evolution of particle states and their interaction intensities. On the other hand, such dynamic frameworks have found widespread applications across diverse fields (such as neuroscience and sociology) where interaction rules must inherently adapt to the state of the system. We refer to \cite{fizyka} for an extensive survey of such adaptive systems and their practical implementations.

To date, rigorous results on adaptive networks have largely been restricted to two settings. The first one is the opinion formation model, where each agent, in addition to its opinion, is characterized by its time evolving impact factor, introduced in  \cite{MPD2019}, and the Kuramoto model. In both regimes the first results focused on continuum limit \cite{AD2021},   \cite{GKX2023}, \cite{Throm}. For a more comprehensive treatment regarding the continuum limit we refer to the overviews \cite{Aover} and \cite{ADover}. The results addressing mean-field limit for adaptive networks appeared only recently \cite{A2026}, \cite{AD2021}, \cite{BCG2024}, \cite{Poyato2026}, \cite{D2022}, \cite{GKX2025},  \cite{Throm2},  \cite{Zhou}. To the best of our knowledge, only the contribution \cite{GKX2025} deals with possibly non-dense graphs. In \cite{GKX2025}  the mean-field limit for a co-evolutionary Kuramoto model with the evolution of connections given in the form $\partial_t w_{ij} = -\ve(w_{ij} + h(\phi_j - \phi_i))$, where $\phi_i$ is a phase of oscillator and $h,\ve$ are given, is established. Thanks to this decoupling between the weights and phases, the authors are able to   
employ theory of digraph measures and the machinery developed in \cite{Kuehn} to obtain the result imposing relatively mild conditions on initial network structure.

Building upon the framework of \cite{JPS2025}, we operate under fairly mild structural conditions on the initial interaction weights $w_{ij}^0$, which are flexible enough to accommodate sparse graphs. On the canonical probabilistic space $(\R^d,\mathcal{B}(\R^d),\p)$ we consider the problem

\begin{equation}\label{part}
\begin{array}{lll}
(i) & \partial_t X_i(t) = \sum_{k=1}^N w_{ik}(t) K_{1}(X_i(t)-X_k(t)) & t \in (0,t_*),\\
(ii) & \partial_t \wji(t) = \wji(t)[A-\sum_{k=1}^N w_{ik}(t)K_{2}(X_i(t)-X_k(t))] & t \in  (0,t_*),\\
(iii) & X_i(0) = X_{i}^0, \hd  \wij(0) = \wij^0,
\end{array}
\end{equation}
for $i,j=1,\dots, N$, $N \in \mathbb{N}$, $t_* > 0$,  where $X_i^0:\mathbb{R}^d \rightarrow \mathbb{R}^d $ and $\wij^0:\mathbb{R}^d\rightarrow \mathbb{R}_{+}$ are given random variables and $A>0,K_1,K_2$ are given deterministic data. Note that, here and henceforth $\R_{+}:=[0,\infty)$ and, for clarity, we suppress the index $N$ in the notation of random variables. \\

Model (\ref{part}) is the extension of the opinion dynamics model discussed in \cite{nasza}. The adaptation of the methodology from \cite{nasza} to the dynamic weights framework requires significant revisions to the theoretical foundations established in \cite{JPS2025}. While the general flow of the proof follows the methodology of \cite{JPS2025}, introducing extended graphons with values in Banach spaces leads to new complexities. As a result, the mathematical machinery must be carefully reformulated to address the case of measures with values in Banach spaces. It should be emphasized that the resulting mean-field limit  takes a fundamentally different form compared to prior works. Motivated by \cite{nasza} we consider generalized weighted empirical measure $\sn \sum_{j=1}^N \wji(t,x)\delta_{X_i(t,x)}$ and prove its convergence to  $\izj \izj f(t,x,\xi,d\eta)d\xi$, where $f$ is a weak solution to the Vlasov-type equation (\ref{meanmean}). 
 The four variable object $f(t,x,\xi,\eta)$ can be interpreted as a state-dependent dynamic extended graphon. Specifically, it represents the joint density of the state distribution of node $\eta$ weighted by its instantaneous interaction strength toward node $\xi$ at time $t$. If we integrate $f$ w.r.t. label variables $\xi$ and $\eta$, the resulting function provides a global description of the system by removing individual network identities. While a classical mean-field limit yields a standard probability density that merely tracks the physical fraction of nodes in given state, $\izj \izj f(t,x,\xi,d\eta)d\xi$ captures a network-weighted state profile. Instead of just measuring population sizes, it tracks the global distribution of network activity. To ground this in an application, consider a model of opinion dynamics where $X_i(t)$ represents a vector of opinions of $i$-th agent at time $t$  and the dynamic weights $w_{ji}(t)$ represent the influence which $i$-th agent has on $j$-th agent at time $t$. In a scenario where a vast majority of the population holds a baseline opinion but shares very weak mutual ties, while a tiny minority holds a distinct opinion but is locked in a hyper-connected cluster with dense, powerful connections, a classical mean-field model would overlook the minority's impact. Conversely, $\izj \izj f(t,x,\xi,d\eta)d\xi$ correctly identifies this small cluster as a major driver of the system's overall activity, because their small population size is multiplied by large structural weights. Hence, our result seems to be useful  for modeling phenomena driven by highly active subgroups. \\
Although the evolution equation for the dynamic weights (\ref{part})$_{(ii)}$ is highly specific, our work yields the first mean-field limit for adaptive weights under the fairly mild assumptions (\ref{M})–(\ref{aswl}) on the initial $w_{ij}^0$ inspired by \cite{JPS2025}. The potential extension to more comprehensive frameworks, as well as applications to real-life problems are exciting directions for future research.\\
The present work focuses on the development of mathematical tools. The main novelty is the extension of the theory from \cite{JPS2025} to cover the framework of vector valued measures. We establish well-posedness of the limiting problem (\ref{meanmean}) in weak$^*$
 Bochner space $\W((W^{1,1}\cap W^{1,\infty})(\R^d))$.  Furthermore, inspired by \cite{JPS2025}, we introduce new observables indexed by trees. In our setting the weights $w_{ij}(t)$ do not appear explicitly but are already incorporated in state-dependent dynamic extended graphon $f(t,x,\xi,\eta)$. Consequently, contrary to the previous paper \cite{nasza}, we cannot directly apply the refined theory developed in \cite{JPS2025} but we carefully adjust it to a new setting. Importantly, since our setup is not symmetric in $\xi$ and $\eta$, some parts of the reasoning deviates significantly from the approach in \cite{JPS2025}. Another key contribution of this work lies in the introduction of the generalized weighted empirical measure $\sn \sum_{j=1}^N \wji(t,x)\delta_{X_i(t,x)}$. Considering the auxiliary measures $\xb_i(t,x)_{\#}[\wjib(t,x)d\p]$, matching the structure of the empirical measure, facilitates the propagation of independence, which is a key ingredient in the proof and a major obstacle in establishing the mean-field limit under the adaptive weights regime
(c.f. \cite{A2026}, \cite{Zhou}). Regarding the Sobolev regularity in the $x$ variable, the framework from \cite{JPS2025} is adapted herein (similarly to \cite{nasza}), without focusing on relaxing the spatial regularity assumptions.\\
Finally, we point out that although  the problem studied in \cite{nasza} seems to be a special case of (\ref{part}) with $w_{ij}(t) = M_j(t)w_{ij}$, the result in \cite{nasza} differs from the one established here. Taking formally  $w_{ij}(t) = M_j(t)w^0_{ij}$, we note that in our new empirical measure the weights depend on the initial matrix of connections, which is not the case in \cite{nasza}. Consequently, the limiting $f$ encapsulates the initial interactions.  

The rest of the paper is organized as follows. 
Below we collect the assumptions that we impose on deterministic data and random initial conditions as well as our main results. The proof of Theorem \ref{maintheorem} is divided in two steps which are presented in the next two chapters. The propagation of independence result is established in Chapter 2. In Chapter 3 we introduce the vector-valued graphons and carefully adjust the theory developed in \cite{JPS2025} to cover the adaptive framework. We end the Chapter 3 with a proof of our main result (Theorem \ref{maintheorem}). In order to streamline the exposition, we move certain technical proofs to the Appendix in Chapter~4.   

 First of all, we assume that the family 
\eqq{ (X_{i}^0,w_{1i}^0,\dots,w_{Ni}^0)_{i=1}^N \hd \m{ is jointly independent. }}{ind}
 
We impose the following assumptions on the given deterministic kernels 
\eqq{K_1 \in W^{1,\infty}(\R^d;\R^d) \cap W^{1,1}(\R^d;\R^d) , \hd \hd K_2 \in W^{1,\infty}(\R^d;\R) \cap L^{1}(\R^d), \hd \hd  K_2 \geq 0.}{asK}
Furthermore, we assume that there exists $M > 0$ such that
\eqq{
\sup_{N \in \mathbb{N}}\max_{1\leq i \leq N}\esssup_{x\in \R^d}\sum_{j=1}^N\wij^0(x) \leq M, \hd \hd \sup_{N \in \mathbb{N}}\max_{1\leq j \leq N}\esssup_{x\in \R^d}\sum_{i=1}^N\wij^0(x) \leq M,
}{M}
\eqq{\frac{1}{N}\sum_{i=1}^N \E \max_{1\leq j \leq N} \wij^0 \rightarrow 0 \m{ as } N\rightarrow \infty \m{ and } \wij^0 \geq 0 \hd \p \hd a.e.  \m{ for every } i,j = 1,\dots, N.}{aswl}

Finally, for every $i=1,\dots, N$, let the density $f_{X_i^0}$ of $X^0_i$ exists, belongs to $(W^{1,1}\cap W^{1,\infty})(\mathbb{R}^d)$, and satisfies the following:
\eqq{
  \sup_{N\in \mathbb{N}}\max_{1\leq i \leq N} \sum_{j=1}^N\norm{g_{ij}}_{W^{1,1}(\R^d)\cap W^{1,\infty}(\R^d)} < \infty, \hd 
  \sup_{N\in \mathbb{N}}\max_{1\leq j \leq N} \sum_{i=1}^N\norm{g_{ij}}_{W^{1,1}(\R^d)\cap W^{1,\infty}(\R^d)} < \infty,
}{asf}
where   
\[
g_{ij}(x):= f_{X^0_i}(x)\E[w_{ji}^0|X_i^0=x].
\]
\begin{remark}
    The assumption (\ref{asf}) simplifies significantly if we assume that for every $i,j \in \{1,\dots, N\}$, $N\in \mathbb{N}$, the random variables $ X^0_i$ and $\wji^0$ are independent. Then,  we arrive at
    \[
    g_{ij}(x) =  f_{X^0_i}(x)  \mathbb{E} \wji^0
    \]
    and due to 
(\ref{M}), the assumption (\ref{asf}) reduces to 
    \[
    \sup_{N\in \mathbb{N}}\max_{1\leq i \leq N} \norm{f_{X^0_i}}_{W^{1,1}(\R^d)\cap W^{1,\infty}(\R^d)} < \infty,
    \]
    which coincides with the assumption for conservative problem imposed in \cite{JPS2025}, (see also \cite{nasza}).
\end{remark}

\begin{remark}
Let us provide a simple example of sparse random graph with matrix of connections $w_{ij}^0$ satisfying (\ref{ind}),(\ref{M})-(\ref{asf}). We fix a sequence  $k_N \rightarrow \infty$ as $N\rightarrow \infty$ and $\frac{k_N}{N} \rightarrow 0$ as $N\rightarrow \infty$ ( for example $k_N=\lfloor N^{1-\alpha}\rfloor, \alpha \in (0,1)$). 
Let $G_N=(V^N,E^N)$ be a sequence of graphs, where  $V^N=\{1,\dots,N\}$ - represent vertices and $E^N=(E^N_{ij})_{i,j=1}^N$ - represents the set of directed edges. Set $E^N_{ij} = 1$ if $j \in \{i+1,i+k_N\}~\pmod{N}$ and $E^N_{ij} = 0$ otherwise. Fix $q \in (0,1)$ and introduce iid random variables $A_{ij}^N \sim Bernoulli(q)$. Then we set
\[
w_{ij}^0(x) = \frac{A_{ij}^N(x)E^N_{ij}}{k_N} .
\]
 Let $X_i^0 : \R^d \to \R^d$ be independent random variables (independent also from all $A_{ij}$) with density $f_{X^0_i}$ satisfying uniform bound in $(W^{1,1}\cap W^{1,\infty})(\R^d)$. Then clearly independence assumption (\ref{ind}) is satisfied. Furthermore, we easily see that
 \[
 \sup_{x \in \R^d} \sum_{i=1}^N \frac{A_{ij}^N(x)E_{ij}^N }{k_N} \leq   \frac{1}{k_N} \sum_{i=1}^N E_{ij}^N =1, \hd \hd \sup_{x \in \R^d} \sum_{j=1}^N \frac{A_{ij}^N(x)E_{ij}^N }{k_N} \leq   \frac{1}{k_N} \sum_{j=1}^N E_{ij}^N =1
 \]
 and the assumption (\ref{M}) is satisfied. As to the assumption (\ref{aswl}) we note that 
\[
\frac{1}{N}\sum_{i=1}^N \E \max_{1\leq j\leq N} w_{ij}^0 \leq \frac{1}{k_N}\rightarrow 0 \m{ as } N\rightarrow \infty.
\]
Finally, (\ref{asf}) follows from independence $w_{ji}^0$ and $X^0_i$, the assumption on densities of $X^0_i$ and (\ref{M}).
\end{remark}

Before we formulate main result, we denote by $d_F$ the flat metric on the space of positive Radon measures defined on $\R^d$, i.e.
\[
d_F(\mu,\nu) = \sup \left\{\abs{\ird \phi(x)d(\mu(x)-\nu(x))}: \phi \in W^{1,\infty}(\R^d): \norm{\phi}_{L^{\infty}(\R^d)} \leq 1, \norm{\nabla\phi}_{L^{\infty}(\R^d)}\leq 1\right\}.
\]

In the whole paper the time derivative is always denoted by $\partial_t$ and all other differential operators refer to the space variable, thus we omit the subscript $x$ and denote simply $\nabla,\divv$ etc. Furthermore, while integrating with respect to measure, we use the short notation   $f(dx)$ which we understand as $df(x)$. This is particularly useful if we deal with multi variable objects.
Below we state the main result of this paper. The underlying space $\W((W^{1,1}\cap W^{1,\infty})(\R^d))$ is defined in Definition \ref{space}.

  \begin{theo} \label{maintheorem}
Assume (\ref{ind}) - (\ref{asf}) and 
\eqq{
\sup_{N\in \mathbb{N}} \sup_{1\leq i\leq N}\mathbb{E}[\vert X_i^0\vert^2]<\infty.
}{secmo}
Finally, let $(X_i(t), \wij(t))_{i=1}^N$ be a solution to (\ref{part}) given by Lemma \ref{exi}. 
Then, there exists nonnegative $f\in L^\infty((0, t_*);\W((W^{1,1}\cap W^{1,\infty})(\R^d)))$,  such that $f$ is a weak solution to
\eqq{
\partial_t f(t,x,\xi,\eta) + \divv ( f(t,x,\xi,\eta) V_1[f](t,x,\eta)) =  f(t,x,\xi,\eta) (A -V_2[f](t,x,\eta)), 
}{meanmean}
where for $k=1,2$
\[
V_k[f](t,x,\eta) =  \ird  K_{k}(x-y)\izj f(t,y,\eta,d\zeta)dy.
\]
 and, up to the extraction of a subsequence,
\eqq{
\esssup_{0\leq t\leq t_*}\,\mathbb{E}\,d_{F}\left(\int_0^1 \izj f(t,\cdot,\xi,d\eta)d\xi, \sn \sum_{j=1}^N \wji(t,x)\delta_{X_i(t,x)}(\cdot)\right)\to 0, \m{ as } N\rightarrow \infty.
}{finalfinal}
 \end{theo}

\section{Propagation of independence}

This chapter is devoted to the proof of Theorem \ref{mainprob} stated below. We follow the approach from \cite[Chapter 2]{nasza}  (see also \cite{JPS2025}) with necessary modifications, which allow us to handle time dependent matrix $w_{ij}(t)$.
Consider the following auxiliary problem
\begin{equation}\label{part2}
\begin{array}{lll}
(i) & \partial_t \xb_i(t) = \sum_{k=1}^N  \ird K_{1}(\xb_i(t)-y)\fik(t,dy) & t \in (0,t_*),\\
(ii) & \partial_t \wjib(t) = \wjib(t)[A-\sum_{k=1}^N \ird K_{2}(\xb_i(t)-y)\fik(t,dy)] & t \in  (0,t_*),\\
(iii) & \xb_i(0) = X_{i}^0, \hd \wijb(0) = \wij^0,
\end{array}
\end{equation}
for $i,j=1,\dots, N$, where 
\eqq{
\fji(t,x) = \xb_i(t,x)_{\#}[\wjib(t,x)d\mathbb{P}(x)].
}{defft}
The notation  $\xb_i(t,x)_{\#}$ refers to the push forward operator, i.e. for any measurable map $T:\mathcal{A}_1\rightarrow \mathcal{A}_2$ and a measure $\mu$ on $\mathcal{A}_1$ we define $(T_{\#}\mu)(C):=\mu(T^{-1}(C))$ for any measurable $C \subset \mathcal{A}_2$.

\begin{lemma}\label{exi}
Assuming (\ref{asK})-(\ref{M}), the nonnegativity part from the assumption (\ref{aswl}) and 
\[
 \max_{1\leq i\leq N}\mathbb{E}|X^0_i| < \infty,
\]
for every $N \in \mathbb{N}$ the problems (\ref{part}) and (\ref{part2}) admit the unique solutions $(X_i(t), \wij(t))_{i=1}^N$ and \\$(\xb_i(t),\wijb(t))_{i=1}^N \in C^{1}((0,t_*);L^{1}(\R^{d},\p))$, respectively. 
Moreover, for each $t\in (0,t_{*})$ the random variables $\wij(t), \wijb(t)$ are nonnegative $\p$ a.e. for every $i,j \in \{1,\dots,N\}$ for every $N \geq 1$ and
\[
\sup_{N\in \mathbb{N}}\max_{1\leq j\leq N}\sup_{t\in (0,t_*)} \esssup_{x \in \R^d} \sum_{i=1}^N \wij(t,x) \leq Me^{At_*},\hd \sup_{N\in \mathbb{N}}\max_{1\leq i\leq N}\sup_{t\in (0,t_*)} \esssup_{x \in \R^d} \sum_{j=1}^N \wij(t) \leq Me^{At_*}, 
\]
\eqq{
\sup_{N\in \mathbb{N}}\max_{1\leq j\leq N}\sup_{t\in (0,t_*)}\esssup_{x \in \R^d}  \sum_{i=1}^N \wijb(t) \leq Me^{At_*},\hd \sup_{N\in \mathbb{N}}\max_{1\leq i\leq N}\sup_{t\in (0,t_*)} \esssup_{x \in \R^d} \sum_{j=1}^N \wijb(t) \leq Me^{At_*}. 
}{Mbound}
Furthermore, if the assumption (\ref{ind}) is satisfied, then also for every $t \in (0,t_{*})$ the family \\
$(\xb_{i}(t),\bar{w}_{1i}(t),\dots,\bar{w}_{Ni}(t))_{i=1}^N$  is jointly independent. 
\end{lemma}

We postpone the proof to the Appendix.

\begin{theo}\label{mainprob}
Let us assume that the assumptions (\ref{ind}) -(\ref{aswl}) are satisfied and there exists $C_1 > 0$ such 
\eqq{
\sup_{N \in \mathbb{N}}\max_{1\leq i \leq N}\E|X_i^0|^2 \leq C_1.
}{ex}
Let $(X_i(t), \wij(t))_{i=1}^N$ and $(\xb_i(t),\wijb(t))_{i=1}^N$ be a unique solution to (\ref{part}) and (\ref{part2}) in $C^{1}((0,t_*);L^{1}(\R^{d},\p))$, respectively.  
Then, there exist  positive constant $C = C(C_1,t_*,M,A,\norm{K_1}_{L^{\infty} (\R^d)})$, such that  for any $N \in \mathbb{N}$ we have
\eqq{
\esssup_{t\in(0,t_{*})}\E d_{F}\left(\sn \sum_{j=1}^N\wjib(t,x)\delta_{\xb_i(t,x)}(\cdot),\frac{1}{N}\sum_{i=1}^N\sum_{j=1}^N \fji (t,\cdot)\right) \leq C N^{-\frac{1}{2+3d/2}}.
}{zGC}
Furthermore, there exists a positive constant $C$ dependent only on $M,A,t_*,\norm{K_{1}}_{W^{1,\infty}(\R^d)},\norm{K_{2}}_{W^{1,\infty}(\R^d)}$ such that
\eqq{
\esssup_{t\in(0,t_{*})}\E d_{F}\left(\sn \sum_{j=1}^N \wjib(t,x)\delta_{\xb_i(t,x)}(\cdot),\sn \sum_{j=1}^N \wji(t,x)\delta_{X_i(t,x)}(\cdot)\right)
\leq C \sqrt{\frac{1}{N}\sum_{i=1}^N\E \max_{1\leq j \leq N}w^0_{ij}}.
}{zEX}
Finally,
\eqq{
\esssup_{t\in(0,t_{*})}\E d_{F}\left(\sn \sum_{j=1}^N \wji(t,x)\delta_{X_i(t,x)}(\cdot),\frac{1}{N}\sum_{i=1}^N\sum_{j=1}^N \fji (t,\cdot)\right) \rightarrow 0 \m{ as } N\rightarrow \infty
}{finalp}
and for $i,j \in \{1,\dots,N\}$ the measure $\fji$ defined in (\ref{defft}) is a measure solution to
\eqq{
\partial_t \fji(t,x) + \divv\left(\fji(t,x)\ird K_{1}(x-y)\sum_{k=1}^N\fik(t,dy)\right) = \fji(t,x)\left[A - \ird K_{2}(x-y)\sum_{k=1}^N\fik(t,dy)\right].
}{semic}
\end{theo}

Before we prove Theorem \ref{mainprob} we establish a couple of auxiliary results.

\begin{prop}
Under the assumptions (\ref{asK}), (\ref{M}) and the second part of the assumption (\ref{aswl}), the solutions to (\ref{part2}) satisfy
\eqq{
\wji^0(x) \exp\left(-t_*\norm{K_2}_{L^{\infty}(\R^d)}  e^{At_*}M\right) \leq \wjib(t,x) \leq \wji^0(x) e^{At_*} \hd \p \hd a.e,   \m{ for every } i,j \in \{1,\dots, N\}.
}{wt}
\end{prop}
\begin{proof}
The estimate from above is already established in Lemma \ref{exi}. Indeed, we note that for each $i,j=1,\dots, N$ the random variables $\wjib(t)$ satisfy
\eqq{
\wjib(t) = \wji^0 \exp\left(\int_0^t \left[A-\ird K_{2}(\xb_i(\tau)-y)\sum_{k=1}^N\fik(\tau,dy) \right] \right).
}{intM}
Hence, for each $i,j \in \{1,\dots N\}$ the family of random variables $\wjib(t)$ is nonnegative $\p$ a.e. for every $t \in (0,t_{*})$.
Thus, since $K_2$ is also nonnegative 
\eqq{
\wjib(t) \leq \wjib^0 e^{At} \hd \p \hd a.e.
}{mab}
In order to obtain the estimate from below we note that  using (\ref{asK}), (\ref{aswl}),  (\ref{mab}) and (\ref{M}) we may estimate as follows
\[
\ird K_{2}(\xb_i(\tau)-y)\sum_{k=1}^N\fik(\tau,dy)\leq \norm{K_2}_{L^{\infty}(\R^d)} \E \sum_{k=1}^N  \bar{w}_{ik}(\tau) \leq \norm{K_2}_{L^{\infty}(\R^d)} e^{At_*}M.
\]

Inserting it in (\ref{intM}) we arrive at
\[
\wjib(t) \geq \wji^0 \exp(At)\exp\left(-t\norm{K_2}_{L^{\infty}(\R^d)}  e^{At_*}M\right) \geq \wji^0 \exp\left(-t_*\norm{K_2}_{L^{\infty}(\R^d)}  e^{At_*}M\right) \hd \p \hd a.e.
\]
\end{proof}

In the proof of next lemma we follow the approach from the proof of \cite[Proposition 3.2.]{JPS2025} (see also \cite[Lemma 4]{nasza}) to show that $(\xb_i,(\wijb)_{j=1}^N)_{i=1}^N$ well approximate $(X_i,(\wij)_{j=1}^N)_{i=1}^N$.

\begin{lemma}\label{exest}
Assuming  (\ref{ind}) -(\ref{M}), (\ref{ex}) and nonnegativity of initial connectivities $\wij^0$, there exists a positive constant $C$ dependent only on $M,A,t_*,\norm{K_{1}}_{W^{1,\infty}(\R^d)},\norm{K_{2}}_{W^{1,\infty}(\R^d)}$, such that 
\[
\frac{1}{N}\sum_{ i =1 }^N \E|X_i(t) - \xb_i(t)|  + \frac{1}{N} \sum_{i,j=1}^N \E\abs{\wji(t) - \wjib(t)} \leq C \sqrt{\frac{1}{N}\sum_{i=1}^N\E \max_{1\leq j \leq N}w^0_{ij}}.
\]
\end{lemma}
\begin{proof}
Let us subtract the equations (\ref{part2})$_{(i)}$ from (\ref{part})$_{(i)}$. Then we obtain for every $i \in \{1,\dots,N\}$
\eqq{
\partial_t(X_i(t) - \xb_i(t)) = \sum_{k=1}^N \left(w_{ik}(t)K_{1}(X_i(t)-X_k(t)) -  \ird K_{1}(\xb_i(t)-y)\fik(t,dy)\right).}{nax}

We integrate the equation with respect to $\p$ and make use of the triangle inequality
\begin{align*}
  \partial_t\E (X_i(t) - \xb_i(t))   &= \sum_{k=1}^N \E \left( w_{ik}(t)[K_{1}(X_i(t)-X_k(t))-K_1(\xb_i(t)-\xb_k(t))]\right)  \\
  &+ \sum_{k=1}^N \E \left([w_{ik}(t)-\bar{w}_{ik}(t)]K_1(\xb_i(t)-\xb_k(t)) \right)\\
  &+ \sum_{k=1}^N \E \left( \bar{w}_{ik}(t)K_1(\xb_i(t)-\xb_k(t))-  \ird K_{1}(\xb_i(t)-y)\fik(t,dy)\right).
\end{align*}
Hence,
\begin{align*}
&\partial_t \frac{1}{N}\sum_{i=1}^N\E \abs{X_i(t) - \xb_i(t)} \leq \frac{1}{N}\sum_{i=1}^N
\sum_{k=1}^N \E w_{ik}(t)\abs{K_{1}(X_i(t)-X_k(t))-K_1(\xb_i(t)-\xb_k(t))}  \\ 
& +\frac{1}{N}\sum_{i=1}^N \sum_{k=1}^N \E \abs{w_{ik}(t)-\bar{w}_{ik}(t)}\abs{K_1(\xb_i(t)-\xb_k(t))} \\
& + \frac{1}{N}\sum_{i=1}^N\E\abs{\sum_{k=1}^N  \left( \bar{w}_{ik}(t)K_1(\xb_i(t)-\xb_k(t))-  \ird K_{1}(\xb_i(t)-y)\fik(t,dy)\right)} =: I_1 + I_2+I_3.
\end{align*}
We estimate the first expression using (\ref{asK}) and (\ref{Mbound}) to the result
\begin{align}\label{I1}
&I_1 \leq \norm{K_{1}}_{W^{1,\infty}(\R^d)}\frac{1}{N}\left(\sum_{i=1}^N \E \sum_{k=1}^N w_{ik}(t)|X_i(t)-\xb_i(t)| + \sum_{k=1}^N \E \sum_{i=1}^N w_{ik}(t)|X_k(t)-\xb_k(t)|\right) \\ \nonumber  
& \leq \norm{K_{1}}_{W^{1,\infty}(\R^d)}\frac{1}{N}\left(\max_{1\leq i \leq N}\sup_{x\in \R^d}\sum_{k=1}^N w_{ik}(t) \sum_{i=1}^N \E |X_i(t)-\xb_i(t)| + \max_{1\leq k \leq N}\sup_{x\in \R^d}\sum_{i=1}^N w_{ik}(t) \sum_{k=1}^N \E |X_k(t)-\xb_k(t)| \right) \\ \nonumber  
& \leq 2Me^{At_*}\norm{K_{1}}_{W^{1,\infty}(\R^d)}  \frac{1}{N}\sum_{k=1}^N \E |X_k(t)-\xb_k(t)|.
\end{align}
Moreover, again by (\ref{asK}) 
\eqq{
I_2 \leq \norm{K_1}_{L^{\infty}(\R^d)}  \frac{1}{N}\E \sum_{i=1}^N\sum_{k=1}^N|w_{ik}(t) - \bar{w}_{ik}(t)|.
}{I2}

In order to estimate $I_3$, we observe that by Lemma \ref{exi} the random variable $(\xb_k(t),\bar{w}_{ik}(t))$ is independent on $\xb_i(t)$ for each $t \in [0,t_*)$ if $i\neq k$. Hence, for $i\neq k$
\[
\E\left[\bar{w}_{ik}(t)K_1(\xb_i(t)-\xb_k(t))| \xb_i(t)\right] = \ird  K_{1}(\xb_i(t)-y)\fik(t,dy).
\]
Let us denote
\[
G_{ik}(t) = \bar{w}_{ik}(t)K_1(\xb_i(t)-\xb_k(t))-  \ird K_{1}(\xb_i(t)-y)\fik(t,dy)
\]
and estimate the second moment
\[
  \E \abs{\sum_{k=1}^N G_{ik}(t)}^2 \leq 2 \E \left(\sum_{k\neq i}^N \sum_{l\neq i}^N G_{ik}(t)\cdot G_{il}(t)\right)
  + 2 \E \abs{ G_{ii}(t)}^2.
\]
We will show that for $k\neq l, k \neq i, l\neq i$
\[
\E [G_{ik}(t)\cdot G_{il}(t)] = 0 \hd \m{ for every }\hd t \in (0,t_{*}).
\]
To that end, note that for $k\neq l, k\neq i, l\neq i$ random variables $G_{il}(t)$ and $G_{ik}(t)$ are conditionally independent under the condition $\xb_i(t)$. Furthermore,  
since for any random variables $X,Y$ we have $\E(\E(X|Y)) = \E X$  and if $X$ and $Y$ are conditionally independent under condition $Z$ then $\E(XY|Z) = \E(X|Z)\E(Y|Z)$, we may write
\[
\E [G_{ik}(t)\cdot G_{il}(t)] = \E \left[\E\left[G_{ik}(t)\cdot G_{il}(t)|\xb_{i}(t)\right]\right] = \E \left[\E\left[G_{ik}(t)|\xb_{i}(t)\right]\E\left[G_{il}(t)|\xb_{i}(t)\right]\right] = 0,
\]
where the last equality follows from $\E\left[G_{ik}(t)|\xb_{i}(t)\right] = 0$, since for every two random variables $X,Y$ there holds $\E(X-\E(X|Y)|Y) = \E(X|Y) - \E(X|Y) = 0.$ All in all, we obtain 
\begin{align*}
&\frac{1}{N} \E \sum_{i=1}^N \abs{\sum_{k=1}^N G_{ik}}^2\leq  \frac{2}{N} \E \sum_{i,k=1}^N G^2_{ik} = \frac{2}{N}\sum_{i,k=1}^N \E \left( \bar{w}_{ik}(t)K_1(\xb_i(t)-\xb_k(t))-  \ird K_{1}(\xb_i(t)-y)\fik(t,dy)\right)^2\\
 & \leq \frac{8}{N} \norm{K_{1}}_{L^{
\infty}(\R^d)}^2 \E \sum_{k=1}^N \sum_{i=1}^N \bar{w}_{ik}^2(t) \leq  \frac{8}{N} \norm{K_{1}}_{L^{
\infty}(\R^d)}^2  \E \sum_{i=1}^N \left(\max_{1\leq k\leq N}\bar{w}_{ik}(t)\sum_{k=1}^N \bar{w}_{ik}(t)\right)\\
& \leq \frac{8}{N} \norm{K_{1}}_{L^{
\infty}(\R^d)}^2 \E \max_{1\leq i\leq N}\sum_{k=1}^N \bar{w}_{ik}(t) \sum_{i=1}^N\max_{1\leq k\leq N}\bar{w}_{ik}(t) \\
&\leq 
\frac{8}{N} \norm{K_{1}}_{L^{
\infty}(\R^d)}^2 \max_{1\leq i\leq N}\esssup_{x\in \R^d}\sum_{k=1}^N \bar{w}_{ik}(t)\cdot\sum_{i=1}^N \E \max_{1\leq k\leq N}\bar{w}_{ik}(t)\leq 8 \norm{K_{1}}_{L^{
\infty}(\R^d)}^2 M e^{2A t_{*}} \frac{1}{N}\sum_{i=1}^N \E \max_{1\leq k\leq N}\bar{w}_{ik}^0, 
\end{align*}
where in the last estimate we applied (\ref{Mbound}) and (\ref{wt}).
Hence, by Jensen inequality 
\eqq{
I_3 \leq \sqrt{8M} e^{At_*}\norm{K_{1}}_{L^{
\infty}(\R^d)} \sqrt{\frac{1}{N}\sum_{i=1}^N \E \max_{1\leq k\leq N}\bar{w}_{ik}^0}.
}{I3}

Combining (\ref{I1}), (\ref{I2}) and (\ref{I3}) we obtain for each $1\leq i \leq N$
\eqq{
\partial_t \frac{1}{N}\sum_{i=1}^N\E \abs{X_i(t) - \xb_i(t)} \leq \frac{C}{N} \left(\sum_{i=1}^N\E \abs{X_i(t) - \xb_i(t)} +  \E \sum_{i,k=1}^N\abs {w_{ik}(t) - \bar{w}_{ik}(t)}\right) + C \sqrt{\frac{1}{N}\sum_{i=1}^N \E \max_{1\leq k\leq N}\bar{w}_{ik}^0},
}{xdif}
where $C$ is a positive constant dependent only on $M,A,t_*,\norm{K_{1}}_{W^{1,\infty}(\R^d)}$. 
Similarly we subtract  (\ref{part2})$_{(ii)}$ from (\ref{part})$_{(ii)}$ and integrate with respect to $\p$. Then, we arrive at
\begin{align}\label{mdifp}
   &\partial_t \E |\wji(t)-\wjib(t)| \leq \\ \nonumber
   & A\E |\wji(t)-\wjib(t)| + \E\wjib(t)\abs{\sum_{k=1}^N \left(w_{ik}(t)K_{2}(X_i(t)-X_k(t)) -  \ird K_{2}(\xb_i(t)-y)\fik(t,dy)\right)}.
\end{align}
Summing over $i$ and $j$ gives
\[
\partial_t\frac{1}{N}  \sum_{i,j=1}^N \E|\wji(t)-\wjib(t)| \leq \frac{A}{N} \sum_{i,j=1}^N \E|\wji(t)-\wjib(t)| 
\]
\[
+\max_{1\leq i \leq N} \esssup_{x\in \R^d}\sum_{j=1}^N \wjib(t) \frac{1}{N}\sum_{i=1}^N\E \abs{\sum_{k=1}^N \left(w_{ik}(t)K_{2}(X_i(t)-X_k(t)) -  \ird K_{2}(\xb_i(t)-y)\fik(t,dy)\right)}.
\]
Note that the last term under expectation differs from the right hand side of $(\ref{nax})$ only by the change of kernel $K_1$ into $K_2$. Hence, repeating the arguments leading to estimates (\ref{I1}),(\ref{I2}), (\ref{I3}) we obtain
\begin{align}\label{mdifpp}
    &\frac{1}{N}\sum_{i=1}^N\E\abs{\sum_{k=1}^N \left(w_{ik}(t)K_{2}(X_i(t)-X_k(t)) -  \ird K_{2}(\xb_i(t)-y)\fik(t,dy)\right)} \\ \nonumber
   & \leq \frac{C}{N} \left(\sum_{i=1}^N\E \abs{X_i(t) - \xb_i(t)} +   \sum_{i,k=1}^N \E\abs {w_{ik}(t) - \bar{w}_{ik}(t)}\right) + C \sqrt{\frac{1}{N}\sum_{i=1}^N \E \max_{1\leq k\leq N}\bar{w}_{ik}^0}.
\end{align}
Here $C$ depends only on $M,A,t_*,\norm{K_{2}}_{W^{1,\infty}(\R^d)}$ and may change from line to line.
Using (\ref{Mbound}) we get
\eqq{
\partial_t \frac{1}{N} \sum_{i,j=1}^N \E|\wji(t)-\wjib(t)|  \leq C\left( \frac{1}{N}\sum_{i,j=1}^N \E |w_{ji}(t)-\bar{w}_{ji}(t)| + \frac{1}{N} \sum_{i=1}^N\E \abs{X_i(t) - \xb_i(t)} + \sqrt{\frac{1}{N}\sum_{i=1}^N \E \max_{1\leq k\leq N}\bar{w}_{ik}^0}\right),
}{mdif}
where $C$  depends only on $M,A,t_*,\norm{K_{2}}_{W^{1,\infty}(\R^d)}$. We integrate  (\ref{mdif}) and  (\ref{xdif}) with respect to time and add the resulting estimates. Then we arrive at
\begin{align*}
 &\frac{1}{N}\sum_{i=1}^N \E |X_i(t)-\xb_i(t)|+\frac{1}{N}\sum_{i,j=1}^N \E |\wji(t)-\wjib(t)| \leq Ct\sqrt{\frac{1}{N}\sum_{i=1}^N \E \max_{1\leq k\leq N}\bar{w}_{ik}^0}\\
&+\frac{C}{N}\int_0^t\ \left(\sum_{i=1}^N \E |X_i(\tau)-\xb_i(\tau)|+ \sum_{i,j=1}^N \E |\wji(\tau)-\wjib(\tau)|\right) d\tau,
\end{align*}
and we finish the proof applying Gr\"onwall inequality. 

\end{proof}

Now, we are ready to prove Theorem \ref{mainprob}.

\begin{proof}[Proof of Theorem \ref{mainprob}]
In order to establish  (\ref{zGC}) we apply Lemma \ref{GC} with $X_i=\xb_i$ and $M_i(t)=\sum_{j=1}^N\wjib(t)$. Clearly, the assumption (\ref{ind}) together with Lemma \ref{exi} implies that $(\xb_i(t),\sum_{j=1}^N\wjib(t))$
and \\ $(\xb_k(t),\sum_{j=1}^N \bar{w}_{jk}(t))$ are independent for $i\neq k$. Secondly, note that integrating (\ref{part2})$_{(i)}$ w.r.t. time and using  (\ref{asK}), and (\ref{Mbound}) we obtain
\eqq{
\E \abs{\xb_i(t)}^2 \leq \E\left(\abs{X^0_i} + t\norm{K_1}_{L^{\infty}(\R^d)}M e^{At_*}\right)^2 \leq 2 \E\abs{X^0_i}^2 + 2\left(t_*\norm{K_1}_{L^{\infty}(\R^d)}M e^{At_*}\right)^2.
}{secmobound}
Hence, due to (\ref{ex}), there exists $C = C(C_1,t_*,M,A,\norm{K_1}_{L^{\infty}(\R^d)})$ such that $\E \abs{\xb_i(t)}^2 \leq C$.

Finally, taking into account estimate (\ref{Mbound}) we obtain that all the assumptions of Lemma \ref{GC} are satisfied and since $\sum_{j=1}^Nf^j_i(t)=\xb_i(t)_{\#}[\sum_{j=1}^N\wjib(t)d\p]$, estimate (\ref{zGC}) follows. 

To show (\ref{zEX}) we perform the following estimate
\begin{align*}
&\E \sup_{\norm{\Phi}_{L^\infty}\leq 1, \norm{\nabla \Phi}_{L^\infty}\leq 1 }\abs{\sn \left[\sum_{j=1}^N \wji(t)\Phi(X_i(t)) - \sum_{j=1}^N \wjib(t)\Phi(\xb_i(t))\right]}\\
&\leq \sn \left[\E\abs{\sum_{j=1}^N (\wji(t)- \wjib(t))} + \E\abs{\sum_{j=1}^N \wjib(t)}\abs{\xb_i(t)-X_i(t)}\right]
\\
&\leq (1+Me^{At_{*}})\left(\frac{1}{N} \E\sum_{i,j=1}^N \abs{\wji(t)-\wjib(t)} +\frac{1}{N}\sum_{i=1}^N\E\abs{\xb_i(t)-X_i(t)}\right) 
\\
&\leq (1+Me^{At_{*}}) C\sqrt{\frac{1}{N}\sum_{i=1}^N\E \max_{1\leq j \leq N}w^0_{ij}}.
\end{align*}
Note that in the last estimate we applied (\ref{Mbound}) together with  Lemma~\ref{exest} and $C$ is the constant from Lemma~\ref{exest}. Thus, (\ref{zEX}) is proven. The convergence (\ref{finalp}) is a direct consequence of (\ref{zGC}), (\ref{zEX}) and the assumption (\ref{aswl}). It remains to show that $\fji$ is a measure solution to (\ref{semic}). To this end, note that by the definition of $\fji$  for any $\vf \in C_{c}^1(\R^d)$ and every $i,j \in \{1,\dots,N\}$ we have
\begin{align*}
 &\partial_t \ird \vf(x)\fji(t,dx) = \partial_t \ird \vf(\xb_i(t,x))\wjib(t,x)d\p(x)  \\
 &=\ird \nabla \vf(\xb_i(t,x)) \partial_t \xb_i(t,x) \wjib(t,x)d\p(x) + \ird  \vf(\xb_i(t,x)) \partial_t \wjib(t,x)d\p(x)\\
 &=\ird \nabla \vf(\xb_i(t,x)) \wjib(t,x)\sum_{k=1}^N  \ird K_{1}(\xb_i(t,x)-y)\fik(t,dy) d\p(x)\\
 &+  \ird  \vf(\xb_i(t,x)) \wjib(t,x)\left[A-\sum_{k=1}^N \ird K_{2}(\xb_i(t,x)-y)\fik(t,dy)\right]d\p(x)\\
 & = \ird \nabla \vf(x) \sum_{k=1}^N  \ird K_{1}(x-y)\fik(t,dy) \fji(t,dx)\\
 &+ \ird  \vf(x) \left[A-\sum_{k=1}^N  \ird K_{2}(x-y)\fik(t,dy)\right]\fji(t,dx).
\end{align*}
Hence, indeed $\fji$ is a measure solution to (\ref{semic}), which finishes the proof of the theorem.
\end{proof}

\section{Vector-valued dynamic extended graphons and  the mean-field limit}

In view of Theorem \ref{mainprob}, in order to prove Theorem \ref{maintheorem} we ought to show the convergence of $\frac{1}{N}\sum_{i=1}^N\sum_{j=1}^N\fji(t)$ to $\izj\izj f(t,\cdot,\xi,d\eta)d\xi$, where $f$ is a solution to (\ref{meanmean}). 
We begin with the crucial definition of the vector-valued extended graphon space $\W(X)$. Then, the second section of this chapter contains the solvability result for (\ref{meanmean}) and the verification that
\[
f^N(t,x,\xi,\eta):= N\sum_{i=1}^N \sum_{j=1}^N \fji(t,x)\mathbb{I}_{[\frac{j-1}{N},\frac{j}{N})}(\xi)\mathbb{I}_{[\frac{i-1}{N},\frac{i}{N})}(\eta)
\]
is a weak solution to (\ref{meanmean}) with $\izj \izj f^N(t,x,\xi,d\eta)d\xi =\frac{1}{N}\sum_{i=1}^N\sum_{j=1}^N\fji(t)$. 
In the third section, inspired by \cite{JPS2025},  we introduce new tree-based observables and establish their crucial properties. In subsequent section we utilize the newly defined observables to establish the stability result for solutions to (\ref{meanmean}). The final section contains the proof of the main result.
\subsection{Introduction to function spaces }

Because the limiting object depends on two labels and the physical state, it is natural to encode its
label dependence by vector-valued measures. For a comprehensive treatment of vector-valued measures, we refer to the classic monograph \cite{DU1977}.

\begin{defi}\label{measurex}
Let $X$ be a Banach space. By $\mathcal{M}([0,1];X)$ we denote the space of regular Borel $X$-valued measures of bounded variations, i.e. $\mu \in \mathcal{M}([0,1];X)$ if
\begin{itemize}
\item $\mu:\mathcal{B}([0,1])\rightarrow X$,
\item for every sequence of pairwise disjoint $E_n \in \mathcal{B}([0,1])$ there holds $\mu\left(\cup_{n=1}^{\infty}E_n\right) = \sum_{n=1}^{\infty}\mu(E_n)$ in the norm topology of $X$,
\item its total variation $\abs{\mu}([0,1])$ is finite, where the variation of $\mu$ on $E \in \mathcal{B}([0,1])$ is defined as follows $\abs{\mu}(E):=\sup\sum_{i=1}^n\norm{\mu(E_i)}_{X}$ and the supremum is taken over all finite partitions $\{E_i\}$ of $E$ into disjoint Borel sets.
\item For every $E \in \mathcal{B}([0,1])$ and $\ve > 0$ there exists a compact set $K\subset [0,1]$ and open set $O\subset [0,1]$ such that $K\subset E\subset O$ and $\abs{\mu}(O\setminus K) < \ve$.
\end{itemize}
\end{defi}

\begin{remark}
By the extension of classical Riesz representation lemma (Singer representation lemma)\cite{Singer} the space $\mathcal{M}([0,1];X^*)$ may be identified with the space of linear functionals of $C([0,1];X)$ for any Banach space $X$. Furthermore, the identification holds by the following relation: for any $\vf \in (C([0,1];X))^{*}$ there exists $\mu \in \mathcal{M}([0,1];X^*)$ such that for any $g \in C([0,1];X)$ 
\eqq{
\vf(g) = \int_0^1\langle g(t),\mu_t\rangle \hd \m{ and } \hd \norm{\vf}=\abs{\mu}([0,1]).
}{intmu}
On the other hand, for any $\mu \in \mathcal{M}([0,1];X^{*})$ the integral above defines the linear functional on $C([0,1];X)$ with its norm equal to the total variation norm of $\mu$.
The integral in (\ref{intmu}) is defined as follows. For any simple function $g=\sum_{i=1}^n x_i \chi_{E_i}(t)$, $x_i \in X$, $E_i \in \mathcal{B}([0,1])$ we define
\[
\int_0^1\langle g(t),\mu_t\rangle = \sum_{i=1}^n \langle x_i,\mu(E_i)\rangle.
\]
Since  any $g \in C([0,1];X)$ is a uniform limit of simple functions $g^n$ and $\mu$ is of bounded variation we may define
\[
\int_0^1\langle g(t),\mu_t\rangle = \lim_{n\rightarrow \infty} \int_0^1\langle g^n(t),\mu_t\rangle
\]
where the limit does not depend on the approximating sequence.
\end{remark}

Below we extend the definition of extended graphon $\W:=L^{\infty}_\xi\mathcal{M}_\eta \cap L^{\infty}_\eta\mathcal{M}_\xi$ introduced in \cite[Definition~4.5 and Definition~4.6]{JPS2025} into vector-valued framework.

\begin{defi}\label{space}
Let $X$ be a Banach space. We denote by  $L^\infty_\xi(\mathcal{M}_\eta(X^*))$ the weak$^*$ Bochner space which is dual  to $L^1_\xi((0,1); C_\eta([0,1];X))$, i.e., $f\in L^\infty_\xi(\mathcal{M}_\eta(X^*))$  if the map  $\xi\in [0,\ 1]\mapsto f(\xi,\cdot,\cdot)\in \mathcal{M}([0,\ 1];X^*)$ is  weak$^*$ measurable (for every $\phi \in C([0,1];X)$ the map: $ 
\xi\mapsto \izj \langle\phi(\eta,\cdot)f(\xi,d\eta,\cdot)\rangle$ is measurable), essentially bounded and
\[
\Vert f\Vert_{L^\infty_\xi(\mathcal{M}_\eta(X^{*}))}:=\esssup_{\xi\in (0, 1)}\sup_{\Vert \phi\Vert_{C([0, 1];X)}\leq 1}\abs{\izj \langle\phi(\eta,\cdot)f(\xi,d\eta,\cdot)\rangle} < \infty.
\]
We define analogously $L^\infty_\eta(\mathcal{M}_\xi(X^*))$, introduce 
\begin{align*}
   \W(X^*): &= L^\infty_\xi(\mathcal{M}_\eta(X^*))\cap L^\infty_\eta(\mathcal{M}_\xi(X^*))\\
   &=(L^1_\xi((0,1); C_\eta([0,1];X)) + L^1_\eta((0,1); C_\xi([0,1];X)))^{*} = ((L^1_\xi C_\eta + L^1_\eta C_\xi)(X))^{*}
\end{align*}
and set 
\[
\norm{f}_{\W(X^{*})} :=\max\{\Vert f\Vert_{L^\infty_\xi(\mathcal{M}_\eta(X^{*}))}, \Vert f\Vert_{L^\infty_\eta(\mathcal{M}_\xi(X^{*}))}\}.
\]
By $\W$ we denote the scalar-valued space $L^{\infty}_\xi\mathcal{M}_\eta  \cap L^{\infty}_\eta\mathcal{M}_\xi$, (i.e. $X^*=\R$).

Note that the space $L^{1}((0,1)^2;X)$ is the ambient space for $(L^1_\xi C_\eta + L^1_\eta C_\xi)(X)$ and we set 
\begin{align*}
&\norm{u}_{(L^1_\xi C_\eta + L^1_\eta C_\xi)(X)}\\
&:= \inf\{\norm{u_1}_{L^1_\xi C_\eta(X)} + \norm{u_2}_{L^1_\eta C_\xi(X)}: u = u_1+u_2 \textrm{ in }L^{1}((0,1)^2;X), u_1 \in L^1_\xi C_\eta(X), u_2 \in L^1_\eta C_\xi(X) \}.
\end{align*}
\end{defi}

\begin{remark}(see \cite[page 683]{JPS2025})
For $\mu \in \mathcal{M}_\eta(X^*)$ denote by $\norm{\mu(\cdot)}_{X^*}(d\eta)$ the variation measure defined on Borel subsets of $[0,1]$ as in Definition \ref{measurex}. Then $\abs{\mu}([0,1]) = \izj \norm{\mu(\cdot)}_{X^*}(d\eta)$. In order to distinguish between the variables $\xi$ and $\eta$ for functions belonging to  $L^{\infty}_{\xi}(\mathcal{M}_\eta(X^*))$, we introduce the notation
\[
\norm{\mu}_{\mathcal{M}_{\eta}([0,1];X^*)}:= \abs{\mu}([0,1]).
\]
Let $X$ be separable Banach space. Then, also $C([0,1];X)$ is separable.  
Thus for $f \in L^{\infty}_{\xi}(\mathcal{M}_\eta(X^*))$the map: $
\xi \mapsto \norm{f(\xi,\cdot,\cdot)}_{\mathcal{M}_{\eta}([0,1];X^{*})}$ is measurable, essentially bounded and 
\[
\norm{f}_{L^{\infty}_{\xi}(\mathcal{M}_\eta(X^{*}))} = \esssup_{\xi\in (0,1)}\norm{f(\xi,\cdot,\cdot)}_{\mathcal{M}_\eta([0,1];X^{*})} = \esssup_{\xi\in (0,1)}\izj \norm{f(\xi,\cdot,\cdot)}_{X^{*}}(d\eta).
\]
\end{remark}

Finally, we incorporate also bounded dependence w.r.t. time.
Note that in our application we consider $f(t,x,\xi,\eta)$ where $(t,x)$ represent time and space (of states) and $(\xi,\eta)$ represent the extended graphon structure (labels). Hence, for notational purposes from now on, we will stick to the variable order $f(t,x,\xi,\eta)$, rather than quite unnatural  $f(t,\xi,\eta,x)$,  even though we consider $f \in L^{\infty}((0,t_{*});\W(X^{*}))$.

\begin{defi}\label{weakspacedef}
For a Banach space $X$ we define  
  $ L^{\infty}((0,t_{*});\W(X^{*}))$ as the weak$^*$ Bochner space which is dual to $L^{1}((0,t_{*});(L^1_\xi C_\eta + L^1_\eta C_\xi)(X))$, i.e. $f \in L^{\infty}((0,t_{*});\W(X^{*}))$ if for almost all $t\in (0,t_{*})$, $f(t) \in \W(X^{*})$, for any $\Phi \in ((L^1_\xi C_\eta + L^1_\eta C_\xi)(X))$ the mapping $t\mapsto \langle f(t), \Phi \rangle$ is measurable and
\[
\norm{f}_{L^{\infty}((0,t_{*});\W(X^{*}))}:=\esssup_{t\in (0,t_{*})}\sup_{\norm{\Phi}_{((L^1_\xi C_\eta + L^1_\eta C_\xi)(X))}\leq 1}\abs{\langle \Phi, f(t)  \rangle} < \infty.
\]
If $X$ is separable then  the norm above may be represented in the form 
\[
\norm{f}_{L^{\infty}((0,t_{*});\W(X^{*}))}=\esssup_{t \in (0,t_{*})} \hd\max\left\{\esssup_{\xi\in (0,1)}\izj \norm{f(t,\cdot,\xi,\cdot)}_{X^{*}}(d\eta), \esssup_{\eta\in (0,1)}\izj \norm{f(t,\cdot,\cdot,\eta)}_{X^{*}}(d\xi) \right\}.
\]

Analogously we define the spaces $L^{\infty}((0,t_{*});L^{\infty}_\xi(M_\eta(X^{*})))$ and $L^{\infty}((0,t_{*});L^{\infty}_\eta(M_\xi(X^{*})))$.
\end{defi}

One of the primary contributions in \cite{JPS2025} is extending the operator $\Phi\mapsto \int_0^1\Phi(\zeta)w(\cdot,d\zeta)$ for $w \in \W$ from continuous functions $\Phi$ to the one acting on merely essentially bounded functions. The result reads as follows.

\begin{lemma}\cite[Lemma 4.7]{JPS2025}\label{impoest}
The bilinear,  bounded operator
\[
\begin{array}{ccl}
\W \times C_\zeta([0,1]) & \longrightarrow & L^\infty_\xi((0,1)),\\
(w,\phi) & \longmapsto & \int_0^1\phi(\zeta)\,w(\cdot,d\zeta)
\end{array}
\]
 extends to a bounded operator from $\W \times L^\infty_\zeta((0,1))$ to $L^\infty_\xi((0,1))$. 
Furthermore, for any $w\in \W$ and $\phi\in L^\infty((0,1))$
\[
\left\|\int_0^1 \phi(\zeta)\,w(\cdot,d\zeta)\right\|_{L^1((0,1))}\leq \|w\|_{L^\infty_\zeta \mathcal{M}_\xi}\,\|\phi\|_{L^1((0,1))}, \hd 
\left\|\int_0^1 \phi(\zeta)\,w(\cdot,d\zeta)\right\|_{L^\infty((0,1))}\leq \|w\|_{L^\infty_\xi \mathcal{M}_\zeta}\,\|\phi\|_{L^\infty((0,1))}.
\]
In addition, for any uniformly bounded sequence $(\phi_n)_{n\in \mathbb{N}} \in L^{\infty}((0,1))$ such that $\phi_n\rightarrow \phi$ in $L^1((0,1))$ and a sequence   $w_n\overset{*}{\rightharpoonup} w$  in $\W$ the following convergence holds in weak-star topology on $L^{\infty}((0,1))$
  \[
\int_0^1\phi_n(\zeta)\,w_n(\cdot,d\zeta)\overset{*}{\rightharpoonup} \int_0^1 \phi(\zeta)w(\cdot,d\zeta).
\]
\end{lemma}

Let us now extend the result from \cite[Lemma 4.7]{JPS2025} for vector-valued measures.

\begin{lemma}\label{impoestv}
Let $X$ be a separable Banach space. The bilinear,  bounded operator
\[
\begin{array}{ccl}
\W(X^*) \times C_\eta([0,1];X) & \longrightarrow & L^\infty_\xi((0,1)),\\
(f,\phi) & \longmapsto & \int_0^1\langle\phi(x,\eta),f(x,\cdot,d\eta)\rangle_{X\times X^*}
\end{array}
\]
 extends to a bounded operator from $\W(X^*) \times L^{\infty}_\eta((0,1);X)$ to $L^\infty_\xi((0,1))$.
Furthermore, for any $f\in \W(X^*)$ and $\phi\in L^\infty((0,1);X)$
\begin{align}\label{vecexta}
    \left\|\int_0^1\langle\phi(x,\eta),f(x,\cdot,d\eta)\rangle_{X\times X^*}\right\|_{L^1((0,1))}&\leq \|f\|_{L^\infty_\eta (\mathcal{M}_\xi(X^*))}\,\|\phi\|_{L^1((0,1);X)},\\ \nonumber
    \left\|\int_0^1\langle\phi(x,\eta),f(x,\cdot,d\eta)\rangle_{X\times X^*}\right\|_{L^\infty((0,1))}&\leq \|f\|_{L^\infty_\xi (\mathcal{M}_\eta(X^*))}\,\|\phi\|_{L^\infty((0,1);X)}.
\end{align}

In addition, for any uniformly bounded sequence $(\phi_n)_{n\in \mathbb{N}} \in L^{\infty}((0,1);X)$ such that $\phi_n\rightarrow \phi$ in $L^1((0,1);X)$ and a sequence   $f_n\overset{*}{\rightharpoonup} f$  in $\W(X^*)$ the following convergence holds in weak-star topology on $L^{\infty}((0,1))$
  \[
\int_0^1\langle\phi_n(x,\eta)\,f_n(x,\cdot,d\eta)\rangle_{X\times X^*}\overset{*}{\rightharpoonup} \int_0^1\langle\phi(x,\eta)\,f(x,\cdot,d\eta)\rangle_{X\times X^*}.
\]
\end{lemma}  

The proof of lemma is based on the result in scalar case and we postpone it to the Appendix.

\begin{remark}
    In the paper we will mainly deal with two cases. In the first one $X^* = (L^{1}\cap L^{\infty})(\R^d)$ which is a dual space to separable $C_{0}(\R^d)+ L^{1}(\R^d)$. In the second case $X^*= (W^{1,1}\cap W^{1,\infty})(\R^d)$. To characterize the predual of  $(W^{1,1}\cap W^{1,\infty})(\R^d)$, we note that by \cite[Theorem 4.9]{Rudin}  if $Y$ is a Banach space, the dual of a quotient $Y / N$ is isometric to the annihilator $N^\perp \subset Y^*$. Setting $Y = (C_0(\mathbb{R}^d) + L^1(\mathbb{R}^d))^{d+1}$ and
    \[
    N = \left\{ (\phi_0, \phi_1, \dots, \phi_d) \in Y : \int_{\mathbb{R}^d} \left( \phi_0(x) \psi(x) + \sum_{j=1}^d \phi_j(x) \frac{\partial \psi}{\partial x_j}(x) \right) dx = 0 \quad \text{for all } \psi \in C_c^\infty(\mathbb{R}^d) \right\}
    \]
    the quotient space $X = Y / N$ is a separable Banach space satisfying $X^* \cong W^{1,1}(\mathbb{R}^d) \cap W^{1,\infty}(\mathbb{R}^d)$ via the gradient embedding $u \mapsto (u, \nabla u)$.

    Finally, we remark that we say that $g\in L^{\infty}((0,1);(L^1\cap L^{\infty})(\R^d))$ if for almost all $\eta \in (0,1)$ the mapping $\eta\mapsto g(\cdot,\eta) \in L^{1}(\R^d)$ is strongly measurable and essentially bounded, i.e.
\[
\norm{g}_{L^{\infty}( (0,1);(L^1\cap L^{\infty})(\R^d))}:= \max\{\esssup_{\eta \in  (0,1)}\norm{g(\cdot,\eta)}_{L^{1}(\R^d)}, \esssup_{(\eta,x) \in (0,1)\times \R^d}|g(x,\eta)| \} < \infty.
\]
\end{remark}

We finish this section establishing two technical lemmas, whose proofs may be found in Appendix. The first one is utilized later to establish that the algebra of transforms $ \mathcal{F}$ is well defined (Definition \ref{defalgebra}). The second one is a useful result that we apply several times in subsequent proofs. 

\begin{lemma} \label{coro2}
For any $g \in L^{\infty}((0,t_{*})\times (0,1);(L^1\cap L^{\infty})(\R^d))$ and  $f \in L^{\infty}((0,t_{*});\W((L^1\cap L^{\infty})(\R^d)))$ the integral
\[
\vf(t,x,y,\xi):=\izj g(t,x,\eta)f(t,y,\xi,d\eta),
\]
for almost all $(t,x,y)$ understood as an extension operator from Lemma \ref{impoest}, is well defined element of $L^{\infty}((0,t_{*})\times (0,1);(L^1\cap L^{\infty})(\R^{2d}))$ with 
\eqq{
\norm{\vf(t,\cdot,\cdot,\cdot)}_{L^{\infty}( (0,1);L^{p}(\R^{2d}))}  \leq \norm{g(t,\cdot,\cdot)}_{L^{\infty}( (0,1);L^{p}(\R^{d}))} \norm{f(t,\cdot,\cdot,\cdot)}_{\W(L^{p}(\R^d)))}, \hd p \in (1,\infty]
}{coro2ei}
and
\eqq{
\norm{\vf(t,\cdot,\cdot,\cdot)}_{L^{\infty}((0,1);L^{1}(\R^{2d}))}  \leq \norm{g(t,\cdot,\cdot)}_{L^{\infty}((0,1);L^{1}(\R^{d}))} \norm{f(t,\cdot,\cdot,\cdot)}_{\W(\mathcal{M}(\R^d))} \m{ for } a.a.\hd t \in (0,t_{*}).
}{coro2ej}
\end{lemma}

\begin{lemma}\label{hmj}
Let $f \in \W((L^1\cap L^{\infty})(\R^d))$. Let us extend $f$ periodically in $\xi$, convolve with standard smoothing periodic kernel $\rho^\ve$ in $\xi$ variable and  denote $f^{\ve}(x,\xi,\eta)=\izj \rho^\ve(\xi-\xi')f(x,d\xi',\eta)$. Then $f^\ve \in L^{\infty}((0,1)^2\times \R^d)$,  $f^\ve \rightarrow f$ in $L^{1}_\xi(H^{-1}_\eta;L^{1}(\R^d))$ and in  $L^{1}_\eta(H^{-1}_\xi;L^{1}(\R^d))$.
Moreover, for $Y=L^{\infty}(\R^d)$ or $Y=\mathcal{M}(\R^d)$  there hold the estimates $\norm{f^\ve}_{\W(Y)} \leq \norm{f}_{\W(Y)}$. Finally, for any $g, g^{\ve} \in L^{1}((0,1)\times \R^{dn}) \cap L^{\infty}((0,1);L^{1}(\R^{dn})) $ such that $g^\ve \rightarrow g$ in $L^{1}((0,1)\times \R^{dn})$, we have
\[
\izj g^{\ve}(y_1,\dots,y_n,\eta) f^{\ve}(x,\cdot,d\eta) \rightarrow \izj g(y_1,\dots,y_n,\eta) f(x,\cdot,d\eta) \m{ in } L^{1}_\xi((0,1)\times \R^{d(n+1)}), 
\]
\[
 \izj g^{\ve}(y_1,\dots,y_n,\xi) f^{\ve}(x,d\xi,\cdot) \rightarrow \izj g(y_1,\dots,y_n,\xi) f(x,d\xi,\cdot) \m{ in } L^{1}_\eta((0,1)\times \R^{d(n+1)}).
\]
\end{lemma}

\subsection{Solvability of the limiting equation}

In this section we establish the solvability of the limiting equation. To facilitate the regularizing approximation in subsequent results, we introduce a dissipative term with $\nu \geq 0$. Hence, we discuss the following problem
\eqq{
\partial_t f(t,x,\xi,\eta) + \divv ( f(t,x,\xi,\eta) V_1[f](t,x,\eta)) =  f(t,x,\xi,\eta) (A -V_2[f](t,x,\eta))+\nu \Delta f(t,x,\xi,\eta),
}{meannu}
where we recall that for $j=1,2$
\[
V_j[f](t,x,\eta) = \izj \ird  K_{j}(x-y) f(t,y,\eta,d\zeta)dy.
\]
The expression above shall be understood as the integral in (\ref{intmu}).
We define a weak solution to (\ref{meannu}) as follows. Let $K_1,K_2 \in L^{1}(\R^d), \nu \geq 0$. We say that $f\in L^{\infty}((0,t_{*});\W((L^{1}\cap L^{\infty})(\R^d))$ is a weak solution to \eqref{meannu} if 
\begin{align*}\label{weakdef}
&\int_0^{t_{*}}\izj \izj \ird\partial_t \varphi(t,x,\xi,\eta) f(t,x,\xi,d\eta) dx d\xi dt \\ \nonumber
 +& \int_0^{t_{*}}\izj \izj\int_{\mathbb{R}^d}\nabla\varphi(t,x,\xi,\eta)\cdot  V_{1}[f](t,x,\eta)f(t,x,\xi,d\eta)dx d\xi dt\\ \nonumber
+\nu&\int_0^{t_{*}}\izj \izj \int_{\mathbb{R}^d}\Delta\varphi(t,x,\xi,\eta)\,  f(t,x,\xi,d\eta) dx\,d\xi dt\\ \nonumber
+&\int_0^{t_*}\izj\izj \int_{\mathbb{R}^d}\varphi(t,x,\xi, \eta) [A-V_{2}[f](t,x,\eta)]\,f(t,x,\xi,d\eta) dx\,d\xi dt=0
\end{align*}
for any $\varphi\in W^{1,\infty}((0,\ t_*);C_c^2(\mathbb{R}^d);C([0,1]^2)) \cap C_{c}((0,t_*);C_c^2(\mathbb{R}^d);C([0,1]^2))$.
Note that the terms involving $V_1$ and $V_2$ are well defined due to Lemma \ref{impoestv}.

Moreover, Young inequality for convolution leads to 
\eqq{
\norm{V_i[f](t,\cdot,\cdot)}_{L^{\infty}((0,1)\times \R^d)}  \leq \norm{K_i}_{L^\infty(\R^d)} \esssup_{\eta \in (0,1)}\izj \norm{f(t,\cdot,\eta,\cdot)}_{\mathcal{M}(\R^d)}(d\zeta)\leq \norm{K_i}_{L^\infty(\R^d)} \norm{f(t,\cdot,\cdot,\cdot)}_{\W(\mathcal{M}(\R^d))}.
}{v1}
Similarly we obtain
\begin{align}
  &\norm{\nabla V_i[f](t,\cdot,\cdot)}_{L^{\infty}((0,1)\times \R^d)}  \leq  \norm{K_i}_{L^1(\R^d)} \norm{\nabla f(t,\cdot,\cdot,\cdot)}_{\W(L^{\infty}(\R^d))},\label{v3prim}\\
  &\norm{ V_1[f](t,\cdot,\cdot)}_{ L^{\infty}( (0,1);W^{1,\infty}(\R^d))} \leq \norm{K_1}_{W^{1,1}(\R^d)}\norm{f(t,\cdot,\cdot,\cdot)}_{\W(L^{\infty}(\R^d))},
\label{v2}\\
  &\norm{V_i[f](t,\cdot,\cdot)}_{L^{\infty}((0,1);L^{1}(\R^d))}   \leq \norm{K_i}_{L^1(\R^d)} \norm{ f(t,\cdot,\cdot,\cdot)}_{\W(\mathcal{M}(\R^d))}, \label{v3}\\
  &\norm{\nabla V_i[f](t,\cdot,\cdot)}_{L^{\infty}((0,1);L^{1}(\R^d))}   \leq \norm{K_i}_{L^1(\R^d)} \norm{ \nabla f(t,\cdot,\cdot,\cdot)}_{\W(\mathcal{M}(\R^d))}, \label{v4}
\end{align}
and finally
\eqq{
\norm{\divv V_1[f](t,\cdot,\cdot)}_{L^{\infty}((0,1);W^{1,\infty}(\R^d))}  \leq  \norm{\divv K_1}_{L^\infty(\R^d)} \left(\norm{  f(t,\cdot,\cdot,\cdot)}_{\W(\mathcal{M}(\R^d))} + \norm{ \nabla f(t,\cdot,\cdot,\cdot)}_{\W(\mathcal{M}(\R^d))} \right).
}{v5}

The above estimates are essential in establishing the following solvability result.

\begin{prop}
Let us assume $K_i\in L^1(\R^d)\cap L^\infty(\R^d)$ for $i=1,2$, $K_1 \in W^{1,1}(\R^d)$ $\divv K_1\in L^\infty(\R^d)$,  $K_2 \geq 0$ and  $f^0\in  \W((W^{1,1}\cap W^{1,\infty})(\R^d))$ with $f^0 \geq 0$. Then for any $\nu\geq 0$, there exists nonnegative $f\in L^\infty((0,\ t_*);\W(( W^{1,1}\cap W^{1,\infty})(\R^d)))$ such that $f$ is weak$^*$ continuous in $\W((W^{1,1}\cap W^{1,\infty})(\R^d))$, $f(t,\cdot,\cdot,\cdot)\Big|_{t=0} = f^0$ and $f$ is a weak solution to
(\ref{meannu}).
\label{existenceprop}
\end{prop}
\begin{proof}
We divide the proof into a few steps. At first we solve the linear equation.
\begin{lemma}\label{linearlemma}
Let  $K_i\in L^1(\R^d)\cap L^\infty(\R^d)$ for $i=1,2$, $K_1 \in W^{1,1}(\R^d)$ $\divv K_1\in L^\infty(\R^d)$,  $K_2 \geq 0$ and  $f^0\in  \W((W^{1,1}\cap W^{1,\infty})(\R^d))$ with $f^0 \geq 0$. Define 
\[
E:=\{f \in L^\infty((0, t_*);\, \W((W^{1,1}\cap W^{1,\infty})(\mathbb{R}^d))): f \geq 0\},
\]
\begin{align*}
&\norm{f}_{E}:=\\
&\max\left\{\norm{f}_{L^{\infty}((0,t_{*});\W(L^{\infty}(\R^d)))} + \norm{\nabla f}_{L^{\infty}((0,t_{*});\W(L^{\infty}(\R^d)))}
, \norm{f}_{L^{\infty}((0,t_{*});\W(\mathcal{M}(\R^d)))} + \norm{\nabla f}_{L^{\infty}((0,t_{*});\W(\mathcal{M}(\R^d)))} \right\}.
\end{align*}
For any $\Upsilon> \norm{f^0}_{\W((W^{1,1}\cap W^{1,\infty})(\R^d))}e^{A}$ we set $E_\Upsilon:=\left\{f\in E:\, \norm{f}_{E}\leq \Upsilon\right\}$.
For a given $g\in E$ consider the linear problem
\begin{equation}\label{independ2-linear}
\left\{
\begin{array}{l}
\displaystyle \partial_t f(t,x,\xi,\eta) + \divv ( f(t,x,\xi,\eta) V_1[g](t,x,\eta)) =  f(t,x,\xi,\eta) (A -V_2[g](t,x,\eta)) + \nu \Delta_x f(t,x,\xi,\eta),\\
f(0,\cdot,\cdot,\cdot)=f^0.
\end{array}
\right.
\end{equation}
Then, there exists exactly one $f \in E$ which is a weak solution to (\ref{independ2-linear}) and is weak$^*$ continuous with values in $\W((W^{1,1}\cap W^{1,\infty})(\R^d))$, such that $f(t,\cdot,\cdot,\cdot)\Big|_{t=0} = f^0$. Furthermore, if $g \in E_{\Upsilon}$, then $f \in E_{\Upsilon}$ for sufficiently small $t_{*}$ which depends on $\Upsilon$, $A$ and norms of $K_{1}$ and $K_{2}$.
\end{lemma}
We postpone the proof of the lemma to Appendix. In order to prove Proposition \ref{existenceprop} using Lemma~\ref{linearlemma} we fix nonnegative $f^0 \in \W((W^{1,1}\cap W^{1,\infty})(\R^d))$ and $t_{1}$ sufficiently small to satisfy the claim of Lemma~\ref{linearlemma}. We introduce the operator $P:E_{\Upsilon} \rightarrow E_{\Upsilon}$ (where we replaced $t_*$ by $t_1$ in the definiton of $E$ in a natural way) by $Pg=f$, where $f$ is a weak solution to (\ref{independ2-linear}). We show that $P$ is a contraction on $(E_{\Upsilon},\norm{\cdot}_{L^{\infty}((0,t_{1});\W(\mathcal{M}(\R^d)))})$ and apply the Banach fixed point theorem. At first, note that $(E_{\Upsilon},\norm{\cdot}_{L^{\infty}((0,t_{1});\W(\mathcal{M}(\R^d)))})$ is a complete metric space. Indeed, any Cauchy sequence $f^n$ in \\$(E_{\Upsilon},\norm{\cdot}_{L^{\infty}((0,t_{1});\W(\mathcal{M}(\R^d)))})$ converges to some limiting $f$ in $L^{\infty}((0,t_{1});\W(\mathcal{M}(\R^d)))$. Since $f^n$ is uniformly bounded in $E$ by weak$^*$ compactness the limiting $f$ also belongs to $E_{\Upsilon}$.

 To show that $P$ is a contraction on $(E_{\Upsilon},L^{\infty}((0,t_{1});\W(\mathcal{M}(\R^d))))$ consider $g_1,g_2\in E_\Upsilon$ and note that $Pg_1 - Pg_2$ satisfies in a weak sense
\[
\partial_t(Pg_1-Pg_2)+\divv(V_1[g_2]\,(Pg_1-Pg_2))+\divv((V_1[g_1]-V_1[g_2])Pg_1)
\]
\eqq{
=\nu\Delta (Pg_1-Pg_2) + (A-V_2[g_1])(Pg_1-Pg_2) - Pg_2(V_2[g_1] - V_2[g_2])
}{contreq}
with zero initial condition. We would like to apply a priori estimate (\ref{propL1}) for fixed $\xi,\eta$, however recall that $Pg_1,Pg_2$ are vector-valued measures. Hence, we argue by approximation. Let us extend $Pg_1,Pg_2$ periodically in $\xi$ and denote $Pg_j^\ve(t,x,\xi,\eta):=\int_0^1 \vf^\ve(\xi-\xi')Pg_j(t,x,d\xi',\eta)$, where $\vf^\ve$ is a standard periodic mollifying kernel, for $j=1,2$. Then $Pg_1^\ve,Pg_2^\ve$ are bounded w.r.t. $\xi$ and $\eta$, convolving (\ref{contreq}) with $\vf^\ve$ we arrive at

\[
\partial_t(Pg_1-Pg_2)^\ve+\divv(V_1[g_2]\,(Pg_1-Pg_2)^\ve)+\divv((V_1[g_1]-V_1[g_2])(Pg_1)^\ve)
\]
\[
=\nu\Delta (Pg_1-Pg_2)^\ve + (A-V_2[g_1])(Pg_1-Pg_2)^\ve - (Pg_2)^\ve(V_2[g_1] - V_2[g_2]).
\]
Applying the classical $L^1$ estimate (\ref{propL1}) leads to
\begin{align}\label{sumjot}
   &\Vert (Pg_1-Pg_2)^\ve(t,\cdot,\xi,\eta)\Vert_{L^1(\R^d)}\leq \sum_{j=1}^3 J_j:=\int_0^t \norm{\divv (V_1[g_1-g_2](s,\cdot,\eta)(Pg_1)^\ve(s,\cdot,\xi,\eta)))}_{L^1(\R^d)}ds \nonumber +\\
 &   \int_0^t \norm{(A-V_2[g_1](s,\cdot,\eta))(Pg_1-Pg_2)^\ve(s,\cdot,\xi,\eta)}_{L^1(\R^d)}ds
+\int_0^t\norm{(Pg_2)^\ve(s,\cdot,\xi,\eta)V_2[g_1-g_2](s,\cdot,\eta)}_{L^1(\R^d)}ds.
\end{align}
Let us estimate term by term. In order to estimate $J_1$ note that 
\[
\norm{\divv V_1[g_1-g_2](t,\cdot,\xi,\eta)}_{L^{1}(\R^d)} \leq \norm{\divv K_1}_{L^{1}(\R^d)}\norm{(g_1-g_2)(t,\cdot,\xi,\eta)}_{L^{1}(\R^d)}.
\]
Applying this estimate  together with (\ref{v3}) leads to
\begin{align}\label{sumjota}
J_{1}&\leq \int_0^t \Vert \divv V_1[g_1-g_2](s,\cdot,\eta)\Vert_{L^1( \R^d)}\Vert (Pg_1)^\ve(s,\cdot,\xi,\eta)\Vert_{{L^{\infty}( \R^d)}}ds \\ \nonumber
& +\int_0^t \Vert V_1[g_1-g_2](s,\cdot,\eta)\Vert_{L^1(\R^d)}\Vert \nabla (Pg_1)^\ve(s,\cdot,\xi,\eta)\Vert_{{L^{\infty}( \R^d)}}ds  \\ \nonumber
 & \leq \norm{\divv K_{1}}_{L^{1}(\R^d)}\int_0^t\norm{(g_1-g_2)(s,\cdot,\cdot,\cdot)}_{\W(\mathcal{M}(\R^d))}\norm{(Pg_1)^\ve(s,\cdot,\xi,\eta)}_{L^{\infty}(\R^d)}ds\\ \nonumber
 & +\norm{ K_{1}}_{L^{1}(\R^d)}\int_0^t\norm{(g_1-g_2)(s,\cdot,\cdot,\cdot)}_{\W(\mathcal{M}(\R^d))}\norm{\nabla (Pg_1)^\ve(s,\cdot,\xi,\eta)}_{L^{\infty}(\R^d)}ds .   
\end{align}
Similarly, applying (\ref{v1}) we have
\eqq{
J_2 \leq \int_0^t\norm{(Pg_1-Pg_2)^\ve(s,\cdot,\xi,\eta)}_{L^1( \R^d)} ds\left(A + \norm{K_{2}}_{L^{1}(\R^d)}\norm{g_1}_{L^{\infty}((0,t_{1});\W(L^{\infty}(\R^d)))}\right).
}{sumjotb}
Finally, by (\ref{v3})
\eqq{
   J_3 \leq \norm{K_{2}}_{L^{1}(\R^d)}\int_0^t\norm{(g_1-g_2)(s,\cdot,\cdot,\cdot)}_{\W(\mathcal{M}(\R^d))} \Vert (Pg_2)^\ve(s,\cdot,\xi,\eta)\Vert_{L^{\infty}(\R^d)}ds.
}{sumjotc}
%Note that by standard properties of mollification for a Banach space $X$
%\begin{align}\label{normsmooth}
%\esssup_{\xi \in (0,1)}\izj \norm{f^\ve(\cdot,\xi,\cdot)}_{X}(d\eta) &\leq \esssup_{\xi \in (0,1)} \izj \izj \vf^\ve(\xi-\xi')\norm{f(\cdot,\cdot,\eta)}_{X}(d\xi')d\eta
%\\ \nonumber
%&= \esssup_{\xi \in (0,1)} \izj  \vf^\ve(\xi-\xi')\izj\norm{f(\cdot,\xi',\cdot)}_X(d\eta)  d\xi' \\  \nonumber
%&\leq \esssup_{\xi \in (0,1)} \izj \norm{f(\cdot,\xi,\cdot)}_X(d\eta)
%\end{align}
%and similarly
%\eqq{
%\esssup_{\eta \in (0,1)}\izj \norm{f^\ve(\cdot,\cdot,\eta)}_{X}(d\xi) \leq \esssup_{\eta \in (0,1)}\izj \norm{f(\cdot,\cdot,\eta)}_{X}(d\xi). 
%}{normsmoothb}
Inserting the estimates (\ref{sumjota}) - (\ref{sumjotc}) into (\ref{sumjot}), applying $\W$ - norm and the estimate from Lemma \ref{hmj}  we arrive at
\begin{align*}
    &\Vert (Pg_1-Pg_2)^\ve (t,\cdot,\cdot,\cdot)\Vert_{\W(\mathcal{M}(\R^d))}\leq  
    \left(A + \norm{K_{2}}_{L^{1}(\R^d)}\Upsilon\right)\int_0^t\norm{(Pg_1-Pg_2)(s,\cdot,\cdot,\cdot)}_{\W(\mathcal{M}( \R^d))}\\
    &+ \Upsilon \left(\norm{\divv K_1}_{L^1(\R^d)} + \norm{ K_1}_{L^1(\R^d)} + \norm{ K_2}_{L^1(\R^d)}\right)\int_0^t\norm{(g_1-g_2)(s,\cdot,\cdot,\cdot)}_{\W(\mathcal{M}(\R^d))}.
\end{align*}
By lower semicontinuity of the norm we may estimate 
\[
\Vert (Pg_1-Pg_2) (t,\cdot,\cdot,\cdot)\Vert_{\W(\mathcal{M}(\R^d))}\leq \liminf_{\ve\rightarrow 0} \Vert (Pg_1-Pg_2)^\ve (t,\cdot,\cdot,\cdot)\Vert_{\W(\mathcal{M}(\R^d))}.
\]
Inserting it in the estimate above and applying Gr\"onwall inequality we get
\begin{align*}
  &\Vert (Pg_1-Pg_2)\Vert_{L^{\infty}((0,t_{1});\W(\mathcal{M}(\R^d)))} \leq \\
  &\exp\left(A + \norm{K_{2}}_{L^{1}(\R^d)}\Upsilon\right)\Upsilon \left(\norm{\divv K_1}_{L^1(\R^d)} + \norm{ K_1}_{L^1(\R^d)} + \norm{ K_2}_{L^1(\R^d)}\right)t_{1}\norm{g_1-g_2}_{L^\infty((0,t_{1});\W(\mathcal{M}(\R^d)))}. 
\end{align*}
Choosing $t_{1}$ dependent on $\Upsilon$, $A$ and norms of $K_{1}$ and $K_{2}$, sufficiently small we obtain that $P$ is contractive on $(E_{\Upsilon},\norm{\cdot}_{L^{\infty}((0,t_{1});\W(\mathcal{M}(\R^d)))})$. By the Banach fixed point theorem there exists a unique fixed point  of $P$, leading to a weak solution of \eqref{meannu} in $[0, t_1]$.
The estimates (\ref{l1bound}), (\ref{finfty}),  with $g$ replaced by $f$ do not blow up in a finite time. Also, (\ref{gradl1}) and (\ref{gradlinfty}) do not blow up in a finite time by Gr\"onwall argument. Hence, we may extend the solution to any finite time interval, iterating the contraction proof step by step.
\end{proof}

We finish this section with the following simple but crucial observation which explains how we can create a solution to (\ref{meanmean}) from the solutions to (\ref{part2}).
\begin{prop}\label{fn}
Let  $\fji$ for $i,j\in\{1,\dots,N\}$ be the measures defined in (\ref{defft}). Set
\eqq{
f^N(t,x,\xi,\eta):= N\sum_{i=1}^N \sum_{j=1}^N \fji(t,x)\mI_j(\xi)\mI_i(\eta), \hd \hd \mI_i:=\mathbb{I}_{[\frac{i-1}{N},\frac{i}{N})}, \hd\mI_j:=\mathbb{I}_{[\frac{j-1}{N},\frac{j}{N})}.
}
{deffn}
Then $f^N$ is a weak solution to (\ref{meanmean}) which belongs to $L^{\infty}((0,t_{*});\W((W^{1,1}\cap W^{1,\infty})(\R^d)))$. Furthermore, 
\eqq{
\izj\izj f^N(t,x,\xi,\eta)d\xi d\eta = \frac{1}{N}\sum_{i,j=1}^N\fji(t,x).
}{fnint}
\end{prop}
\begin{proof}
The identity (\ref{fnint}) is trivial. 
Let us now determine the density of $f^j_i(0)$. For any set $U \in \mathcal{B}(\R^d)$ we have
\[
f^j_i(0,\cdot)(U)= \ird I(\{x:X^0_{i}(x)\in U\})\wji^0(x)d\p(x)  = \int_{(X_i^0)^{-1}(U)} \wji^0 d\p 
= \int_{(X_i^0)^{-1}(U)} \mathbb{E}(\wji^0|X^0_i)d\p,
\]
where in the last equality we used the definition of conditional expectation. Denoting $h_{ji}(x) = \mathbb{E}(\wji^0|X_i^0~=~x)$, we may write 
\[
f^j_i(0,\cdot)(U)= \int_{(X_i^0)^{-1}(U)} h_{ji}(X^{0}_i)d\p = \int_{U} h_{ji}(x) f_{X^{0}_i}(x)dx = \int_U g_{ij}(x)dx,
\]
where $g_{ij}$ is defined in (\ref{asf}). Thus $g_{ij}$ is a density of $f^j_i(0)$ and by the assumption (\ref{asf}) $f^j_i(0) \in (W^{1,1}\cap W^{1,\infty})(\R^d)$.
In order to show that   $f^j_i \in L^{\infty}((0,t_{*});(W^{1,1}\cap W^{1,\infty})(\R^d)))$ for each $i,j \in \{1,\dots,N\}$ we apply the analogous reasoning as in the proof of Proposition 8 from \cite{nasza}. For  the sake of completeness we provide it in the Appendix. With this established, it becomes clear that $f^N \in L^{\infty}((0,t_{*});\W((W^{1,1}\cap W^{1,\infty})(\R^d)))$, since
\[
\sup_{\xi\in(0,1)}\izj \norm{f^N(t,\cdot,\xi,\cdot)}_{(W^{1,1}\cap W^{1,\infty})(\R^d)}(d\eta) \leq \max_{1\leq j \leq N}\sum_{i=1}^N\norm{f^j_i(t)}_{(W^{1,1}\cap W^{1,\infty})(\R^d)}
\]
and
\[
\sup_{\eta\in(0,1)}\izj \norm{f^N(t,\cdot,\cdot,\eta)}_{(W^{1,1}\cap W^{1,\infty})(\R^d)}(d\xi) \leq \max_{1\leq i \leq N}\sum_{j=1}^N\norm{f^j_i(t)}_{(W^{1,1}\cap W^{1,\infty})(\R^d)}.
\]
In order to show that $f^N$ satisfies (\ref{meanmean}) we
 define for every $i\in\{1,\dots,N\}$
\[
f^N_i(t,x,\xi) := N\sum_{j=1}^N\fji(t,x)\mI_j(\xi), \hd f^i_N(t,x,\eta) := N\sum_{k=1}^N\fik(t,x)\mI_k(\eta).
\]
Then, clearly
\[
f^N(t,x,\xi,\eta) = \sum_{i=1}^N f^N_i(t,x,\xi) \mI_{i}(\eta) \hd \m{ and } \hd 
f^N(t,x,\xi,\eta) =\sum_{i=1}^N f_N^i(t,x,\eta) \mI_{i}(\xi).
\]
For each $i=1,\dots, N$, we multiply the equation (\ref{semic}) by $N \mI_j(\xi)$ and sum over $j=1,\dots, N$. Then, 
\[
\partial_t f_i^N(t,x,\xi) + \divv\left(f_i^N(t,x,\xi)\ird K_{1}(x-y)\sum_{k=1}^N\fik(t,dy)\right) = f_i^N(t,x,\xi)\left[A - \ird K_{2}(x-y)\sum_{k=1}^N\fik(t,dy)\right].
\]
Note that
\[
\izj f^i_N(t,y,\zeta)d\zeta = N \sum_{l=1}^N\int_{\mI_l}\sum_{k=1}^N \fik(t,y) \mI_k(\zeta)d\zeta = \sum_{k=1}^N \fik(t,y)
\]
and thus,
\eqq{
\partial_t f_i^N(t,x,\xi) + \divv\left(f_i^N(t,x,\xi)\ird K_{1}(x-y)\izj f^i_N(t,y,d \zeta)\right)dy = f_i^N(t,x,\xi)\left[A - \ird K_{2}(x-y)\izj f^i_N(t,y,d \zeta)dy\right].
}{nafin}
Let us now multiply (\ref{nafin}) by $\mI_i(\eta)$ and sum over $i=1,\dots, N$
\[
\partial_t f^N(t,x,\xi,\eta) + \divv\sum_{i=1}^N \left(f_i^N(t,x,\xi)\mI_i(\eta)\ird K_{1}(x-y)\izj \mI_i(\eta) f^i_N(t,y,d \zeta)dy\right)
\]
\[
=Af^N(t,x,\xi,\eta) - \sum_{i=1}^N \left( f_i^N(t,x,\xi)\mI_i(\eta)\ird K_{2}(x-y)\izj \mI_i(\eta)f^i_N(t,y,d \zeta)dy\right).
\]
Since $\mI_i(\eta) \mI_j(\eta) \equiv 0$ for $i\neq j$ we have in fact 
\[
\sum_{i=1}^N \left(f_i^N(t,x,\xi)\mI_i(\eta)\ird K_{1}(x-y)\izj \mI_i(\eta) f^i_N(t,y,d \zeta)dy\right)
\]
\[
= \sum_{i=1}^N\sum_{j=1}^N \left(f_i^N(t,x,\xi)\mI_i(\eta)\ird K_{1}(x-y)\izj \mI_j(\eta) f^j_N(t,y,d \zeta)dy\right)
\]
\[
=f^N(t,x,\xi,\eta)
\ird K_{1}(x-y)\izj f^N(t,y,\eta,d \zeta)dy
\]
and indeed we obtain that $f^N$ solves (\ref{meanmean}).

\end{proof}

\subsection{New observables: definition and properties}

In this section we adjust the ideas introduced in \cite{JPS2025} to our setting. We introduce and rigorously define the new observables. Then, we utilize them to establish the key convergence result in Theorem \ref{lg1}. We finish this section showing that if $f$ is a weak solution to (\ref{meannu}), then the family of observables built upon $f$ solves the linear Vlasov hierarchy.

At first, we define the family of rooted directed trees and then use this notation to introduce our observables.

\begin{defi}(see \cite[Definition 4.1]{JPS2025})
~
\begin{enumerate}
\item We call a   (simple) directed graph a pair $G=(V,E)$ where $V$ is a finite set (and represents vertices) and $E\subseteq \{(i,j)\in V\times V:\,i\neq j\}$, where $(i,j)$ is an ordered pair (and represent the edges). 
We denote by $|G|$ the number of vertices. 
\item A graph $T=(V,E)$ is called rooted directed tree if:
\begin{itemize}
\item $V=\{1,2,\dots,|T|\}$ and the vertex indexed by $1$ is called the root, 
\item for every $u\in V$, $(u,1)\notin E$,
\item for every $u\in V$, $u\neq 1$, there exists exactly one $v \in V$ such that $(v,u) \in E$, 
\item for every $u\in V$, $u\neq 1$, there exists a directed path from $\{1\}$ to $u$.
\end{itemize}
A vertex in tree which is not a root and has only one edge connecting it with another vertex is called a leaf.

\item We define the family $\Tree_n$ of  labeled rooted directed trees of order $n$ by the following recursive formula
\begin{align*}
\Tree_1&:=\{T_1\},\\
\Tree_{n+1}&:=\{T+i:\,T\in \Tree_n,\,i\in \{1,\ldots,n\}\},\quad n\in \mathbb{Z}_0^+,
\end{align*}
where $T_1$ is the  tree with only one vertex $\{1\}$ and for any tree $T$ with vertices $\{1,\ldots,n\}$ by $T+i$ we denote a tree created from $T$ by adding a leaf indexed by $n+1$ to an $i$-th vertex. The family of all labeled rooted directed trees of arbitrary order with at least two vertices is then defined by $\Tree:=\cup_{n=2}^\infty\Tree_n$. 
\end{enumerate}
\end{defi}

Now we are ready to introduce our observables indexed by trees. Similarly as in \cite{JPS2025} at first we define the observables for more regular functions. Then we extend the definition for vector-valued measures.

\begin{defi}\label{tau}
For any $T\in \Tree$ and $f\in L^\infty((0,t_{*});L^{\infty}_\xi L^1_\eta ( (L^1\cap L^{\infty})(\R^d)))\cap L^\infty((0,t_{*});L^{\infty}_\eta L^1_\xi ( (L^1\cap L^{\infty})(\R^d)))$, we set $n=|T|-1$ and define 
\eqq{
\tau(T,f)(t,x_1,\dots,x_n) = \int_{[0,1]^{|T|}} \prod_ {(k,m)\in E(T)} f(t,x_{m-1},\xi_k,\xi_m)d\xi_1\dots d\xi_{|T|}.
}{reptau}
\end{defi}

In order to extend the definition of $\tau$ for $f\in L^\infty((0,t_{*});\mathcal{W}((L^1\cap L^{\infty})(\R^d)))$ we introduce, similarly as in \cite{JPS2025} the following countable algebra  of
transforms over spaces of arbitrary large dimensions in the following way.
\begin{defi}\label{defalgebra}
Let $f \in L^{\infty}((0,t_{*});\W((L^1\cap L^{\infty})(\R^d)))$. We define the algebra $\mathcal{F}$ of transforms as follows. For each transform $F \in \mathcal{F}$,
there exists  $n \in \mathbb{N}$ (the rank of $F$) so that $F$ maps $f$ into a scalar function $F(f)$
on $[0,t_{*}]\times\R^{dn} \times [0, 1]$. The algebra is constructed inductively:
\begin{itemize}
    \item $F_0 \in \mathcal{F}$, where  $F_0(f)(t,x,\xi):=\izj f(t,x,\xi,d\eta).$
    \item If $F_1$ of rank $k$ and $F_2$ of rank $l$ belong to $\mathcal{F}$, then $F_1\cdot F_2 \in \mathcal{F}$, where
    \[
    (F_{1}\cdot F_2)(f)(t,x_1,\dots,x_{k+l},\xi):=F_1(f)(t,x_1,\dots,x_k,\xi)F_2(f)(t,x_{k+1},\dots,x_{k+l},\xi). 
    \]
    \item If $F$ of rank $k$ belongs to $\mathcal{F}$, then $F^{*} \in \mathcal{F}$, where 
    \[
    F^{*}(f)(t,x_1,\dots,x_{k+1},\xi):=\izj F(f)(t,x_1,\dots,x_k,\eta)f(t,x_{k+1},\xi,d \eta)
    \]
   and  the integral is understood in the sense of extension operator from Lemma \ref{impoest}. 
\end{itemize}
\end{defi}

Now we are ready to establish the following proposition.

\begin{prop}\label{fntau}
The transforms $F \in \mathcal{F}$ are well defined for $f\in L^\infty((0,t_{*});\W((L^1\cap L^{\infty})(\R^d)))$. If $F\in \mathcal{F}$ is of rank $n$, then $F(f) \in L^{\infty}((0,t_{*})\times (0,1);(L^{1}\cap L^{\infty})(\R^{dn}))$
and for almost all $t \in (0,t_{*})$
\begin{align}\label{fnest}
 &\norm{F(f)(t,\cdot,\cdot)}_{L^{\infty}((0,1);L^{p} (\R^{dn}))} \leq \norm{f(t,\cdot,\cdot,\cdot)}^n_{\W( L^p(\R^d))}, \hd p \in (1,\infty],  \\ 
 & \label{fnestj}\norm{F(f)(t,\cdot,\cdot)}_{L^{\infty}((0,1);L^{1} (\R^{dn}))} \leq \norm{f(t,\cdot,\cdot,\cdot)}^n_{\W( \mathcal{M}(\R^d))}.
\end{align}

Moreover, if $f \in L^{\infty}((0,t_{*});\W((W^{1,1}\cap W^{1,\infty}))(\R^d))$, then for every $F \in \mathcal{F}$ of rank $n$
\begin{align}\label{fnestw}
&\norm{\nabla F(f)(t,\cdot,\cdot)}_{L^{\infty}((0,1);L^{p} (\R^{dn}))} \leq \norm{ f(t,\cdot,\cdot,\cdot)}^n_{\W( W^{1,p}(\R^d))}, \hd p \in (1,\infty], \\
& \label{fnestwj} \norm{\nabla F(f)(t,\cdot,\cdot)}_{L^{\infty}((0,1);L^{1} (\R^{dn}))} \leq \max\{\norm{ f(t,\cdot,\cdot,\cdot)}^n_{\W( \mathcal{M}(\R^d))},\norm{\nabla f(t,\cdot,\cdot,\cdot)}^n_{\W( \mathcal{M}(\R^d))}\},
\end{align}
where we define the gradient norm as the maximum absolute value of its components.
The result applies to time independent functions with the obvious modifications.
\end{prop}
\begin{proof}
Let us proceed by induction. We begin with $n=1$. At first we have to show that 
$(t,\xi)\mapsto F_0(f)(t,x,\xi) := \izj f(t,x,\xi,d\eta)$ 
is strongly measurable with values in $L^{1}(\R^d)$. We apply a similar approach as in the proof of Lemma \ref{coro2}. From the regularity of $f$ we obtain that for almost all $(t,\xi)$ $\izj f(t,\cdot,\xi,d\eta) \in L^{1}(\R^d)$. By Pettis theorem it is enough to show weak measurability. Let $\Phi \in C_{0}(\R^d)$, then 
\[
\ird \Phi(x)\izj f(t,x,\xi,d\eta)dx = \izj \langle \Phi,  f(t,\cdot,\xi,d\eta)\rangle_{C_0(\R^d)\times  \mathcal{M}(\R^d)},
\]
is measurable since $f \in L^{\infty}((0,t_{*});\W(\mathcal{M}(\R^d)))$ defined with weak$^*$ topology. To show that the expression above is measurable also for arbitrary $\Phi \in L^{\infty}(\R^d)$, we choose the approximating sequence of $\Phi_n \in C_{c}(\R^d)$, with $\norm{\Phi_n}_{L^{\infty}(\R^d)} \leq \norm{\Phi}_{L^{\infty}(\R^d)}$, such that $\Phi_n \rightarrow \Phi$ pointwisely almost everywhere. Then, for almost all $(t,\xi)$ 
\[
\abs{\Phi_n(x)\izj f(t,x,\xi,d\eta)} \leq \norm{\Phi}_{L^{\infty}(\R^d)}\abs{\izj f(t,x,\xi,d\eta)} \in L^{1}(\R^d),
\]
 hence by the Lebesgue dominated convergence theorem we obtain that for almost all $(t,\xi)$
 \[
 \ird \Phi(x)\izj f(t,x,\xi,d\eta)dx =\lim_{n\rightarrow \infty} \ird \Phi_n(x)\izj f(t,x,\xi,d\eta)dx
 \]
 and the mapping $(t,\xi)\mapsto  \ird \Phi(x)\izj f(t,x,\xi,d\eta)dx$ is measurable as almost everywhere pointwise limit of measurable functions. Hence, $(t,\xi)\mapsto F_0(f)(t,\cdot,\xi)$ is measurable with values in $L^{1}(\R^d)$. To show the estimate we note that for almost all $t \in (0,t_{*})$
\[
\norm{F_0(f)(t,\cdot,\cdot)}_{L^{\infty}( (0,1);L^{p}(\R^{d}))} \leq \esssup_{\xi \in (0,1)}\izj \norm{f(t,\cdot,\xi,\cdot)}_{L^{p}(\R^d)}(d\eta) = \norm{f(t,\cdot,\cdot,\cdot)}_{L^{\infty}_\xi(\mathcal{M}_\eta(L^{p}(\R^d)))}
\]
for $p > 1$ and 
\[
\norm{F_0(f)(t,\cdot,\cdot)}_{L^{\infty}( (0,1);L^{1}(\R^{d}))} \leq  \norm{f(t,\cdot,\cdot,\cdot)}_{L^{\infty}_\xi(\mathcal{M}_\eta(\mathcal{M}(\R^d)))}.
\]
Hence,
 (\ref{fnest}) and (\ref{fnestj}) follow for $n=1$.
Let us assume that for fixed $n$ every transform $F\in \mathcal{F}$ of rank $k \leq n-1$ satisfies $F(f) \in L^{\infty}((0,t_{*})\times (0,1);(L^{1}\cap L^{\infty})(\R^{dk}))$ with the estimates (\ref{fnest}) and (\ref{fnestj}) with number $n$ replaced by $k$.
We show that under this assumption (\ref{fnest}) and (\ref{fnestj}) follow for any transform  of rank $n$. Indeed, for fixed $F \in \mathcal{F}$ of rank $n$ one of the two scenarios is satisfied. Either there exist $F_1$ of rank $k < n$ and $F_2$ of rank $l<n$ with  $k+l = n$, such that
\eqq{
F(t,x_1,\dots x_n,\xi) = F_{1}(t,x_1,\dots x_k,\xi)\cdot F_{2}(t,x_{k+1},\dots x_{k+l},\xi)
}{Fjrep}
or there exists $F_3$ of rank $n-1$ such that
\eqq{
F(f)(t,x_1,\dots x_n,\xi) = \izj F_{3}(f)(t,x_1,\dots x_{n-1},\eta)f(t,x_n,\xi,d \eta).
}{Fdrep}
In the first case we note that $(t,\xi)\mapsto F_{1}(f)(t,\cdot,\xi)\cdot F_{2}(f)(t,\cdot,\xi)$ is measurable with values in $L^{1}(\R^{d(k+l)})$, since $(t,\xi)\mapsto (F_{1}(f)(t,\cdot,\xi), F_{2}(f)(t,\cdot,\xi))$ is strongly measurable on $L^{1}(\R^{dk})\times L^{1}(\R^{dl})$ and the operator $T:L^{1}(\R^{dk})\times L^{1}(\R^{dl})\rightarrow L^{1}(\R^{d(k+l)})$ given by $T(F_1,F_2) = F_1 \cdot F_2$ is bounded and bilinear.
Applying the induction hypothesis we immediately arrive at the estimate 
\[
\norm{F(f)(t,\cdot,\cdot)}_{L^{\infty}((0,1);L^{p}( \R^{d(k+l)}))} \leq \norm{f(t,\cdot,\cdot,\cdot)}^{k+l}_{\W( L^p(\R^d))},
\]
with the obvious modification in case $p=1$.
In the second case the measurability and estimates (\ref{fnest}), (\ref{fnestj}) follow from the induction hypothesis and Lemma \ref{coro2}.
Hence, from the principle of mathematical induction, every  transform $F\in \mathcal{F}$ is well defined and satisfies the estimates in (\ref{fnest}) and (\ref{fnestj}).  In order to obtain higher Sobolev bounds (\ref{fnestw}), (\ref{fnestwj}), we note that analogously we obtain
\begin{align*}
    &\norm{\nabla F_0(f)(t,\cdot,\cdot)}_{L^{\infty}( (0,1);L^{p}(\R^{d}))} \leq  \norm{\nabla f(t,\cdot,\cdot,\cdot)}_{L^{\infty}_\xi(\mathcal{M}_\eta(L^{p}(\R^d)))} \hd \m{ for } \hd p>1,\\
&\norm{\nabla F_0(f)(t,\cdot,\cdot)}_{L^{\infty}( (0,1);L^{1}(\R^{d}))} \leq  \norm{\nabla f(t,\cdot,\cdot,\cdot)}_{L^{\infty}_\xi(\mathcal{M}_\eta(\mathcal{M}(\R^d)))}.
\end{align*}
We again proceed by induction. Assume that (\ref{fnestw}), (\ref{fnestwj}) are true for any transform of rank $k\leq n-1$ for fixed $n$. Let take $F \in \mathbb{F}$ of rank $n$. If $F$ is represented by  (\ref{Fjrep}), then applying induction hypothesis and (\ref{fnest}) we obtain for $i\in \{1,\dots,k \}$ and  $p>1$ with the obvious modification in case $p=\infty$
\begin{align*}
   & \sup_{\xi \in (0,1)}\int_{\R^{dn}}|\nabla_{x_i}F(f)(t,x_1,\dots,x_n,\xi)|^{p}dx_1\dots dx_n \\
&=  \sup_{\xi \in (0,1)}\int_{\R^{dk}}|\nabla_{x_i}F_1(f)(t,x_1,\dots,x_k,\xi)|^{p}dx_1\dots dx_k\int_{\R^{dl}}|F_2(f)(t,x_{k+1},\dots,x_n,\xi)|^{p}dx_{k+1}\dots dx_n \\
&\leq \norm{ f(t,\cdot,\cdot,\cdot)}^{kp}_{\W( W^{1,p}(\R^d))}\norm{ f(t,\cdot,\cdot,\cdot)}^{lp}_{\W( L^{p}(\R^d))} \leq \norm{ f(t,\cdot,\cdot,\cdot)}^{np}_{\W( W^{1,p}(\R^d))}. 
\end{align*}
Clearly, we obtain analogous estimate for $i \in \{k+1,\dots n\}$. If $F$ is given by (\ref{Fdrep}) we have by Lemma \ref{coro2} and induction hypothesis,  for $i \in \{1,\dots,n-1\}$,
\begin{align*}
&  \sup_{\xi \in (0,1)}\int_{\R^{dn}}|\nabla_{x_i}F(f)(t,x_1,\dots,x_n,\xi)|^{p}dx_1\dots dx_n\\  
&  =    \sup_{\xi \in (0,1)}\int_{\R^{dn}}\abs{\izj \nabla_{x_i}F_3(f)(t,x_1,\dots,x_{n-1},\xi)f(t,x_n,\xi,d\eta)}^{p}dx_1\dots dx_n  \\  
&  \leq \norm{\nabla_{x_i}F_3(f)(t,\cdot,\cdot)}_{L^{\infty}((0,1);L^{p}(\R^{d(n-1)}))}^p\norm{f(t,\cdot,\cdot,\cdot)}_{\W(L^{p}(\R^d))}^p \leq \norm{f(t,\cdot,\cdot,\cdot)}_{\W(W^{1,p}(\R^d))}^{np} 
\end{align*}
and similarly
\begin{align*}
&  \sup_{\xi \in (0,1)}\int_{\R^{dn}}|\nabla_{x_n}F(f)(t,x_1,\dots,x_n,\xi)|^{p}dx_1\dots dx_n\\  
&  =    \sup_{\xi \in (0,1)}\int_{\R^{dn}}\abs{\izj F_3(f)(t,x_1,\dots,x_{n-1},\xi)\nabla_{x_n} f(t,x_n,\xi,d\eta)}^{p}dx_1\dots dx_n  \\  
&  \leq \norm{F_3(f)(t,\cdot,\cdot)}_{L^{\infty}((0,1);L^{p}(\R^{d(n-1)}))}^p\norm{\nabla f(t,\cdot,\cdot,\cdot)}_{\W(L^{p}(\R^d))}^p \leq \norm{f(t,\cdot,\cdot,\cdot)}_{\W(W^{1,p}(\R^d))}^{np}. 
\end{align*}
Hence, we obtain (\ref{fnestw}) for arbitrary $n$-rank $F$ and by the mathematical induction principle (\ref{fnestw}) follows for every $F$. The proof of (\ref{fnestwj}) is analogous.
\end{proof}
Below we show that for any $T \in \Tree$  one may represent the observable $\tau(T,\cdot)$ by one of the transforms $F \in \mathcal{F}$. The lemma below is an equivalent of \cite[Lemma 5.4]{JPS2025} and also the proof is analogous. Hence, we leave the proof to the Appendix.

\begin{lemma}\label{ftau}
 For any $T\in \Tree$   there exists a transform $F\in \mathcal{F}$ of rank $n=|T|-1$ and the permutation $\sigma:\{1,\dots,n\}\rightarrow \{1,\dots,n\}$  such that    \[
 \tau(T,f)(t,x_1,\dots,x_n) = \izj F(f)(t,x_{\sigma(1)},\dots,x_{\sigma(n)},\eta)d\eta
 \]
 for any $f\in L^\infty((0,t_{*});L^{\infty}_\xi L^1_\eta ( (L^1\cap L^{\infty})(\R^d)))\cap L^\infty((0,t_{*});L^{\infty}_\eta L^1_\xi ( (L^1\cap L^{\infty})(\R^d)))$.
 \end{lemma}
 
 Now we are ready to extend the definition of $\tau$ for less regular $f$.
 \begin{defi}\label{deftf}(cf. \cite[Definition 5.5]{JPS2025})
   For any $T\in \Tree$,  we set $n=|T|-1$ and take the transform $F\in \mathcal{F}$ and  the permutation $\sigma:\{1,\dots,n\}\rightarrow \{1,\dots,n\}$ from Lemma \ref{ftau}. Then we define 
\[
\tau(T,f)(t,x_1,\dots,x_n) = \izj F(f)(t,x_{\sigma(1)},\dots,x_{\sigma(n)},\eta)d\eta,
\]  
for every $f\in L^\infty((0,t_{*});\W((L^1\cap L^{\infty})(\R^d)))$.
 \end{defi}

The Proposition established below ensures that Definition \ref{deftf} is correct.  

 \begin{prop}\label{deftfprop}
     The Definition \ref{deftf} is independent of the choice of transform $F$ and permutation, i.e. if for $F_1,F_2 \in \mathcal{F}$ and two permutations $ \sigma_1,\sigma_2$ we have that $\izj F_1^{\sigma_1}(f)(\cdot,\cdot,\eta)d\eta = \izj F_2^{\sigma_2}(f)(\cdot,\cdot,\eta)d\eta$ almost everywhere for every $f\in L^\infty((0,t_{*});L^{\infty}_\xi L^1_\eta \cap L^{\infty}_\eta L^1_\xi((L^1\cap L^{\infty})(\R^d)))$, then the identity  holds almost everywhere also for every $f\in L^\infty((0,t_{*});\W((L^1\cap L^{\infty})(\R^d)))$. Here, by $F^{\sigma}$ we understand the composition of the transform $F$ with the permutation $\sigma$ of spatial variables. 
\end{prop}

\begin{proof}
Let us take $f\in L^\infty((0,t_{*});\W((L^1\cap L^{\infty})(\R^d)))$, extend it periodically in $\xi$  on the whole real line and convolve in $\xi$ with standard periodic mollifier  $\rho^\ve$. Let us assume that for every $\ve > 0$ the $f^\ve(t,x,\xi,\eta):= \izj \rho^{\ve}(\xi-\xi')f(t,x,\xi',\eta)d\xi'$ satisfies $\izj F_1^{\sigma_1}(f^\ve)(\cdot,\cdot,\eta)d\eta = \izj F_2^{\sigma_2}(f^\ve)(\cdot,\cdot,\eta)d\eta$ almost everywhere.
In order to show that the identity above holds also for $f$ it is enough to apply the following lemma, whose proof we postpone to Appendix.
\begin{lemma}\label{apro}
Let $f\in L^\infty((0,t_{*});\W((L^1\cap L^{\infty})(\R^d)))$ and $f^\ve$ denotes the approximation by convolution with periodic mollifying kernel in $\xi$- variable given in Lemma \ref{hmj}. For every $n-$rank transform  $F \in \mathcal{F}$ we have $ F(f^\ve)(t) \rightarrow  F(f)(t)$ in $L^{1}((0,1)\times\R^{dn})$ for almost all $t \in (0,t_{*})$ as $\ve \rightarrow 0$.
\end{lemma}

\end{proof}

With the Definition \ref{deftf}, the estimates form Proposition \ref{fntau} immediately imply the estimates for $\tau(T,f)$.

\begin{coro}\label{tauest}
Let $f\in L^\infty((0,t_{*});\W((L^1\cap L^{\infty})(\R^d)))$, then   $\tau(T,f) \in L^{\infty}((0,t_{*});(L^{1}\cap L^{\infty})(\R^{dn}))$ for each $T \in \Tree$, $n=|T|-1$ and the following estimates hold for almost all $t \in (0,t_{*})$ and $p\in(1,\infty]$
\[
\norm{\tau(T,f)(t,\cdot)}_{L^{p}(\R^{dn})} \leq \norm{f(t,\cdot,\cdot,\cdot)}^n_{\W( L^p(\R^d))}, \hd \norm{\tau(T,f)(t,\cdot)}_{L^{1}(\R^{dn})} \leq \norm{f(t,\cdot,\cdot,\cdot)}^n_{\W( \mathcal{M}(\R^d))}.
\]
Furthermore, if $f\in L^\infty((0,t_{*});\W((W^{1,1}\cap W^{1,\infty})(\R^d)))$, then
\begin{align*}
&\norm{\nabla \tau(T,f)(t,\cdot)}_{L^{p}(\R^{dn})} \leq \norm{ f(t,\cdot,\cdot,\cdot)}^n_{\W( W^{1,p}(\R^d))}, \\ &\norm{\nabla \tau(T,f)(t,\cdot)}_{L^{1}(\R^{dn})} \leq \max\{\norm{ f(t,\cdot,\cdot,\cdot)}^n_{\W( \mathcal{M}(\R^d))},\norm{\nabla f(t,\cdot,\cdot,\cdot)}^n_{\W( \mathcal{M}(\R^d))}\}.
\end{align*}
\end{coro}

Now we aim at the main result of this section, namely Theorem \ref{lg1}. We consider it as an analogue of \cite[Theorem 5.1]{JPS2025}, hence we follow the similar reasoning in the proof, which is based on the key compactness lemma \cite[Lemma 5.7, Corollary 5.9]{JPS2025}.

\begin{lemma}\cite[Corollary 5.9]{JPS2025}\label{lg}
Consider any sequence $g_n$ in $L^\infty([0, 1])$. Then, there exists $\Phi: [0, 1]\to [0, 1]$, {\it a.e.} injective, measure-preserving, such that the following estimate is verified
      \[
\int_0^{1} |(g_n\circ\Phi)(\xi)-(g_n\circ\Phi)(\xi+h)|\,d\xi\leq 2^n\,\Vert g_n\Vert_{L^\infty(0,1)}\,2^{-C\,\sqrt{\log\frac{1}{\vert h\vert}}},
      \]
for any $0<\vert h\vert <1$, each $n\in \mathbb{N}$ and some universal constant $C$.
\end{lemma}

\begin{theo}\label{lg1}
    For any sequence $\{f^N\}$ bounded in $\W((W^{1,1}\cap W^{1,\infty})(\R^d))$ there exists a subsequence still denoted by $N$ and  $f \in \W((W^{1,1}\cap W^{1,\infty})(\R^d))$ such that for every $T \in \Tree$  $\tau(T,f^N) \rightarrow \tau(T,f)$ in $L^{p}_{loc}(\R^{dn})$ for every $p \in [1,\infty)$. Furthermore, if the sequence $f^N$ consists of nonnegative elements, then also the limiting $f$ is nonnegative.
\end{theo}
\begin{proof}
Since the family $\mathcal{F}$ is countable let us index its elements by $(F_{k})_{k=0}^{\infty}$. At first we will prove that for every $N\in \mathbb{N}$ there exist a measure-preserving map $\Phi_N:[0,\ 1]\rightarrow [0,\ 1]$, such that for every transform $F_k$ of rank $n_k$, there exists $\phi_k(x_1,\dots,x_{n_{k}},\xi)\in L^{\infty}((0,1);(W^{1,1}\cap W^{1,\infty})(\R^{dn_k}))$ such that we have
\eqq{
F_k(f^N)(x_1,\dots,x_{n_k},\Phi_N(\eta))\rightarrow \phi_k(x_1,\dots,x_{n_k},\eta)\quad \mbox{in}\quad L^p_{loc}([0, 1]\times \mathbb{R}^{dn_{k}}),
}{mapcon}
as $N\rightarrow \infty$ (up to a subsequence on $N$) for  any $1\leq p<\infty$.

 Proposition \ref{fntau} applied in time-independent case implies that for each $k \in \mathbb{N}$ 
\eqq{
\norm{F_k(f^N)}_{L^{\infty}((0,1);(W^{1,1}\cap W^{1,\infty})(\R^{dn_k}))} \leq \norm{f^N}^{n_k}_{\W((W^{1,1}\cap W^{1,\infty})(\R^d))}\leq C^{n_k}.
}{unif}

We would like to apply Lemma \ref{lg} pointwisely in $x$ variable. To do so, for every $k\in \mathbb{N}$ we introduce a countable dense subset of $\mathbb{R}^{dn_k}$
and denote it by $\mathcal{D}_n:=\{(x_1^l,\dots,x_{n_k}^l):\,l\in \mathbb{N}\}$. For every $k,\,l,\,N\in \mathbb{N}$ we introduce $g_{k,l}^N(\xi):=F_k(f^N)(x_1^l,\ldots,x_{n_k}^l,\xi)$ for $\xi\in [0,\ 1]$. Thus, for every fixed $N$ we may apply Lemma~\ref{lg} to the countable two-indexed sequence $\{g_{k,l}^N\}_{k,l\in \mathbb{N}}\subset L^\infty([0,\ 1])$. As a result for each $N\in \mathbb{N}$ there exists a measure-preserving map $\Phi_N:[0,\ 1]\longrightarrow [0,\ 1]$ satisfying for any $0<|h|<1$
\[
\sup_{N\in \mathbb{N}}\int_0^1 \vert  g_{k,l} ^N\circ \Phi_N(\xi+h) -g_{k,l} ^N\circ \Phi_N(\xi)\vert \,d\xi\leq C_{k,l}\,2^{-C\sqrt{\log \frac{1}{|h|}}}.
\]
Note that here $C,C_{k,l}$ are positive constants and  $C_{k,l}$ are indeed $N$-independent due to (\ref{unif}). Applying the Fr\'echet-Kolmogorov theorem together with the diagonal extraction technique, we may choose the subsequence  (still denoted by $N$) such that $ g_{k,l}^N\circ\Phi_N\rightarrow \tilde g_{k,l}$ in $L^1((0, 1))$ as $N\rightarrow \infty$ for every $k,l\in \mathbb{N}$ to some $\tilde g_{k,l}\in L^1((0, 1))$. Now we would like to show that the above convergence holds not only on $\mathcal{D}_n$ but on $\mathbb{R}^{dn_k}$. Hence, let us define the functions $\tilde g_k: \mathcal{D}_n \times [0,1]\longrightarrow \mathbb{R}$ by $\tilde g_k(x_1^l,\ldots,x_{n_k}^l,\xi):=\tilde g_{k,l} (\xi)$. Since $\tilde g_{k,l}$ is also {\em a.e.} limit of a  subsequence of $g_{k,l}^N\circ\Phi_N$, hence the bound in \eqref{unif}, which is uniform in $\xi$ and $W^{1,\infty}$ in $x$ leads to
\begin{align*}
&\vert \tilde g_k(x_1^l,\ldots,x_{n_k}^l,\xi)\vert \leq C^{n_k},\\
&\vert \tilde g_k(x_1^{l_1},\dots,x_{n_k}^{l_1},\xi)-\tilde g_k(x_1^{l_2},\ldots,x_{n_k}^{l_2},\xi)\vert \leq C^{n_k}\,\vert (x_1^{l_1},\ldots,x_{n_k}^{l_1})-(x_1^{l_2},\ldots,x_{n_k}^{l_2})\vert,
\end{align*}
for any $k,\,l,\,l_1,\,l_2\in \mathbb{N}$ and {\em a.e.} $\xi\in [0,\ 1]$. Thus, by continuity we may uniquely extend $\tilde g_k:[0,\ 1]\times \mathcal{D}_n\rightarrow \mathbb{R}$ to $\phi_k\in L^\infty((0,1); W^{1,\infty}(\R^{dn_k}))$ with $\Vert \phi_k\Vert_{L^\infty((0,1)); W^{1,\infty}(\R^{dn_k})}\leq C^{n_k}$. Finally, for given $\ve > 0$ and $(x_1,\dots,x_{n_k})\in \R^{dn_k}$ we may choose $(x_1^l,\dots,x_{n_k}^l) \in \mathcal{D}_n$ such that $2C^{n_k} \abs{(x_1,\dots,x_{n_k}) - (x_1^l,\dots,x_{n_k}^l)} < \ve/2$. Then, passing again to subsequence if necessary
\begin{align*}
&\abs{F_k(f^N)(x_1,\ldots,x_{n_k},\Phi_N(\xi))-\phi_k(x_1,\ldots,x_{n_k},\xi)} \leq \abs{F_k(f^N)(x^l_1,\ldots,x^l_{n_k},\Phi_N(\xi))-\phi_k(x^l_1,\ldots,x^l_{n_k},\xi)}\\
 +&\abs{F_k(f^N)(x_1,\ldots,x_{n_k},\Phi_N(\xi))-F_k(f^N)(x^l_1,\ldots,x^l_{n_k},\Phi_N(\xi))} + \abs{\phi_k(x_1,\ldots,x_{n_k},\xi) - \phi_k(x^l_1,\ldots,x^l_{n_k},\xi)}\\
\leq & \abs{F_k(f^N)(x^l_1,\ldots,x^l_{n_k},\Phi_N(\xi))-\phi_k(x^l_1,\ldots,x^l_{n_k},\xi)} + 2C^{n_k} \abs{(x_1,\dots,x_{n_k}) - (x_1^l,\dots,x_{n_k}^l)} < \ve
\end{align*}
for $N$ big enough. Thus, for every $k\in \mathbb{N}$ and every $x_1,\ldots,x_{n_k}\in \mathbb{R}^d$
\[F_k(f^N)(x_1,\ldots,x_{n_k},\Phi_N(\xi))\rightarrow \phi_k(x_1,\ldots,x_{n_k},\xi) \m{ as } N\rightarrow \infty \m{ for a.a. } \xi\in [0,\ 1].\]

It remains to improve the almost everywhere convergence to local $L^p$ convergence for any $1\leq p<~\infty$.   At first, note that by~(\ref{unif}) functions $F_k(f^N)(\cdot,\Phi_N(\cdot))$ are bounded uniformly w.r.t. $N$ in $L^\infty((0,1); W^{1,1}(\R^{dn_k}))\subseteq L^\infty((0,1); BV(\R^{dn_k})$. Hence, from  weak$^*$ compactness of closed ball in $L^\infty((0,1); BV(\R^{dn_k}))$ and the uniform estimate in $L^\infty((0,1); W^{1,\infty}(\R^{dn_k}))$ we obtain \\$\phi_k\in L^\infty((0,1); W^{1,1}(\R^{dn_k}))$.
Finally, applying the uniform bound (\ref{unif}) together with the Lebesgue dominated convergence theorem we obtain that $F_k(f^N)(\cdot,\Phi_N(\cdot))$ converges in $L^{p}_{loc}(\R^{dn_k}\times [0,1])$ to $\phi_k$ for every $p \in [1,\infty)$ and we arrive at (\ref{mapcon}).\\
In the second step of the proof we identify the limit. 
At first, for any a.e. injective measure preserving $\Phi:[0,1]\rightarrow [0,1]$, we define $(\Phi^{*}f)(x,\xi,d\eta):=(\Phi^{-1})_{\#}f(x,\Phi(\xi),d\eta)$. Applying Lemma \ref{impoest} we may ensure that the definition does not depend upon different choices of inverse $\Phi^{-1}$ on a set of zero measure. Indeed, denoting by $\Phi^{-1}_1$ and $\Phi^{-1}_2$ two choices of almost everywhere inverse of $\Phi$. Then, by Lemma \ref{impoest}, for any $\Psi \in L^{\infty}((0,1))$ and $f \in \W(X^{*})$ we have 
\begin{align*}
&\norm{\izj \Psi(\eta)(\Phi^{-1}_1)_{\#}f(\cdot,\Phi(\xi),d\eta) - \izj \Psi(\eta)(\Phi^{-1}_2)_{\#}f(\cdot,\Phi(\xi),d\eta)}_{X^{*}} \leq\\
& \izj \abs{\Psi(\Phi^{-1}_1(\eta)) - \Psi(\Phi^{-1}_2(\eta))}\norm{f(\cdot,\Phi(\xi),\cdot)}_{X^{*}}(d\eta) \leq 
\norm{\Psi(\Phi^{-1}_1) - \Psi(\Phi^{-1}_2)}_{L^{\infty}(0,1)}\norm{f(\cdot,\Phi(\cdot),\cdot)}_{L^{\infty}_\xi(\mathcal{M}_\eta(X^*))},
\end{align*}
which equals zero, because $\Psi(\Phi^{-1}_1) = \Psi(\Phi^{-1}_2)$ almost everywhere.
Furthermore, from the Definition \ref{space} we obtain that $\norm{f}_{\W(X^{*})} = \norm{\Phi^{*}f}_{\W(X^{*})}$. Hence, from the uniform boundedness there exists $f$ such that on the subsequence $\Phi_N^{*}f^N \overset{*}{\rightharpoonup} f$ in $\W((W^{1,1}\cap W^{1,\infty})(\R^d))$.
We will inductively prove that  $\phi_k = F_k(f)$ for every $k \in \mathbb{N}$. Let us begin with $k=0$ and $F_0(f) = \izj f(x,\xi,d\eta)$. Take $\vf \in C([0,1]),\psi \in L^{1}(\R^d)$, then by weak$^*$ convergence and since $\Phi_N$ is measure preserving we have
\begin{align*}
 &\izj\ird \vf(\xi) \psi(x) F_0(f^N)(x,\Phi_N(\xi)) dx d\xi  \\
 =&\izj \izj\ird \psi(x)  (\Phi_N^{*}f^N)(x,\xi,\cdot) \vf(\xi) dx (d\eta)  d\xi  \rightarrow \izj \izj \ird \psi(x)  f(x,\xi,\cdot) dx(d\eta) \vf(\xi) d\xi   
\end{align*}
and thus
$F_0(f^N)(\cdot,\Phi_N(\cdot)) \rightarrow \izj f(\cdot,\cdot,d\eta) = F_0(f)$ weakly$^*$ in $\mathcal{M}([0,1];L^{\infty}(\R^d))$.
Hence, by the uniqueness of the limit we observe that $\phi_0 = F_0(f)$. Next, let us fix $n \in \mathbb{N}$ and assume that $\phi_k = F_k(f)$ for every $k$ such that $F_k$ is of rank not greater than $n$. We will show that it implies that if $F_l$ is of rank $n+1$, then $\phi_l = F_l(f)$. Let us take arbitrary $F_l$ of rank $n+1$, then one of two scenarios is satisfied. Either there exists $F_a$ of rank $m < n+1$ and $F_b$ of rank $s<n+1$ with  $m+s = n+1$, such that
\[
F_l(x_1,\dots x_{n+1},\xi) = F_{a}(x_1,\dots x_m,\xi)\cdot F_{b}(x_{m+1},\dots x_{n+1},\xi)
\]
or  there exists $F_c$ of rank $n$ such that
\[
F_l(f)(x_1,\dots x_{n+1},\xi) = \izj F_{c}(f)(x_1,\dots x_{n},\eta)f(x_{n+1},\xi,d \eta).
\]
In the first case the induction hypothesis leads to
\begin{align*}
\phi_l(x_1,\dots,x_{n+1},\xi) &= \lim_{N\rightarrow \infty}F_l(f^N)(x_1,\dots,x_{n+1},\Phi_N(\xi))  \\
&= \lim_{N\rightarrow \infty}F_a(f^N)(x_1,\dots,x_{m},\Phi_N(\xi))\lim_{N\rightarrow \infty}F_b(f^N)(x_{m+1},\dots,x_{n+1},\Phi_N(\xi))\\
& = F_a(f)(x_1,\dots,x_{m},\xi)F_b(f)(x_{m+1},\dots,x_{n+1},\xi)=F_l(f)(x_1,\dots x_{n+1},\xi).
\end{align*}
In the second case we take $\vf \in C_c([0,1]\times \R^{d(n+1)})$ and note that
\begin{align*}
 & \izj \int_{\R^{d(n+1)}}  F_l(f^N)(x_1,\dots,x_{n+1},\Phi_N(\xi))\vf(\xi,x_1,\dots,x_{n+1})   dx_1 \dots dx_{n+1} d\xi \\
 =& \izj \izj \int_{\R^{dn}}F_c(f^N)(x_1,\dots,x_{n},\eta)\vf(\xi,x_1,\dots,x_{n+1})  dx_1\dots dx_{n} \ird f^N(x_{n+1},\Phi_N(\xi),d\eta)  dx_{n+1} d\xi \\
 =& \izj \izj  \int_{\R^{dn}}F_c(f^N)(x_1,\dots,x_{n},\Phi_N(\eta))\vf(\xi,x_1,\dots,x_{n+1})  dx_1\dots dx_{n} \left[\ird (\Phi_N^*f^N)(x_{n+1},\xi,\cdot)  dx_{n+1}\right](d\eta)d\xi.
\end{align*}
Since from the induction hypothesis
\[
\int_{\R^{dn}}F_c(f^N)(x_1,\dots,x_{n},\Phi_N(\eta))\vf(\xi,x_1,\dots,x_{n+1})  dx_1\dots dx_{n} \rightarrow  
\]
\[
\int_{\R^{dn}}F_c(f)(x_1,\dots,x_{n},\eta)\vf(\xi,x_1,\dots,x_{n+1})  dx_1\dots dx_{n} \m{ in } L^1_\eta C_\xi(L^{1}(\R^d)) 
\]
and $\Phi_N^*f^N \overset{*}{\rightharpoonup} f$ in $\W((W^{1,1}\cap W^{1,\infty})(\R^d))$ we obtain
\[
\izj \int_{\R^{d(n+1)}} F_l(f^N)(x_1,\dots,x_{n+1},\Phi_N(\xi))\vf(\xi,x_1,\dots,x_{n+1})   dx_1 \dots dx_{n+1} d\xi \rightarrow 
\]
\[
\izj \izj  \int_{\R^{dn}}F_c(f)(x_1,\dots,x_{n},\eta)\vf(\xi,x_1,\dots,x_{n+1})  dx_1\dots dx_{n}  \ird  f(x_{n+1},\xi,\cdot) dx_{n+1}(d\eta) d\xi
\]
and thus $F_l(f^N(\cdot,\Phi_N(\cdot)))\overset{*}{\rightharpoonup} F_l(f)$ in $\mathcal{M}([0,1]\times \R^{dn})$, hence from the uniqueness of the limit $\phi_l = F_l(f)$. By the mathematical induction principle $\phi_k = F_k(f)$ for every $k \in \mathbb{N}$. \\ 
From the convergence $\Phi_N^{*}f^N \overset{*}{\rightharpoonup} f$ in $\W((W^{1,1}\cap W^{1,\infty})(\R^d))$ we see that if $f^N \geq 0$, then also $f \geq 0$ by weak$^*$ closedness of nonnegative cone in $\W((W^{1,1}\cap W^{1,\infty})(\R^d))$. 
We finish the proof noticing that since we have shown that  every transform $F \in \mathcal{F}$ of rank $n$ on appropriately selected subsequence satisfies $F(f^N(\cdot,\Phi_N(\cdot))) \rightarrow F(f) $ in $L^{p}_{loc}([0,1]\times \R^{dn})$, then on the same subsequence we have for every $p \in [1,\infty)$
\[
\tau(T,f^N)(x_1,\dots,x_n) = \izj F(f^{N})(x_1,\dots,x_n,\eta)d\eta = \izj F(f^{N})(x_1,\dots,x_n,\Phi_N(\eta))d\eta \rightarrow \tau(T,f)(x_1,\dots,x_n) 
\]
 in  $L^{p}_{loc}(\R^{dn})$. In this way we proved the claim.  
\end{proof}

We finish this section with the crucial result, which states that if $f$ is a solution to (\ref{meannu}), then the family of observables $\tau(T,f)$ solves the Vlasov hierarchy. In order to rigorously prove Proposition \ref{hier} we argue by approximation. To this end, we  establish the following approximation lemma, which is of independent interest. The proof of the lemma may be found in the Appendix.

\begin{lemma}\label{taustar}
Let $f \in \W((L^{1}\cap L^{\infty})(\R^d))$, $T \in \Tree$ and $f^\ve$ be the approximation from Lemma \ref{hmj}. For every $m \in \{1,\dots,|T|\}$ we may represent $\tau(T,f^{\ve})$ as follows
\[
\tau(T,f^{\ve})(x_1,\dots,x_n) = \izj \tau_{*}^m(T,f^{\ve})(x_1,\dots,x_n,\xi_m)d\xi_m,
\]
\[
\tau_{*}^m(T,f^{\ve})(x_1,\dots,x_n,\xi_m):=\int_{[0,1]^{|T|-1}}\prod_{(k,l)\in E(T)}f^{\ve}(x_{l-1},\xi_k,\xi_l) \prod_{i\neq m} d\xi_i.
\]
Then as $\ve \rightarrow 0$
\[
\tau^m_{*}(T,f^{\ve}) \rightarrow \tau_{*}^m(T,f) \in L^{1}((0,1)\times \R^{dn}) \m{ and } \norm{\tau^m_{*}(T,f^\ve)}_{L^{\infty}((0,1)\times \R^{dn})} \leq \norm{f}_{\W(L^{\infty}(\R^d))}^n. 
\]
\end{lemma}
Note that this result is more involved than Lemma \ref{apro}. From Lemma \ref{ftau} we know that there exists a transform $F$ such that one may represent the observable $\tau$ as $\izj F^\sigma(f)(\xi)d\xi$, but the proof of Lemma \ref{ftau} reveals that the variable $\xi$ is the one connected with the root of a tree $T$. In general it is not true that each $\tau_*^m$ may be represented by some transform $F(f)$, thus we cannot apply Lemma \ref{apro} directly.

\begin{prop}\label{hier}
Let us assume that $K_1, K_2\in L^1(\R^d)$ and $\nu\geq 0$. Let  $f\in L^{\infty}((0,t_{*});\W((L^{1}\cap L^{\infty})(\R^d)))$ be a weak solution to \eqref{meannu}. Then,  the family $\tau(T,f)$ solves the following  non-exchangeable Vlasov hierarchy in the sense of distributions
\[
\partial_t \tau(T,f)(t,x_1,\dots,x_n)+\sum_{i=1}^{n} \divv_{x_i}\left(\int_{\R^d} K_1(x_i-y)\,\tau(T+(i+1),f)(t,x_1,\ldots,x_{n},y) dy\right)
\]
\eqq{
=An\tau(T,f) - \sum_{i=1}^{n}\ird K_2(x_i-y)\tau(T+(i+1),f)(t,x_1,\ldots,x_{n},y)dy +\nu\sum_{i=1}^{n}\Delta_{x_i}\tau(T,f).
}{treeeq}
for every $T\in \Tree$ with $n=|T|-1$.
\end{prop}
\begin{proof}
At first note that, by Corollary \ref{tauest}, if $f\in L^{\infty}((0,t_{*});\W((L^{1}\cap L^{\infty})(\R^d)))$ then $\tau(T,f) \in L^{\infty}((0,t_{*});(L^{1}\cap L^{\infty}(\R^{dn})))$ for every $n \in \mathbb{N}$ and thus the expressions in (\ref{treeeq}) are well defined in the sense of distributions. The formal argument using the representation of observables for more regular $f$ (Definition~\ref{tau}) is straightforward. The most involved part is to prove rigorously the result assuming $f$ belongs only to $\W$ with respect to variables $\xi$ and $\eta$. Hence, for clarity, below we calculate time and space derivatives pointwisely, however we understand them in distributional sense. At first, we extend $f$ periodically in $\xi$ and convolve (\ref{meannu}) with a smoothing periodic kernel in $\xi$ variable. Then the resulting $f^\ve(t,x,\xi,\eta):= \izj \rho(\xi-\xi')f(t,x,d\xi',\eta)$ satisfies
\[
\partial_t f^\ve(t,x,\xi,\eta) + \divv ( f^\ve(t,x,\xi,\eta) V_1[f](t,x,\eta)) =  f^\ve(t,x,\xi,\eta) (A -V_2[f](t,x,\eta))+\nu \Delta f^\ve(t,x,\xi,\eta).
\]

Since $f^\ve \in L^{\infty}((0,t_{*})\times (0,1)^2;(L^1\cap L^{\infty})(\R^d))$ we may use representation (\ref{reptau}) of the observable $\tau$. 
For any $T \in \Tree $ we differentiate in time $\tau(T,f)$, to the result
\begin{align*}
   &\partial_t \tau(T,f^\ve)(t,x_1,\dots,x_n) = \partial_t \int_{[0,1]^{|T|}} \prod_{(k,m)\in E(T)} f^\ve(t,x_{m-1},\xi_k,\xi_m)d\xi_1\dots d\xi_{|T|} \\
   & = \int_{[0,1]^{|T|}}\sum_{_{(k,m)\in E(T)}} \partial_t f^\ve(t,x_{m-1},\xi_k,\xi_{m})\prod _{(l,s)\neq (k,m), (l,s) \in E(T)} f^\ve(t,x_{s-1},\xi_l,\xi_{s}) d\xi_1,\dots d \xi_{|T|}\\
   & =  -\sum_{_{(k,m)\in E(T)}}\int_{[0,1]^{|T|}}\divv_{x_{m-1}} \left( f^\ve(t,x_{m-1},\xi_k,\xi_{m}) \ird K_{1}(x_{m-1}-y)\izj f(t,y,\xi_{m},d\zeta)  dy\right)\\
   &\times\prod _{(l,s)\neq (k,m), (l,s) \in E(T)} f^\ve(t,x_{s-1},\xi_l,\xi_{s}) d\xi_1,\dots d \xi_{|T|} \\
   &+ \sum_{_{(k,m)\in E(T)}}\int_{[0,1]^{|T|}}\left(A f^\ve(t,x_{m-1},\xi_k,\xi_{m})+\nu \Delta_{x_{m-1}} f^\ve(t,x_{m-1},\xi_k,\xi_{m}) \right)\hspace{-0.5cm}\prod_{(l,s)\neq (k,m), (l,s) \in E(T)} \hspace{-1cm}f^\ve(t,x_{s-1},\xi_l,\xi_{s}) d\xi_1,\dots d \xi_{|T|} \\
   &-\sum_{_{(k,m)\in E(T)}}\int_{[0,1]^{|T|}} f^\ve(t,x_{m-1},\xi_k,\xi_{m}) \ird K_{2}(x_{m-1}-y)\izj f(t,y,\xi_{m},d\zeta)  dy\\
   &\times \prod _{(l,s)\neq (k,m), (l,s) \in E(T)} f^\ve(t,x_{s-1},\xi_l,\xi_{s}) d\xi_1,\dots d \xi_{|T|}.
\end{align*}
At first we see that recalling Definition \ref{deftf} and applying Lemma \ref{apro}
\begin{align*}
&\sum_{_{(k,m)\in E(T)}}\int_{[0,1]^{|T|}}\left(A f^\ve(t,x_{m-1},\xi_k,\xi_{m})+\nu \Delta_{x_{m-1}} f^\ve(t,x_{m-1},\xi_k,\xi_{m}) \right)\hspace{-0.5cm}\prod_{(l,s)\neq (k,m), (l,s) \in E(T)} \hspace{-1cm}f^\ve(t,x_{s-1},\xi_l,\xi_{s}) d\xi_1,\dots d \xi_{|T|}\\
&= An\tau(T,f^\ve) + \nu \sum_{i=1}^n\Delta_{x_i} \tau(T,f^\ve)\overset{\ve\rightarrow 0}{\longrightarrow} An\tau(T,f) + \nu \sum_{i=1}^n\Delta_{x_i} \tau(T,f)
\end{align*}
in the sense of distributions.
In order to pass to the limit with remaining two terms, we note that
\begin{align*}
 &\sum_{_{(k,m)\in E(T)}}\int_{[0,1]^{|T|}}\divv_{x_{m-1}} \left( f^\ve(t,x_{m-1},\xi_k,\xi_{m}) \ird K_{1}(x_{m-1}-y)\izj f(t,y,\xi_{m},d\zeta)  dy\right) \\
&\prod _{(l,s)\neq (k,m), (l,s) \in E(T)} f^\ve(t,x_{s-1},\xi_l,\xi_{s}) d\xi_1,\dots d \xi_{|T|}\\
&=\sum_{_{(k,m)\in E(T)}}\divv_{x_{m-1}}\ird K_{1}(x_{m-1}-y)\int_{[0,1]^{|T|+1}}\prod _{ (l,s)\in E(T)}f^\ve(t,x_{s-1},\xi_l,\xi_{s})f(t,y,\xi_{m},d\zeta)d\xi_1,\dots d \xi_{|T|}  dy=:(*)^\ve
\end{align*}
and 
\[
\int_{[0,1]^{|T|}}\hspace{-0.2cm}\prod _{ (l,s)\in E(T)}f^\ve(t,x_{s-1},\xi_l,\xi_{s})\izj f(t,y,\xi_{m},d\zeta) d\xi_1,\dots d \xi_{|T|} 
=\izj \tau_{*}^m(T,f^{\ve})(x_1,\dots,x_n,\xi_m) \izj f(t,y,\xi_{m},d\zeta) d\xi_m.
\]
Since by Lemma \ref{taustar} $\tau_{*}^m(T,f^{\ve}) \rightarrow \tau_{*}^m(T,f)$ in $L^{1}((0,1)\times \R^{dn})$ we may pass to the limit with $\ve$ resulting in 
\[
(*)^\ve \overset{\ve\rightarrow 0}{\longrightarrow }\sum_{_{(k,m)\in E(T)}}\divv_{x_{m-1}}\ird K_{1}(x_{m-1}-y) \tau(T+m,f)(t,x_1,\dots,x_n,y)dy
\]
in the sense of distributions. Analogously we pass to the limit in the last term. Applying Lemma \ref{apro} we note that also $\partial_t\tau(T,f^\ve)\rightarrow \partial_t\tau(T,f)$ in a sense of distributions, hence we arrive at
\begin{align*}
    \partial_t \tau(T,f)(t,x_1,\dots,x_n)=&-\sum_{i=1}^n \divv_{x_i}\ird K_{1}(x_{i}-y) \tau(T+(i+1),f)(t,x_1,\dots,x_n,y)dy\\
    &+ An\tau(T,f)(t,x_1,\dots,x_n) + \nu \sum_{i=1}^n\Delta_{x_i} \tau(T,f)(t,x_1,\dots,x_n)\\
    &-\sum_{i=1}^n \ird K_{2}(x_{i}-y) \tau(T+(i+1),f)(t,x_1,\dots,x_n,y)dy,
\end{align*}
which finishes the proof.
\end{proof}

\subsection{Stability estimates}

In this section we prove quantitative stability estimates for weak solutions  to (\ref{meannu}) with $\nu = 0$, with respect to the initial condition. The approach adopted here follows the path introduced in \cite{JPS2025} for conservative setting and then applied to non-conservative setting in \cite{nasza}. At first, we formulate the stability result in terms of generic distribution solutions  $h=(h_T)_{T\in \Tree}$  to the  hierarchy of  non-exchangeable Vlasov equations with $\nu>0$
\begin{align}\label{hiergeneq}
 &\partial_t h_T+\sum_{i=1}^{n} \mbox{div}_{x_i}\left(\int_{\R^d} K_1(x_i-y)\,h_{T+(i+1)}(t,x_1,\ldots,x_{n},y)\,dy\right)\\ \nonumber 
 =&Anh_T -\sum_{i=1}^{n} \int_{\R^d} K_2(x_i-y)\,h_{T+(i+1)}(t,x_1,\ldots,x_{n},y)\,dy +  \nu\,\sum_{i=1}^{n} \Delta_{x_i} h_T,
\end{align}
where $n=|T|-1$. Then, by linearity of (\ref{hiergeneq}) we apply this result to the difference $\tau(T,f^\nu)-\tau(T,\tilde{f}^\nu)$, $T \in \mathcal{T}$, where $f^\nu,\tilde{f}^\nu$ are two solutions to (\ref{meannu}) with $\nu > 0 $ and different initial conditions.  Finally, we compare weak solution $f$ to (\ref{meanmean}) with $f^\nu$ and in this way, after appropriate optimization in $\nu$, we arrive at our main stability result Proposition~\ref{mainstab}. 

Let us introduce the norm of the hierarchy in which our stability result holds. 

\begin{defi}\label{defnorml}[see \cite[Definition 4.18]{JPS2025}]
Let $h_T\in L^2(\R^{d n})$ for every $T\in \Tree$ with $n=|T|-1$ and  set $h=(h_T)_{T\in \Tree}$. Then,  we define the family of norms:
\[
 \|h\|_{\lambda}=\sup_{T\in \Tree} \lambda^{n/2}\,\|h_T\|_{L^2(\R^{dn})}, \hd \lambda>0.
\]
\end{defi}
Note that in our application $h_T=\tau(T,f)$ and if $f\in \W(L^2(\R^d))$ applying Corollary~\ref{tauest} we obtain
\[
\Vert h\Vert_\lambda\leq \sup_{T\in \Tree} \lambda^{n/2}\Vert f\Vert_{\W(L^2(\R^d))}^{n},
\]
for every $\lambda>0$. Thus,  choosing $\lambda  < \Vert f\Vert_{\W(L^2(\R^d))}^{-2}$, we get $\norm{h}_{\lambda} < \infty$. \\
%
%%%%
In comparison to \cite{JPS2025} and \cite{nasza} our hierarchy of Vlasov  equations differs in the way that each $h_T$ has $n=|T|-1$ spatial variables and in terms involving $K_1,K_2$, we never add the leaf to the root (vertex indexed by $1$). However, since it does not influence the proof of \cite[Lemma 8]{nasza} (see also \cite[Theorem 4.19]{JPS2025}) we may state what follows.
\begin{prop} \label{hlambdabound}\cite[Lemma 8]{nasza}
Let $h=(h_T)_{T\in \Tree}$ be a distributional solution to \eqref{hiergeneq} with  $\nu>0$, $K_1, K_2\in L^2(\R^d)$ and $K_2 \geq 0$. Additionally, assume that $h_T\in L^\infty((0,t_{*});(L^1\cap L^2)(\R^{d |T|-1}))$ for any $T\in \Tree$ and that there exists $\lambda>0$ such that $C_\lambda:=\sup_{t\in (0, t_*)}\|h(t,\cdot)\|_{\lambda}<\infty$. Then, for every $p>1$ and every $\theta\in (0,2^{-p'}e^{-p'(2A+1)t_*})$ with $p'=\frac{p}{p-1}$, there exists a constant $C > 0$ dependent only on $p,A,t_*,\theta$ such that 
\[
\Vert h(t,\cdot)\Vert_{\theta\lambda}\leq \max\{C_{\lambda},1\}C_\lambda C\exp\left(p^{-\frac{\left((2\nu)^{-1}\norm{K_1}_{L^{2}(\R^d)}^2  + \norm{K_2}_{L^{2}(\R^d)}^2\right)}{\theta\lambda}t}\,\log\frac{\Vert h(0,\cdot)\Vert_{\theta \lambda}}{C_\lambda}\right),
\]
for every $t\in [0, t_*]$.
\end{prop}

Next, we provide the analogue to \cite[Lemma 9]{nasza} and \cite[Lemma 4.22]{JPS2025} in our setting.

\begin{prop}\label{linftyestf}
Assume that $f$ is a nonnegative weak solution to \eqref{meannu} with $\nu \geq 0$ with $K_1\in (W^{1,1}\cap  W^{1,\infty})(\R^d)$, $K_2 \in L^{1}(\R^d)$, $K_2 \geq 0$,   and  $f^0, \in \W( (L^1\cap L^\infty\cap H^1)(\R^d))$, $f \in L^{\infty}((0,t_{*});\W( (L^1\cap L^\infty\cap H^1)(\R^d)))$.  Then,  it satisfies
  \[
 \|f(t,\cdot,\cdot,\cdot)\|_{\W((L^1\cap L^\infty\cap H^1)(\R^d))}\leq C \exp(\exp(\tilde{C}e^{At}t)),
  \]
for some $C,\tilde{C}\in \mathbb{R}_+$ depending only on $\|f^0\|_{\W((L^1\cap L^\infty\cap H^1)(\R^d))}$, $\|K_1\|_{W^{1,1}(\R^d)}$, $\|\divv K_1\|_{L^\infty}$, $\norm{K_2}_{L^{1}(\R^d)}$ and $A$.
\end{prop}
\begin{proof}

In order to carry the estimates for almost all $\xi,\eta$ we firstly convolve (\ref{meannu}) with periodic mollifying kernel in $\xi$ variable. Then we denote $f^\ve(t,x,\xi,\eta):=\izj \vf^{\ve}(\xi-\xi')f(t,x,d\xi',\eta)$, where we extended $f$ periodically in $\xi$.
 Let us test the equation  by $p(f^\ve)^{p-1}$ for fixed $p>1$. We note that this operation is allowed since the boundedness of $f^\ve$ in space ensures that $\nabla (f^\ve)^p \in L^2(\R^d)$.
\begin{align*}
    &\frac{d}{dt} \|f^\ve(t,\cdot,\xi,\eta)\|_{L^p(\R^d)}^p +\nu p(p-1) \ird \abs{\nabla f^\ve(t,x,\xi,\eta)}^{2}(f^\ve)^{p-2}(t,x,\xi,\eta)dx\\
    & +p(p-1)\ird (f^\ve)^{p-1}(t,x,\xi,\eta)\nabla f^\ve(t,x,\xi,\eta)\cdot\izj \ird K_{1}(x-y)f(t,y,\eta,d\zeta)dy dx \\
    & = p\ird (f^\ve)^{p}(t,x,\xi,\eta)\left[A - \izj \ird  K_{2}(x-y)f(t,y,\eta,d\zeta)dy  \right]dx.
\end{align*}
Integrating by parts gives
\begin{align*}
    & p(p-1)\ird (f^\ve)^{p-1}(t,x,\xi,\eta)\nabla f^\ve(t,x,\xi,\eta)\cdot\izj \ird K_{1}(x-y)f(t,y,\eta,d\zeta)dy dx\\
    & =-(p-1)\ird (f^\ve)^{p}(t,x,\xi,\eta)\izj \ird \divv K_{1}(x-y)f(t,y,\eta,d\zeta)dy dx\\
    & \leq (p-1)\norm{\divv K_1}_{L^{\infty}(\R^d)}\norm{f(t,\cdot,\cdot,\cdot)}_{\W(\mathcal{M}(\R^d))}\ird (f^\ve)^{p}(t,x,\xi,\eta)dx.
\end{align*}

Hence, recalling that $K_2, f \geq 0$ we arrive at
\[
\frac{d}{dt} \|f^\ve(t,\cdot,\xi,\eta)\|_{L^p(\R^d)}^p 
\leq \left((p-1)\,\|\divv K_1\|_{L^\infty(\R^d)}\,\| f(t,\cdot,\cdot,\cdot)\|_{\W(\mathcal{M}(\R^d))}\, + Ap\right)\|f^\ve(t,\cdot,\xi,\eta)\|_{L^p(\R^d)}^p,
\]
for a.a. $t\in (0,\ t_*)$ and $\xi\in (0,\ 1)$.
On the other hand, since, $f$ is a nonnegative weak solution and $K_2 \geq 0$, integrating (\ref{meannu}) in space,  we may easily estimate
\eqq{
\norm{f(t,\cdot,\cdot,\cdot)}_{\W(\mathcal{M}(\R^d))} \leq e^{At} \norm{f^0}_{\W(\mathcal{M}(\R^d))}, \hd \norm{f}_{L^{\infty}((0,t_*);\W(\mathcal{M}(\R^d))} \leq e^{At_*} \norm{f^0}_{\W(\mathcal{M}(\R^d))}.
}{flinfty}
Hence, by Gr\"onwall inequality
\eqq{
\|f^\ve(t,\cdot,\xi,\eta)\|_{L^p(\R^d)}^p \leq  \|f^\ve(0,\cdot,\xi,\eta)\|_{L^p(\R^d)}^p e^{[(p-1)\|\divv K_1\|_{L^\infty(\R^d)}e^{At}\norm{f^0}_{\W(\mathcal{M}(\R^d))} +Ap]t}.
}{flinftytwo}
Rising the inequality to the power $1/p$ and passing to the limit with $p$ we arrive at
\[
\norm{f^\ve(t,\cdot,\xi,\eta)}_{L^\infty(\R^d))} \leq \norm{f^\ve(0,\cdot,\xi,\eta)}_{L^\infty(\R^d)} e^{Ce^{At}t}.
\]
Applying the norm in $\W$, using the classical properties of convolution with mollifying kernel and the lower semicontinuity of the norm we obtain
\eqq{
 \norm{f(t,\cdot,\cdot,\cdot)}_{\W(L^\infty(\R^d))} \leq e^{Ce^{At}t}\norm{f^0}_{\W(L^\infty(\R^d))}
}{finftyinfty}
with a positive constant $C$ dependent only on $A, \|\divv K_1\|_{L^\infty(\R^d)},\norm{f^0}_{\W(\mathcal{M}(\R^d))}$. 
Analogously, (\ref{flinftytwo}) leads to
\eqq{
\|f(t,\cdot,\cdot,\cdot)\|_{\W(L^p(\R^d))} \leq  \|f^0\|_{\W(L^p(\R^d))}e^{[\frac{(p-1)}{p}\|\divv K_1\|_{L^\infty(\R^d)}e^{At}\norm{f^0}_{\W(\mathcal{M}(\R^d))} +A]t} \hd \m{ for } \hd p \in [1,\infty).
}{flinftytwof}
Now we test the equation by $\Delta f^\ve$. To justify it we mollify the equation in space by standard smoothing convolution and apply Friedrichs commutator lemma (Lemma II.1 \cite{diPerna}), to $\nabla f^\ve$ which is possible since $f^\ve \in L^{\infty}((0,t_{*})\times (0,1)^2;H^{1}(\R^d))$ and $K_1 \in W^{1,\infty}(\R^d;\R^d)$ . To keep the exposition concise, we proceed directly with the calculations
\begin{align}\label{gradba}
    &\frac{d}{dt} \frac{1}{2}\|\nabla f^\ve(t,\cdot,\xi,\eta)\|_{L^2(\R^d)}^2=\frac{1}{2}\ird \divv V_1[f] \abs{\nabla f^\ve}^{2} dx-\ird \nabla^T f^\ve\cdot \nabla V_1[f] \cdot \nabla f^\ve dx - \ird \nabla f^\ve \cdot \nabla \divv V_1[f]  f^\ve dx\\ \nonumber
    &+\ird \abs{\nabla{f^\ve}}^2[A-V_2[f]] dx - \ird f^\ve \nabla V_2[f] \cdot \nabla f^\ve dx -\nu \ird \abs{\Delta  f^\ve}^2dx =:
\sum_{j=1}^6 I_j(t,\xi,\eta).
\end{align}
Clearly, $I_6 \leq 0$ and since $V_2[f] \geq 0$
\[
I_4(t,\xi,\eta) \leq A\norm{\nabla f^\ve(t,\cdot,\xi,\eta)}^2_{L^{2}(\R^d)}.
\]
With the rest of the terms we proceed as follows. 
\begin{align*}
I_1(t,\xi,\eta)&=\frac{1}{2}\ird |\nabla f^\ve(t,x,\xi,\eta)|^2\int_0^1\int_{\R^{d}}\divv K_1(x-y)\,\,f(t,y,\eta,d\zeta)dy\,dx \\
I_2(t,\xi,\eta)&=- \ird\nabla f^\ve(t,x,\xi,\eta)^\top\cdot  \int_0^1 \ird \nabla K_1(x-y)\, f(t,y,\eta,d\zeta) \,dy\,  \cdot \nabla f^\ve(t,x,\xi,\eta) dx.\\
I_3(t,\xi,\eta)&=- \int_{\R^{d}} f^\ve(t,x,\xi,\eta)\nabla f^\ve(t,x,\xi,\eta) \cdot\izj \ird \divv K_1(x-y)\,\nabla_y f(t,y,\eta,d\zeta)\,\, dy\,dx,\\
I_5(t,\xi,\eta)&=- \int_{\R^{d}} f^\ve(t,x,\xi,\eta)\nabla f^\ve(t,x,\xi,\eta) \cdot  \izj \ird K_2(x-y)\, \nabla_y f(t,y,\eta,d \zeta)\,\, dy dx,
\end{align*}
where in $I_3$ and $I_5$ we made use of the properties of convolution.
We may easily estimate the first to components
\begin{align*}
 &\abs{I_1(t,\xi,\eta)} \leq 
\frac{1}{2}\Vert \divv K_1\Vert_{L^\infty(\R^d)}\norm{\nabla f^\ve(t,\cdot,\xi,\eta)}_{L^{2}(\R^d)}^2 \norm{f(t,\cdot,\eta,\cdot)}_{\mathcal{M}_{\zeta}(\mathcal{M}(\R^d))},\\
&\abs{I_2(t,\xi,\eta)} \leq \Vert \nabla K_1\Vert_{L^1(\R^d)}\norm{\nabla f^\ve(t,\cdot,\xi,\eta)}_{L^{2}(\R^d)}^2 \norm{f(t,\cdot,\eta,\cdot)}_{\mathcal{M}_{\zeta}(L^{\infty}(\R^d))}.
\end{align*}

%\[
%\izj \izj \abs{I_1(t,\xi,\eta)} d\eta d\xi \leq \frac{1}{2}\Vert \divv K_1\Vert_{L^\infty(\R^d)}\norm{f(t,\cdot,\cdot,\cdot)}_{\W(\mathcal{M}(\R^d))} \izj \izj \norm{\nabla f^\ve(t,\cdot,\xi,\cdot)}_{L^{2}(\R^d)}^2(d\eta)d\xi,
%\]

%\[
%\izj \izj \abs{I_2(t,\xi,\eta)} d\eta d\xi \leq \Vert \nabla K_1\Vert_{L^\infty(\R^d)}\norm{f(t,\cdot,\cdot,\cdot)}_{\W(\mathcal{M}(\R^d))} \izj \izj \norm{\nabla f^\ve(t,\cdot,\xi,\cdot)}_{L^{2}(\R^d)}^2(d\eta)d\xi.
%\]
While for $I_3$ and $I_5$ we apply H\"older inequality together with Young inequality for convolution
\begin{align*}
\abs{I_3(t,\xi,\eta)}& \leq
\norm{f^\ve(t,\cdot,\xi,\eta)}_{L^{\infty}(\R^d)}\norm{\nabla f^\ve(t,\cdot,\xi,\eta)}_{L^{2}(\R^d)}\norm{\divv K_1}_{L^{1}(\R^d)}\izj \norm{\nabla f(t,\cdot,\eta,\cdot)}_{L^{2}(\R^d)}(d\zeta), \\
\abs{I_5(t,\xi,\eta)}& \leq
\norm{f^\ve(t,\cdot,\xi,\eta)}_{L^{\infty}(\R^d)}\norm{\nabla f^\ve(t,\cdot,\xi,\eta)}_{L^{2}(\R^d)}\norm{ K_2}_{L^{1}(\R^d)}\izj \norm{\nabla f(t,\cdot,\eta,\cdot)}_{L^{2}(\R^d)}(d\zeta).
\end{align*}
Inserting the estimates above in (\ref{gradba}) and dividing by one norm we obtain
\begin{align*}
  & \frac{d}{dt} \|\nabla f^\ve(t,\cdot,\xi,\eta)\|_{L^2(\R^d)} \leq \frac{1}{2} \Vert \divv K_1\Vert_{L^\infty(\R^d)}\norm{\nabla f^\ve(t,\cdot,\xi,\eta)}_{L^{2}(\R^d)}\norm{f(t,\cdot,\cdot,\cdot)}_{L^{\infty}_\xi(\mathcal{M}_{\eta}(\mathcal{M}(\R^d)))} \\
  &+\Vert \nabla K_1\Vert_{L^1(\R^d)} \norm{\nabla f^\ve(t,\cdot,\xi,\eta)}_{L^{2}(\R^d)}\norm{f(t,\cdot,\cdot,\cdot)}_{L^{\infty}_{\xi}(\mathcal{M}_{\eta}(L^{\infty}(\R^d)))}\\
  &+A\|\nabla f^\ve(t,\cdot,\xi,\eta)\|_{L^2(\R^d)} +\norm{f^\ve(t,\cdot,\xi,\eta)}_{L^{\infty}(\R^d)}\norm{\nabla f(t,\cdot,\cdot,\cdot)}_{L^{\infty}_\xi(\mathcal{M}_{\eta}(L^2(\R^d)))}\left(\norm{\divv K_1}_{L^{1}(\R^d)} + \norm{ K_2}_{L^{1}(\R^d)}\right).
\end{align*}
Integrating with respect to time and applying the $\W$ norm we arrive at
\begin{align*}
  &\|\nabla f^\ve(t,\cdot,\cdot,\cdot)\|_{\W (L^{2}(\R^d))} \leq \|\nabla (f^0)^\ve(\cdot,\cdot,\cdot)\|_{\W (L^{2}(\R^d))}+ \int_0^t \norm{\nabla f^\ve(s,\cdot,\cdot,\cdot)}_{\W(L^{2}(\R^d))}ds  \\
  &  \times \left(\frac{1}{2}\Vert \divv K_1\Vert_{L^\infty(\R^d)}\norm{f}_{L^{\infty}((0,t);\W(\mathcal{M}(\R^d)))}
 +\Vert \nabla K_1\Vert_{L^1(\R^d)}\norm{f}_{L^{\infty}((0,t);\W(L^{\infty}(\R^d)))}  +A\right)\\
 & +\int_0^t\norm{\nabla f(s,\cdot,\cdot,\cdot)}_{\W(L^2(\R^d))}ds
\norm{f^\ve}_{L^{\infty}((0,t);\W(L^{\infty}(\R^d)))}\left(\norm{\divv K_1}_{L^{1}(\R^d)} + \norm{ K_2}_{L^{1}(\R^d)}\right).
\end{align*}
Using the properties of mollifier, weak-lower semicontinuity of the norm and estimates (\ref{flinfty}), (\ref{finftyinfty}) we obtain
\[
 \|\nabla f(t,\cdot,\cdot,\cdot)\|_{\W (L^{2}(\R^d))} \leq \|\nabla f^0(\cdot,\cdot,\cdot)\|_{\W (L^{2}(\R^d))}+ \int_0^t \norm{\nabla f(s,\cdot,\cdot,\cdot)}_{\W(L^{2}(\R^d))}ds  
 \]
\[ \cdot\left(A+\frac{\Vert \divv K_1\Vert_{L^\infty(\R^d)}}{2}e^{At}\norm{f^0}_{\W(\mathcal{M}(\R^d))} + \left[\Vert \nabla K_1\Vert_{L^1(\R^d)}+ \norm{\divv K_1}_{L^{1}(\R^d)} + \norm{ K_2}_{L^{1}(\R^d)}\right]e^{Ce^{At}t}\norm{f^0}_{\W(L^{\infty}(\R^d))}\right).
\]
Applying Gr\"onwall lemma we obtain the desired estimate. 

\end{proof}

Finally, we are in a position to establish the main stability result of this section, which serves as an analogue of \cite[Theorem 4.23]{JPS2025} (see also \cite[Lemma 10]{nasza}).

\begin{prop}\label{mainstab}
Let  $f,\tilde{f} \in L^\infty((0,t_{*});\W(  (L^1\cap L^\infty\cap H^1)(\R^d)))$ be two weak solutions to \eqref{meanmean} with nonnegative initial conditions $f^0$ and $\tilde{f}^0$, respectively. Furthermore, let $K_1\in (W^{1,1} \cap W^{1,\infty})(\R^d)$,  $K_2 \in (L^1\cap L^{2})(\R^d), K_2 \geq 0$. Then, there exists $C,\lambda >0 $ such that for almost all  $t \in (0,t_{*})$ 
  \[
\left\|  \int_0^1 \int_0^1 ( f -\tilde f) (t,\cdot, \xi,d\eta)\,d\xi\,\right\|_{L^2(\R^d)}\leq {\frac{C}{\sqrt{(\log |\log \|\tau(\cdot, f^0)-\tau(\cdot,\tilde f^0)\|_{\lambda}|)_+}}.}
  \]
Here $C,\,\lambda>0$  depend only on $t_*,A, \norm{K_1}_{L^{2}(\R^d)}, \norm{K_1}_{W^{1,1}(\R^d)}$, $\norm{\divv K_1}_{L^{\infty}(\R^d)}, \norm{K_2}_{L^{1}(\R^d)}, \norm{K_2}_{L^{2}(\R^d)}$ and the norm of the initial data $f^0,\,\tilde f^0$ in $\W((L^1\cap L^\infty\cap H^1)(\R^d))$.
 \end{prop}
\begin{proof}
Let us  denote by $f^\nu,\tilde{f}^\nu$ the solutions  to the system \eqref{meannu} with $
\nu  > 0$ and initial data $f^0, \tilde{f}^0$ respectively,  given by Proposition~\ref{existenceprop}.
At first, using Proposition \ref{linftyestf}, we estimate the difference of $f$ and $f^\nu$ in terms of time and viscosity parameter $\nu$.  Then, invoking Proposition~\ref{hlambdabound} we compare $f^\nu$ and $\tilde f^\nu$. In order to obtain the claim, we argue by triangle inequality and choose appropriate $\nu$.

Note that the difference $f^\nu-f$ satisfies the following equation in a weak sense
\begin{align*}
 & \partial_t ( f^\nu- f)+\divv\left( ( f^\nu- f)V_1[f^\nu]\right)
     +\divv(  f V_1[f^{\nu}-f]) -\nu\,\Delta ( f^\nu- f)  \\
 & =\nu\,\Delta  f  + A( f^\nu- f) - ( f^\nu- f)V_2[f^\nu]
   - fV_2[f^\nu-f].    
\end{align*}
In order to justify pointwise almost everywhere estimates w.r.t. $\xi$ and $\eta$ one may mollify the equation in $\xi$ as in the proof of previous lemma. This time, we omit the argument with mollification for the sake of clarity. An integration against $(f^{\nu}-f)$, together with the definitions of $V_1,V_2$, leads to 
\begin{align*}
    &  \frac{d}{dt} \frac{1}{2}\,\| (f^\nu-f)(t, \cdot, \xi,\eta)\|_{L^2(\R^d)}^2 +{\nu}\,\int_{\R^d} |\nabla  (f^\nu-f) (t, x, \xi,\eta)|^2\,dx \\
 &   =\frac{1}{2}\ird\nabla(f^\nu-f)^2(t,x,\xi,\eta)\,\cdot\int_0^1\int_{\R^{d}}  K_1(x-y)\, f^\nu(t,y,\eta, d\zeta) \,dy \,dx\\
&+ \ird\nabla(f^\nu-f)(t,x,\xi,\eta) f(t,x,\xi,\eta)\cdot \int_0^1\int_{\R^{d}}  K_1(x-y)\, (f^\nu-f)(t,y, \eta,d\zeta) \,dy \,dx\\
&  - \nu \ird \nabla f(t,x,\xi,\eta)\cdot\nabla(f^\nu - f)(t,x,\xi,\eta)dx + A \ird (f^\nu - f)^2(t,x,\xi,\eta)dx  \\
&-\ird (f^\nu-f)^2(t,x,\xi,\eta)\,\int_0^1\int_{\R^{d}}  K_2(x-y)\, f^\nu(t,y,\eta, d\zeta) \,dy \,dx\\
&-\ird f(t,x,\xi,\eta)(f^\nu - f)(t,x,\xi,\eta)\,\int_0^1\int_{\R^{d}}  K_2(x-y)\, (f^\nu-f)(t,y,\eta, d\zeta) \,dy \,dx =: \sum_{j=1}^6 I_j.
\end{align*}
We integrate by parts $I_1$ and $I_2$, to remove the gradient from $(f^\nu -f)^2$
\begin{align*}
    I_1 &= -\frac{1}{2}\ird(f^\nu-f)^2(t,x,\xi,\eta)\,\int_0^1\int_{\R^{d}} \divv K_1(x-y)\, f^\nu(t,y,\eta, d\zeta) \,dy \,dx,\\
    I_2 &=  -\ird(f^\nu-f)(t,x,\xi,\eta) \nabla f(t,x,\xi,\eta)\cdot \int_0^1\int_{\R^{d}}  K_1(x-y)\, (f^\nu-f)(t,y,\eta, d\zeta) \,dy \,dx \\
    &-\ird(f^\nu-f)(t,x,\xi,\eta) f(t,x,\xi,\eta)\int_0^1\int_{\R^{d}} \divv K_1(x-y)\, (f^\nu-f)(t,y,\eta, d\zeta) \,dy \,dx.
\end{align*}
Now we estimate term by term, applying  H\"older  inequality and Young inequality for convolution to the result
\begin{align*}
\vert I_1(t,\xi,\eta)\vert&\leq \frac{1}{2}\Vert \divv K_1\Vert_{L^\infty(\R^d)}\,\Vert f^\nu(t,\cdot,\eta,\cdot)\Vert_{\mathcal{M}_{\zeta}(\mathcal{M}(\R^d))}\,\Vert  (f^\nu - f)(t,\cdot,\xi,\eta)\Vert_{L^2(\R^d)}^2,\\
\vert I_2(t,\xi,\eta)\vert &\leq  \Vert  K_1\Vert_{L^2(\R^d)}\,\Vert \nabla f(t,\cdot,\xi,\eta)\Vert_{L^2(\R^d))}\Vert (f^\nu-f)(t,\cdot,\xi,\eta)\Vert_{L^2(\R^d)} \Vert (f^\nu-f)(t,\cdot,\eta,\cdot)\Vert_{\mathcal{M}_{\zeta}(L^2(\R^d))}\\ 
+& \Vert  \divv K_1\Vert_{L^1(\R^d)}\,\Vert f(t,\cdot,\xi,\eta)\Vert_{L^{\infty}(\R^d)}\Vert (f^\nu-f)(t,\cdot,\xi,\eta)\Vert_{L^2(\R^d)}\Vert (f^\nu-f)(t,\cdot,\eta,\cdot)\Vert_{\mathcal{M}_{\zeta}(L^2(\R^d))},\\
\vert I_3(t,\xi,\eta)\vert &\leq \frac{\nu}{2}\norm{\nabla(f^\nu - f)(t,\cdot,\xi,\eta)}^2_{L^{2}(\R^d)} + \frac{\nu}{2}\norm{\nabla f(t,\cdot,\xi,\eta)}_{L^{2}(\R^d)}^2,\\
\vert I_4(t,\xi,\eta)\vert &\leq A \norm{(f^\nu - f)(t,\cdot,\xi,\eta)}^2_{L^{2}(\R^d)},\\
\vert I_5(t,\xi,\eta)\vert &\leq \,\Vert  K_2\Vert_{L^1(\R^d)}\,\Vert f^\nu(t,\cdot,\eta,\cdot)\Vert_{\mathcal{M}_{\zeta}(L^{\infty}(\R^d))}\Vert (f^\nu- f)(t,\cdot,\xi,\eta)\Vert_{L^2(\R^d)}^2,\\
\vert I_6(t,\xi,\eta)\vert &\leq \Vert  K_2\Vert_{L^1(\R^d)}\,\Vert f(t,\cdot,\xi,\eta)\Vert_{L^\infty(\R^d)}\,\Vert (f^\nu- f)(t,\cdot,\xi,\eta)\Vert_{L^2(\R^d)}\norm{(f^\nu - f)(t,\cdot,\eta,\cdot)}_{\mathcal{M}_{\zeta}(L^{2}(\R^d))}.
\end{align*}
Hence, applying Young inequality for products in $I_2$ and $I_6$ we arrive at
\begin{align}\label{newgg}
   & \frac{d}{dt} \,\| (f^\nu-f)(t, \cdot, \xi,\eta)\|_{L^2(\R^d)}^2 \leq \nu \norm{\nabla f(t,\cdot,\xi,\eta)}_{L^{2}(\R^d)}^2\\ \nonumber
    & +\| (f^\nu-f)(t, \cdot, \xi,\eta)\|_{L^2(\R^d)}^2 \left[3+2A+\norm{\divv K_1}_{L^{\infty}(\R^d)}\norm{f^\nu(t,\cdot,\cdot,\cdot)}_{L^{\infty}_\xi(\mathcal{M}_{\eta}(\mathcal{M}(\R^d)))}\right] \\ \nonumber
    &
    +\| (f^\nu-f)(t, \cdot, \xi,\eta)\|_{L^2(\R^d)}^2 \norm{K_{2}}_{L^{1}(\R^d)}\norm{f^\nu(t,\cdot,\cdot,\cdot)}_{L^{\infty}_\xi(\mathcal{M}_{\eta}(L^{\infty}(\R^d)))}\\ \nonumber
    &+\| (f^\nu-f)(t, \cdot, \cdot,\cdot)\|_{L^{\infty}_\xi (\mathcal{M}_{\eta}(L^2(\R^d)))}^2\\ \nonumber
    & \times \left[\norm{\nabla f(t,\cdot,\xi,\eta)}_{L^{2}(\R^d)}\norm{K_1}_{L^{2}(\R^d)} + \norm{f(t,\cdot,\xi,\eta)}_{L^{\infty}(\R^d)}(\norm{\divv K_1}_{L^{1}(\R^d)} + \norm{ K_2}_{L^{1}(\R^d)})\right]^2.
\end{align}
Denote $g^{\nu}(t,\xi,\eta):=\| (f^\nu-f)(t, \cdot, \xi,\eta)\|_{L^2(\R^d)}^2$ and notice that by Jensen inequality
\eqq{
\| (f^\nu-f)(t, \cdot, \cdot,\cdot)\|_{L^{\infty}_\xi (\mathcal{M}_{\eta}(L^2(\R^d)))}^2 \leq \norm{g^{\nu}(t,\cdot,\cdot)}_{L^{\infty}_\xi (\mathcal{M}_{\eta})}.
}{newggq}
Thus, integrating (\ref{newgg}) with respect to time and applying the norm in $\W$, in view of  Proposition \ref{linftyestf} we obtain the estimate of the form
\[
\norm{g^\nu(t,\cdot,\cdot)}_{\W} \leq C_1(t)\nu + C_2(t) \int_0^t \Vert g^\nu(s,\cdot,\cdot)\Vert_{\W}ds, 
\]
where $C_1(t)$ comes from Proposition \ref{linftyestf}, i.e. using the notation from Proposition \ref{linftyestf} we have $C_1(t)=C \exp(\exp(\tilde{C}e^{At}t))$ and 
\begin{align*}
  C_2(t) = &(\norm{\divv K_1}_{L^{\infty}(\R^d)} + \norm{K_{2}}_{L^{1}(\R^d)})C_1(t)  \\
  &+ 2A+3 +C_1^2(t)\left[\norm{K_1}_{L^{2}(\R^d)} + \norm{\divv K_1}_{L^{1}(\R^d)} + \norm{ K_2}_{L^{1}(\R^d)}\right]^2.
\end{align*}
Applying Gr\"onwall lemma we arrive at
 \[
 \| g^\nu(t,\cdot,\cdot)\|_{\W}\leq C(t)\,\nu,
 \] 
 for some continuous and non-decreasing function $C=C(t)\in \R_+$ that only depends on the same norms as $C_1,C_2$ above. This together with (\ref{newggq})
 leads to \begin{equation}\label{aproxst}
 \norm{(f^\nu-f)(t,\cdot,\cdot,\cdot)}_{\W(L^2(\R^d))}\leq C(t)\,\sqrt{\nu}.
  \end{equation} 
 Clearly, we obtain  the same estimate for $\tilde f^\nu -\tilde f$. 

In the second step of the proof, We compare  $f^\nu$ with $\tilde{f}^\nu$ using estimates on generalized Vlasov hierarchy. Observe, that due to  Proposition \ref{hier},  $\tau(T, f^\nu)$ and $\tau(T, \tilde f^\nu)$ with $T\in \Tree$ both solve linear problem \eqref{treeeq}. Hence,  $h=(h_T)_{T\in \Tree}$ defined by $h_T=\tau(T,  f^\nu) - \tau(T,\tilde f^\nu)$ solves \eqref{hiergeneq}. Denote $n=|T|-1$, Corollary~\ref{tauest} yields the estimate
\begin{align*}
\Vert h_T(t,\cdot)\Vert_{L^2(\R^{d n})}&\leq \Vert \tau(T,f^\nu(t,\cdot))\Vert_{L^2(\R^{d n})}+\Vert \tau(T,\tilde{f}^\nu(t,\cdot))\Vert_{L^2(\R^{d n})}\\
&\leq \Vert f^\nu(t,\cdot,\cdot)\Vert_{\W(L^2(\R^d))}^{n}+\Vert \tilde f^\nu(t,\cdot,\cdot)\Vert_{\W(L^2(\R^d))}^{n}.
\end{align*}

Applying the estimate (\ref{flinftytwof}) with $p=2$ and introducing the notation
\[
f^0_{\text{max},2}:=\max\left\{\norm{f^0}_{\W(L^2(\R^d))},\norm{\tilde{f}^0}_{W(L^2(\R^d))} \right\}, \hd \hd f^0_{\text{max},1}:=\max\left\{\norm{f^0}_{\W(\mathcal{M}(\R^d))},\norm{\tilde{f}^0}_{\W(\mathcal{M}(\R^d))}\right \}
\]
we arrive at
\[
\norm{h_T(t,\cdot)}_{L^2(\R^{d n})}\leq 2(f^0_{\max,2})^{n} \exp\left(\left(\frac{\Vert \divv K_1\Vert_{L^\infty(\R^d)}e^{At_*}\,f^0_{\max,1}}{2}+A \right)nt\right).
\]

Recalling Definition \ref{defnorml} we note that $\sup_{t\in [0, t_*]}\|h(t,\cdot)\|_\lambda<1$, for any $\lambda>0$ such that
\[
\sqrt{\lambda}<\exp\left(-\left(\frac{\Vert \divv K_1\Vert_{L^\infty(\R^d)}e^{At_*}\,f^0_{\max,1}}{2}+A \right)t_{*}\right)\frac{1}{2f^0_{\max,2}}.
\]
Take arbitrary $p>1$, $\theta\in (0,2^{-p'}e^{-p'(2A+1)t_*})$ and  $\lambda$ satisfying the estimate above.  Then, Proposition~\ref{hlambdabound} yields 
  \[
\sup_{t\in [0, t_*]}\|h(t,\cdot)\|_{\theta\lambda}\leq C\,\exp\left(p^{-\frac{\left((2\nu)^{-1}\norm{K_1}_{L^{2}(\R^d)}^2  + \norm{K_2}_{L^{2}(\R^d)}^2\right)}{\theta\lambda}t}\log \Vert h(0,\cdot)\Vert_{\theta\lambda}\right).
  \]
Choosing particular $T=T_2$ we arrive at 
  \begin{align*}
   \|h(t,\cdot)\|_{\theta\lambda}&\geq (\theta\lambda)^{\frac{1}{2}} \|h_{T_2}(t,\cdot)\|_{L^2(\R^d)}=\sqrt{\lambda}\sqrt{\theta} \left\|\tau(T_2, f^\nu)(t,\cdot)-\tau(T_2,\tilde f^\nu)(t,\cdot)\right\|_{L^2(\R^d)}   \\
   & =\sqrt{\lambda}\sqrt{\theta}\,\left\| \int_0^1\int_0^1 ( f^\nu -\tilde f^\nu) (t,\cdot, \xi,d\eta)\,d\xi\,\right\|_{L^2(\R^d)}
  \end{align*}
  and consequently for every $t \in (0,t_{*})$
  \[
\left\| \int_0^1 \int_0^1 ( f^\nu -\tilde f^\nu) (t,\cdot, \xi,d\eta)\,d\xi\,\right\|_{L^2(\R^d)} \leq \frac{C}{\sqrt{\lambda}\sqrt{\theta}}\,\exp\left(p^{-\frac{\left((2\nu)^{-1}\norm{K_1}_{L^{2}(\R^d)}^2  + \norm{K_2}_{L^{2}(\R^d)}^2\right)}{\theta\lambda}t}\log \Vert h(0,\cdot)\Vert_{\theta\lambda}\right).\]
Combining the estimate above with (\ref{aproxst}), we obtain by the triangle inequality

\[
\left\| \int_0^1 \int_0^1 ( f -\tilde f) (t,\cdot, \xi,d\eta)\,d\xi\,\right\|_{L^2(\R^d)}  \leq \frac{C}{\sqrt{\lambda}\sqrt{\theta}}\,\exp\left(p^{-\frac{\left((2\nu)^{-1}\norm{K_1}_{L^{2}(\R^d)}^2  + \norm{K_2}_{L^{2}(\R^d)}^2\right)}{\theta\lambda}t}\log \Vert h(0,\cdot)\Vert_{\theta\lambda}\right)+2C(t)\,\sqrt{\nu}
\]
and we arrive at the same bound as in the proof of \cite[Lemma 10]{nasza}. Choosing for simplicity
 $p=e$,
\[
\nu=\frac{\Vert K_1\Vert_{L^2(\R^d)}^2}{\theta\lambda}\,t_*\,(\log |\log \Vert h(0,\cdot)\Vert_{\theta\lambda}|)_+^{-1}
\]
and estimating as in the proof of \cite[Lemma 10]{nasza}, we arrive at the claim. Note that in case $K_1 = 0$, the term $\norm{K_1}^2_{L^2(\R^d)}/2\nu$ disappears in Proposition \ref{hlambdabound} and the estimate still holds with $\nu=0$, since $u^{\kappa} \leq \frac{C(\kappa)}{\sqrt{(\log\abs{\log u})_{+}}}$ for small $u$ and any $\kappa \in (0,1)$. 
%%%%%%%%%%%
\nic{
Then,
\begin{align}\label{fest7}
   \left\| \int_0^1  \int_0^1 ( f -\tilde f) (t,\cdot, \xi,d\eta)\,d\xi\,\right\|_{L^2(\R^d)}&\leq \frac{C_{p,\theta}}{\sqrt{\lambda}}\,\exp(e^{-\frac{t}{2t_*}(\log|\log\norm{h(0,\cdot)}_{\theta\lambda}|)_{+}}\log \norm{h(0,\cdot)}_{\theta\lambda})  \\ \nonumber
   &+ 2 C(t)\frac{\Vert K_1\Vert_{L^2(\R^d)}e^{At_*}}{\sqrt{\theta\lambda}}\,\sqrt{t_*}\,\frac{1}{\sqrt{(\log |\log \Vert (0,\cdot)\Vert_{\theta\lambda}|)_+}}.
\end{align}
Note that $\lambda$ was chosen in such a way that $\norm{h(0,\cdot)}_{\lambda} < 1$. Since $\theta < 1$ also $\norm{h(0,\cdot)}_{\theta\lambda} < 1$ and we may estimate as follows
\begin{align*}
    &\exp(e^{-\frac{t}{2t_*}(\log|\log\norm{h(0,\cdot)}_{\theta\lambda}|)_{+}}\log \norm{h(0,\cdot)}_{\theta\lambda}) \leq \exp(e^{-(\log|\log\norm{h(0,\cdot)}_{\theta\lambda}|^\frac{1}{2})_{+}}\log \norm{h(0,\cdot)}_{\theta\lambda})\\
    &= \exp\left(\frac{\log\norm{h(0,\cdot)}_{\theta\lambda}}{\abs{\log\norm{h(0,\cdot)}_{\theta\lambda}}^{\frac{1}{2}}}\right) = \exp\left(-\abs{\log\norm{h(0,\cdot)}_{\theta\lambda}}^{\frac{1}{2}}\right) \leq \frac{1}{(\log |\log \Vert h(0,\cdot)\Vert_{\theta\lambda}|)_+^\frac{1}{2}}. 
\end{align*}
Inserting the above estimate in (\ref{fest7}) and noting that $C(t)$ is an increasing function of time we arrive at the claim of lemma.
}
\end{proof}

\subsection{Proof of the main result}
Finally we may prove  Theorem~\ref{maintheorem}. Here again we adjust the reasoning introduced in \cite{JPS2025} to our setting. 
\begin{proof}[Proof of Theorem~\ref{maintheorem}]
Define
\[
f_N^0(x,\xi,\eta):=N\sum_{i,j=1}^N \fji(0,x)\mI_j(\xi) \mI_i(\eta), \hd  x\in \R^d,\hd \xi,\eta \in [0,\ 1], \hd \fji(0,x) = X_i^0(x)_{\#}[\wji^0(x)d\mathbb{P}(x)],
\]
where $\mI_i,\mI_j$ are defined as in (\ref{deffn}). Since due the proof of Proposition \ref{fn} functions $g_{ij}$ are densities of $\fji(0,x)$ we have
\begin{align*}
 &\sup_{N\in \mathbb{N}}\norm{f^0_N}_{L^{\infty}_\eta(\mathcal{M}_\xi (W^{1,1}\cap W^{1,\infty})(\R^d))} \leq \sup_{N\in \mathbb{N}}\max_{1\leq i \leq N}\sum_{j=1}^N\norm{g_{ij}}_{(W^{1,1}\cap W^{1,\infty})(\R^d)} < \infty,\\ 
 &\sup_{N\in \mathbb{N}}\norm{f^0_N}_{L^{\infty}_\xi(\mathcal{M}_\eta (W^{1,1}\cap W^{1,\infty})(\R^d))} \leq \sup_{N\in \mathbb{N}}\max_{1\leq j \leq N}\sum_{i=1}^N\norm{g_{ij}}_{(W^{1,1}\cap W^{1,\infty})(\R^d)} < \infty
\end{align*}
and due to the assumption (\ref{asf}) 
\eqq{
\sup_{N \in \mathbb{N}}\norm{f^0_N}_{\W((W^{1,1}\cap W^{1,\infty})(\R^d))}=:C_0 < \infty.
}{ufnb}
Hence $f^0_N$ satisfies the assumption of Theorem \ref{lg1}.
Applying Theorem~\ref{lg1}, we infer that there exist nonnegative $f^0\in \W( (W^{1,1}\cap W^{1,\infty})(\R^d))$ and a subsequence (still denoted by $N$) such that for all $T \in \Tree$
\[
\tau(T,f^0)=\lim_{N \to \infty} \tau(T,f_N^0) \hd \m{ in } \hd L^p_{loc}(\R^{dn}), \hd p \in [1,\infty),
\]
where $n=|T|-1$. Using the assumption (\ref{secmo}) we will improve this convergence to hold globally in $\R^{dn}$.
At first, note that,  by Corollary \ref{tauest}, for every $T \in \Tree$, $\tau(T,f^0) \in (L^{1}\cap L^{\infty})(\R^{dn})$ and
\[
\norm{\tau(T,f^0)}_{L^\infty(\R^{dn})} \leq \norm{f^0}_{\W(L^\infty(\R^d))}^{n} , \hd 
\norm{\tau(T,f^0)}_{L^1(\R^{dn})} \leq \norm{f^0}_{\W(\mathcal{M}(\R^d))}^{n}.
\]

Clearly, 
\[
\int_{\R^d} |x|^2\,f_N^0(dx,\xi,\eta)
= N\sum_{i,j=1}^N \ird |X^0_i(x)|^2 \wji^0(x)d\p(x) \,\mI_i(\eta)\mI_j(\xi),
\]
hence, by the assumptions (\ref{M}) and (\ref{secmo}) there exists a positive constant $C$ independent of $N$ such that
\eqq{
\sup_{\eta \in (0,1)}\int_0^1\int_{\R^d} |x|^2\,f_N^0(dx,\xi,\eta)d\xi\leq \max_{1\leq i \leq N}\ird |X^0_i(x)|^2 \sum_{j=1}^N\wji^0(x)d\p(x) \leq  M \max_{1\leq i \leq N } \mathbb{E} |X^0_i|^2 \leq C.
}{boundetaxi}
However, due to asymmetrical role of $\xi,\eta$ in the definition of $f^0_N$, we cannot bound in the same way  $\sup_{\xi \in (0,1)}\int_0^1\int_{\R^d} |x|^2\,f_N^0(dx,\xi,\eta)d\eta$.
Hence, to prove equitightness of $\tau(T,f^0_N)$, we cannot use a direct extension of Corollary~\ref{tauest}, which gives for any fixed $R>0$
\[
\|\tau(T,f^0_N)\|_{L^1(\R^{dn}\setminus B_R^{n})}\leq n \|f^0_N\|_{\W( \mathcal{M}(\R^{d}))}^{n-1}\,\|f^0_N\|_{\W( \mathcal{M}(\R^{d}\setminus B_R))},
\]
where by $B_R$ we denote the ball centered at zero with radius $R$.
Instead, we investigate again the construction of algebra $\mathcal{F}$ and obtain the following asymmetrical estimate, whose inductive proof we provide in the Appendix.
\begin{prop}\label{asym}
Fix $R>0$.  If $F \in \mathcal{F}$ is a transform of rank $n$, then it satisfies the estimate
\[
\izj \norm{F(f)(\cdot,\xi)}_{L^{1}(\R^{dn}\setminus B_R^n)}d\xi \leq n \norm{f}^{n-1}_{\W((\mathcal{M}(\R^d)))} \norm{f}_{L^{\infty}_{\eta}(\mathcal{M}_\xi(\mathcal{M}(\R^d \setminus B_R)))}
\]
for any $f \in \W((L^1\cap L^{\infty})(\R^d))$.
\end{prop}
Recalling Definition \ref{deftf}, Proposition \ref{asym} together with estimate (\ref{boundetaxi}) leads to
\[
\|\tau(T,f^0_N)\|_{L^1(\R^{dn}\setminus B(0,R)^{n})}\leq \frac{Cn}{R^2}\|f^0_N\|_{\W( \mathcal{M}(\R^{d}))}^{n-1} \leq \frac{Cn}{R^2}C_0^{n-1}.
\]
Combining it with Corollary \ref{tauest} and (\ref{ufnb}), we infer that the sequence $\tau(T,f^0_N)$ is equicontinuous and equitight in $L^{1}(\R^{dn})$.
This together with the local convergence in $L^1$, it implies that $\tau(T,f_N^0)$ converges strongly to $\tau(T,f^0)$ in  $L^1(\R^{dn})$. Since $\tau(T,f_N^0)$ is also bounded in $L^{\infty}(\R^{dn})$, applying the interpolation inequality we obtain that $\tau(T,f_N^0)$ converges strongly to $\tau(T,f^0)$ in $L^p(\R^{dn})$ for any $p\in [1,\infty)$ and in particular in $L^2(\R^{dn})$.
Besides, applying again Corollary~\ref{tauest}, we may estimate as follows
\[
\|\tau(T,f^0)\|_{L^2(\R^{dn})}\leq \|f^0\|_{\W( L^2(\R^d))}^{n}.
\]
Thus, combining the convergence for any fixed tree together with the estimate above we obtain that there exists some $\lambda>0$ small enough such that 
\begin{equation}
\|\tau(\cdot,f^0)-\tau(\cdot,f^0_N)\|_{\lambda}\to 0,\quad\mbox{as}\ N\to\infty.\label{finalconv}
  \end{equation}
By Proposition~\ref{existenceprop}, there exists a nonnegative weak solution $f\in L^\infty((0,t_*);\W( W^{1,1}\cap W^{1,\infty})(\R^d))$ to \eqref{meanmean}  and the initial condition $f^0$. On the other hand, Proposition \ref{fn} states that $f^N(t,x,\xi,\eta)=N\sum_{i,j=1}^N \fji(t,x)\mathbb{I}_{i}(\eta) \mathbb{I}_{j}(\xi)$ is also a weak solution to \eqref{meanmean}  and belongs to $L^\infty((0,t_*);\W( W^{1,1}\cap W^{1,\infty})(\R^d))$. Hence, both $f^N$ and $f$ satisfy the assumptions of Proposition~\ref{mainstab}. Combining Proposition~\ref{mainstab} with \eqref{finalconv}, we obtain
\[
\esssup_{t \in (0,t_{*})}\left\|\int_0^1\int_0^1 (f-f_N)(t,\cdot,\xi,d\eta)\,d\xi\right\|_{L^2(\R^d)}\leq \frac{C}{{(\log |\log \|\tau(\cdot,f^0_N)-\tau(\cdot,f^0)\|_{\lambda}|)^{1/2}_+}}\to 0,\quad \mbox{as}\ N\to \infty.
\]
Let us prove that the above convergence is true also in $L^1(\R^d)$. Take  any $R>0$, then
\[
\left\|\int_0^1\int_0^1 (f-f_N)(t,\cdot,\xi,d\eta)\,d\xi\right\|_{L^1(\R^d)}\hspace{-0.2cm}
\]
\eqq{
\leq \left\|\int_0^1 \int_0^1 (f-f_N)(t,\cdot,\xi,d\eta)\,d\xi\right\|_{L^2(B_{R})} \hspace{-0.2cm}|B_{R}|^{\frac{1}{2}} + \frac{1}{R^2}\int_{\R^d \setminus B_R}|x|^2 \abs{\int_0^1 \int_0^1 (f+f_N)(t,\cdot,\xi,d\eta)\,d\xi}dx.  
}{l1conv}
Using (\ref{Mbound}) and  (\ref{secmobound}) we obtain that there exists $C>0$ which does not depend on $N$ such that
\[
\ird |x|^2 \izj \izj f_N(t,x,\xi,d\eta)d\xi dx \leq \max_{1\leq i \leq N}\ird |\xb_i(t)|^2 \sum_{j=1}^N\wjib(t)d\p(x) \leq C.
\]
Passing to the limit with $N$ in (\ref{l1conv}), using $L^2$- convergence, the bound above and Fatou Lemma 
 we obtain
\[
\lim_{N\rightarrow \infty}\esssup_{t \in (0,t_{*})}\left\|\int_0^1 \int_0^1 (f-f_N)(t,\cdot,\xi,d\eta)\,d\xi\right\|_{L^1(\R^d)} \leq \frac{C}{R^2}
\]
for some constant $C>0$. Due to the fact that $R$ is arbitrary we obtain the $L^1$ convergence and hence also convergence in flat norm: 
\[
\esssup_{t \in (0,t_{*})}d_F\left(\int_0^1 \int_0^1 f_N(t,\cdot,\xi,d\eta)\,d\xi,\int_0^1 \int_0^1 f(t,\cdot,\xi,d\eta)\,d\xi\right)\to 0,\quad \mbox{as}\ N\to \infty.
\]
Since, by Proposition \ref{fn},  $\int_0^1 \int_0^1 f_N(t,x,\xi,d\eta)\,d\xi=\frac{1}{N}\,\sum_{i,j=1}^N  f_i^j(t,x)$ we conclude the proof of Theorem~\ref{maintheorem}, recalling \eqref{finalp} from Theorem \ref{mainprob}. 
\end{proof}

\section{Appendix}

\subsection{Auxiliary results and technical proofs from Chapter 2}

We begin this section with the proof of Lemma \ref{exi}.

\begin{proof}[Proof of Lemma \ref{exi}]
The proof of the existence is rather standard and may be obtained following the fixed point argument, analogous to the one in the proof of \cite[Lemma 3]{nasza}. Let us provide the proof of independence. Since $\sum_{k=1}^N\fik$ is a deterministic measure for every $i \in 1,\dots, N$, we may rewrite (\ref{part2})$_{(i)}$ as $\partial_t \xb_i(t)=a_i(t,\xb_i(t))$, where, due to (\ref{asK}) function $a_i(t,x)$ is Lipschitz continuous in $x$, continuous in $t$ and bounded. Thus, there exists a flow $h_i:([0,t_{*});\R^d)$ continuous in time and Lipschitz in space such that $\xb_i(t) = h_i(t,X^0_i)$. Similarly, $\wjib(t)$ satisfies $\wjib(t)=\wjib^0\exp(\int_0^t\tilde{a}_i(s,\xb_i(s))ds)$, with $\tilde{a}_i(t,x)$ also Lipschitz continuous in $x$, continuous in $t$ and bounded. Thus, $\wjib(t)=\wjib^0\exp(\int_0^t\tilde{g}_i(s,h_i(s,X^0_i))ds)$. All in all, we may write 
\[
(\xb_i(t),\wjib(t)) = \Phi^i_t(X^0_i,\wji^0), \hd \hd \Phi^i_t(x,m):= \left(h_i(t,x),m\exp(\int_0^t\tilde{a}_i(s,h_i(s,x))ds)\right)
\]
and $\Phi^i_t:\R^d \times \R_{+}$ is a deterministic measurable map. Thus, the family  $(\xb_{i}(t),\bar{w}_{1i}(t),\dots,\bar{w}_{Ni}(t))_{i=1}^N$  is jointly independent for any $t \in (0,t_{*})$, which finishes the proof.
\end{proof}

Now let us recall the extension of the classical Glivenko-Cantelli lemma from \cite{Dudley}.

\begin{lemma}[Glivenko-Cantelli]\label{GC}\cite[Lemma 5]{nasza}
Discuss the probabilistic space $(\R^d,\mathcal{B}(\R^d),\p)$ with  $d\geq 1$. Let $X_i:\mathbb{R}^d\rightarrow \mathbb{R}^d$ and $M_i:\mathbb{R}^d \rightarrow \mathbb{R}_{+}$ for $i=1,\dots,N$, $N \in \mathbb{N}$ be two sequences of random variables such that $(X_i,M_i)$ and $(X_j,M_j)$ are independent for every $i\neq j$. Furthermore, assume that 
\eqq{\sup_{N \in \mathbb{N}}\frac{1}{N}\sum_{i=1}^N \mathbb{E}|X_i|^2 \leq C}{xas}
and that there exist a positive constant $M > 0$, such that
\eqq{\max_{1\leq i \leq N} \esssup_{x \in \R^d} |M_i(x)| \leq M.}{mas}
 Let us define
\[
\mu_N := \frac{1}{N}\sum_{i=1}^N M_i\delta_{X_i}, \hd \hd \nu_N:= \frac{1}{N}\sum_{i=1}^N d\tilde{f}_i, \hd \tilde{f}_i = X_{i\#}[M_i d\mathbb{P}].
\]
Then, there exists a positive $\bar{C} = \bar{C}(M,C,d)$ such that for every $N \in \mathbb{N} $
\[
\mathbb{E} d_{F}\left(\mu_N,\nu_N\right) \leq \bar{C}(M,C,d) N^{-\frac{1}{2+3d/2}}.
\]
\end{lemma}

Below we collect the technical proofs of selected results from Chapter 3. We begin with the proof of Lemma \ref{impoestv}.

\subsection{Technical proofs from Section 3.1 }

We begin this section with a proof of Lemma \ref{impoestv}.

\begin{proof}[Proof of Lemma \ref{impoestv}]
The proof is based on the result in scalar-valued case (Lemma \ref{impoest}) and repeats the ideas of its proof. Take $\phi_n \in C([0,1];X)$, $\phi_n\rightarrow \phi$ in $L^1((0,1);X)$ and satisfying the bounds $\norm{\phi_n}_{L^{\infty}((0,1);X)} \leq \norm{\phi}_{L^{\infty}((0,1);X)}$ and $\norm{\phi_n}_{L^{1}((0,1);X)} \leq \norm{\phi}_{L^{1}((0,1);X)}$. We apply Cauchy-Schwarz inequality  to arrive at
\[
\abs{\izj \left\langle \phi_n(\cdot,\eta),f(\cdot,\xi,d\eta)\right\rangle_{X\times X^*}} \leq \izj \norm{\phi_n(\cdot,\eta)}_{X}\norm{f(\cdot,\xi,\cdot)}_{X^*}(d\eta).
\]
Then, using the extension from Lemma \ref{impoest}
\begin{align} \label{vecext} \nonumber
    &\norm{\izj \left\langle \phi_n(\cdot,\eta),f(\cdot,\cdot,d\eta)\right\rangle_{X\times X^*}}_{L^{1}(0,1)} \leq  \norm{\phi}_{L^{1}((0,1);X)}\norm{f}_{L^{\infty}_\eta(\mathcal{M}_{\xi}(X^*))},\\ 
    &\norm{\izj \left\langle \phi_n(\cdot,\eta),f(\cdot,\cdot,d\eta)\right\rangle_{X\times X^*}}_{L^{\infty}(0,1)} \leq  \norm{\phi}_{L^{\infty}((0,1);X)}\norm{f}_{L^{\infty}_\xi(\mathcal{M}_{\eta}(X^*))}.
\end{align}
Furthermore, taking $\Psi \in C([0,1])$ and using the continuity of duality pairing we obtain
\begin{align*}
  \izj \Psi(\xi) \izj \left\langle \phi_n(\cdot,\eta),f(\cdot,\xi,d\eta)\right\rangle_{X\times X^*} d\xi& = \izj \left\langle \phi_n(\cdot,\eta),\izj \Psi(\xi)f(\cdot,d\xi,\eta)\right\rangle_{X\times X^*}d\eta   \\
  & \rightarrow \izj \Psi(\xi) \izj \left\langle \phi(\cdot,\eta),f(\cdot,\xi,d\eta)\right\rangle_{X\times X^*} d\xi
\end{align*}
and hence 
\[
\izj \left\langle \phi_n(\cdot,\eta),f(\cdot,\cdot,d\eta)\right\rangle_{X\times X^*} \overset{*}{\rightharpoonup} \izj \left\langle \phi(\cdot,\eta),f(\cdot,\cdot,d\eta)\right\rangle_{X\times X^*} \m{ in } \mathcal{M}([0,1]).
\]
But from the uniqueness of the limit due to (\ref{vecext}) the convergence holds also weakly$^*$ in $L^{\infty}$. Thus, indeed we have $\izj \left\langle \phi(\cdot,\eta),f(\cdot,\cdot,d\eta)\right\rangle_{X\times X^*} \in L^{\infty}((0,1))$
together with the estimates (\ref{vecexta}). In order to show the convergence, we take $\phi_n$ and $f_n$ as in the claim, $\Psi \in C([0,1])$ and again argue by linearity of duality pairing
\[
\izj \izj \left\langle \Psi(\xi)\phi_n(\cdot,\eta),f_n(\cdot,\xi,d\eta)\right\rangle_{X\times X^*} d\xi \rightarrow \izj \Psi(\xi)\izj \left\langle \phi(\cdot,\eta),f(\cdot,\xi,d\eta)\right\rangle_{X\times X^*} d\xi, 
\]
where we applied $\Psi(\xi)\phi_n(\cdot,\eta) \rightarrow \Psi(\xi)\phi(\cdot,\eta)$ in $L^{1}_\eta(C_\xi(X))$. Thus, we obtain weak$^*$ convergence in $\mathcal{M}([0,1])$ but since we have the bound (\ref{vecext}) we may lift it to weak$^*$ in $L^{\infty}((0,1))$, which finishes the proof.
\end{proof}

Below we provide a proof of Lemma \ref{coro2}.

\begin{proof}[Proof of Lemma \ref{coro2}]
The most involved part of the proof is to show that $(t,\xi)\mapsto \vf(t,x,y,\xi)=\\ \izj g(t,x,\eta)f(t,y,\xi,d\eta)$ is measurable with values in $L^{1}(\R^{2d})$. At first note that by Lemma \ref{impoest} for almost all $t,\xi$
\eqq{
\int_{\R^{2d}}\abs{\izj g(t,x,\eta)f(t,y,\xi,d\eta)}dx dy \leq \izj \ird |g(t,x,\eta)|dx \ird\abs{ f(t,y,\xi,\cdot)}dy(d\eta) < \infty
}{calfi}
and hence $\vf(t,\cdot,\cdot,\xi) \in L^{1}(\R^{2d})$ for almost all $t,\xi$. Since $L^1(\R^{2d})$ is separable by Pettis theorem it is enough to show that the mapping is weakly measurable. Let us take arbitrary $\Phi \in L^{\infty}(\R^{2d})$. Then applying standard  mollification argument together with the Stone-Weierstrass theorem we obtain that there exists $\Phi^n(x,y):=\sum_{k=1}^{m_n} \alpha_k^n(x) \beta_k^n(y)$, where $\alpha_k^n, \beta_k^n \in C_{c}(\R^d)$ for every $k=1,\dots m_n$, $n\in \mathbb{N}$, such that $\Phi^n \rightarrow \Phi$ pointwisely almost everywhere and $\norm{\Phi^n}_{L^{\infty}(\R^{2d})} \leq \norm{\Phi}_{L^{\infty}(\R^{2d})} + 1$. In order to show that the mapping $
(t,\xi)\mapsto \int_{\R^{2d}}\Phi^{n}(x,y) \vf(t,x,y,\xi)dxdy$ is measurable, by linearity it is enough to show that for fixed $\al,\beta \in C_{c}(\R^d)$ the mapping  $
(t,\xi)\mapsto \int_{\R^{2d}}\al(x)\beta(y) \vf(t,x,y,\xi)dxdy$ is measurable.
Note that for almost all $(t,\xi)$
\[
\int_{\R^{2d}}\al(x)\beta(y) \vf(t,x,y,\xi)dxdy = \izj \ird \al(x)g(t,x,\eta)dx \ird \beta(y)f(t,y,\xi,\cdot)(d\eta).
\]
Denoting $G(t,y,\eta) = \ird \al(x)g(t,x,\eta)dx \beta (y)$ we observe that $G \in L^{\infty}((0,t_{*})
\times (0,1);C_{c}(\R^d))$ and 
\[
\int_{\R^{2d}}\al(x)\beta(y) \vf(t,x,y,\xi)dxdy = \izj \langle G(t,\cdot,\eta),f(t,\cdot,\xi,d\eta) \rangle_{C_0(\R^d) \times \mathcal{M}(\R^d)}
\]
which is measurable due to Lemma \ref{impoestv} and the definition of $L^{\infty}((0,t_{*});\W(X^{*}))$. Hence, we obtain that for every $n$ the mapping $
(t,\xi)\mapsto \int_{\R^{2d}}\Phi^{n}(x,y) \vf(t,x,y,\xi)dxdy$ is measurable. Furthermore, for almost all $t,\xi$
\[
\abs{\Phi^{n}(x,y) \vf(t,x,y,\xi)} \leq (\norm{\Phi}_{L^{\infty}(\R^{2d})} +1)\abs{\vf(t,x,y,\xi)},
\]
which is integrable in $x$ and $y$ due to (\ref{calfi}). Hence, we may apply the Lebesgue dominated convergence theorem to obtain that for almost all $(t,\xi)$
\[
\int_{\R^{2d}} \Phi(x,y)\vf(t,x,y,\xi)dxdy = \lim_{n\rightarrow \infty} \int_{\R^{2d}} \Phi^n(x,y)\vf(t,x,y,\xi)dxdy 
\]
and the mapping is weakly measurable as a pointwise almost everywhere limit of measurable mappings.
To show estimate (\ref{coro2ej}) we take supremum with respect to $t,\xi$ in (\ref{calfi}) and apply Lemma \ref{impoest}, then
\[
\esssup_{t \in (0,t_{*})}\esssup_{\xi \in (0,1)} \int_{\R^{2d}}\abs{\izj g(t,x,\eta)f(t,y,\xi,d\eta)}dx dy \leq \esssup_{t \in (0,t_{*})}\norm{g(t,\cdot,\cdot)}_{L^{\infty}((0,1);L^{1}(\R^d))}\norm{f(t,\cdot,\cdot,\cdot)}_{L^{\infty}_\xi(\mathcal{M}_\eta(\mathcal{M}(\R^d)))}
\]
and we arrive at (\ref{coro2ej}). Similarly, for $p\in (1,\infty]$
\[
\norm{\vf(t,\cdot,\cdot,\cdot)}_{L^{\infty}( (0,1);L^{p}(\R^{2d}))}   \leq 
\esssup_{\xi \in (0,1)}\izj \norm{g(t,\cdot,\eta)}_{L^{p}(\R^d)}\norm{f(t,\cdot,\xi,\cdot)}_{L^{p}(\R^d)}(d\eta)
\]
and applying again Lemma \ref{impoest} we arrive at (\ref{coro2ei}).
\end{proof}

Let us now prove Lemma \ref{hmj}.

\begin{proof}[Proof of Lemma \ref{hmj}]
Clearly, $f^\ve \in L^{\infty}((0,1)^2\times \R^d)$. Since $\mathcal{M}([0,1])\hookrightarrow H^{-1}(0,1)$ the convergence in $L^1_\xi(H^{-1}_\eta(L^{1}(\R^d)))$ follows from standard properties of convergence with mollifier in Banach spaces. On the other hand, since $f \in L^{\infty}_\eta(H^{-1}_\xi(L^{1}(\R^d)))$, the convergence in $L^1_\eta(H^{-1}_\xi(L^{1}(\R^d)))$ follows from the convergence of convolution with mollifier in $H^{-1}_\xi(L^{1}(\R^d))$, together with Lebesgue dominated convergence. Furthermore, for $Y=L^{\infty}(\R^d)$ and $Y = \mathcal{M}(\R^d)$ we have
\begin{align*}
\norm{f^\ve}_{L^{\infty}_\xi(\mathcal{M}_\eta(Y))} &= \esssup_{\xi \in (0,1)} \izj \norm{f^\ve(\cdot,\xi,\cdot)}_{Y}(d\eta)\\
&\leq \esssup_{\xi \in (0,1)}\left( \rho^{\ve} * \izj  \norm{f(\cdot,\cdot,\cdot)}_{Y}(d\eta)\right)(\xi) \leq \norm{f}_{L^{\infty}_\xi(\mathcal{M}_\eta(Y))} 
\end{align*}
and 
\begin{align}\label{etxi}
\norm{f^\ve}_{L^{\infty}_\eta(\mathcal{M}_\xi(Y))}& = \esssup_{\eta \in (0,1)} \izj \norm{f^\ve(\cdot,\cdot,\eta)}_{Y}(d\xi)=\esssup_{\eta \in (0,1)}\izj \left( \rho^{\ve} *  \norm{f(\cdot,\cdot,\eta)}_{Y}\right)(\xi)d\xi\\ \nonumber
& = \esssup_{\eta \in (0,1)}
\izj \izj \rho^{\ve}(\xi-\xi')\norm{f(\cdot,\cdot,\eta)}_{Y}(\xi')d\xi
= \norm{f}_{L^{\infty}_\eta(\mathcal{M}_\xi(Y))}, 
\end{align}
where in the last identity we applied Fubini theorem, together with the fact that the periodic mollifier $\rho^\ve$ integrates to one on any interval of length one. Hence, we have established that $\norm{f^\ve}_{\W(L^{\infty}(\R^d))} \leq \norm{f}_{\W(L^{\infty}(\R^d))}$. It remains to prove the convergence. Fix $g \in L^{1}((0,1)\times \R^d)$ and any sequence $g^{\ve}\rightarrow g$ in $L^{1}((0,1)\times \R^d)$. For $g \in L^{1}((0,1)\times \R^{dn})$ with $n>1$ the proof is analogous. Let us estimate the difference
\[
\izj \int_{\R^{2d}}\abs{\izj g^\ve(y,\eta)f^\ve(x,\xi,d\eta) - \izj g(y,\eta)f(x,\xi,d\eta)}dx dy d\xi
\]
\[
\leq \izj \int_{\R^{2d}}\abs{\izj(g^\ve - g)(y,\eta)f^\ve(t,x,\xi,d\eta)} dx dy d\xi 
+ \izj \int_{\R^{2d}}\abs{\izj g(y,\eta)(f^\ve-f)(t,x,\xi,d\eta)} dx dy d\xi =: I^\ve_1+I_2^\ve.
\]
Applying Fubini theorem, Lemma \ref{impoest} and the   estimate (\ref{etxi})  we have as $\ve \rightarrow 0$
\[
I^\ve_1 \leq \norm{g^\ve - g}_{L^{1}((0,1);L^{1}(\R^{d}))} \norm{f^\ve}_{L^{\infty}_\eta(\mathcal{M}_\xi(\mathcal{M}(\R^d)))} \longrightarrow 0.
\]
In order to estimate $I_2^\ve$ we introduce a smooth approximation of $g$. Let $\Psi^N \in C_c((0,1)\times \R^{d})$, $\Psi^N\rightarrow g$ in $L^{1}((0,1)\times \R^{d})$. Then by a triangle inequality, together with Fubini theorem, Cauchy-Schwarz estimate and Lemma \ref{impoest}
\begin{align*}
    I^\ve_2 &\leq \norm{f^\ve-f}_{L^{1}_\xi(H^{-1}_\eta(L^{1}(\R^d)))}\norm{\Psi^N}_{H^{1}_\eta(L^{1}(\R^{d}))}\\
    & + \norm{\Psi^N-g}_{L^{1}((0,1)\times \R^{d})} \left(\norm{f^\ve}_{L^{\infty}_\eta(\mathcal{M}_\xi(\mathcal{M}(\R^d)))} + \norm{f}_{L^{\infty}_\eta(\mathcal{M}_\xi(\mathcal{M}(\R^d)))} \right).
\end{align*}
From what we have already established the first term converges to zero as $\ve \rightarrow 0$. Thus, we pass firstly with $\ve$ to the limit, make use of the bound (\ref{etxi}) and then pass with $N$ to the limit to obtain that $I^\ve_2$ converges to zero. Hence, $\izj g^\ve(y,\eta)f^\ve(x,\xi,d\eta) \rightarrow \izj g(y,\eta)f(x,\xi,d\eta)$ in $L^{1}_\xi(L^{1}(\R^{2d}))$. To show the convergence $\izj g^\ve(y,\xi)f^\ve(x,d\xi,\eta) \rightarrow \izj g(y,\xi)f(x,d\xi,\eta)$ in $L^{1}_\eta(L^{1}(\R^{2d}))$, we proceed symmetrically with respect to $\xi$ and $\eta$, which is possible since we established symmetric bound  and strong convergence results in $L^1_\xi(H^{-1}_\eta(L^{1}(\R^d)))$ and $L^1_\eta(H^{-1}_\xi(L^{1}(\R^d)))$ above.

\end{proof}

\subsection{Technical proofs from Section 3.2 }

Below we provide the proof of the existence of solutions to the linear problem (\ref{independ2-linear}).

\begin{proof}[Proof of Lemma \ref{linearlemma}]
We would like to firstly apply the theory of parabolic/transport equations for almost all $\xi,\eta$. In order to do so let us firstly assume that $f^0 \in L^{\infty}((0,1)^{2};(W^{1,1}\cap W^{1,\infty})(\R^d))$.

Problem \eqref{independ2-linear} is a linear equation parametrized by $\xi,\eta$, hence, for $\nu=0$ by Di Perna-Lions theory (\cite{diPerna}, \cite[Theorem 6.4]{Perthame}) there exists a unique solution to (\ref{independ2-linear}) in $L^{\infty}((0,t_{*})\times (0,1)^2\times \R^d)$, which 
may be obtained by the method of characteristics, namely
\eqq{
f(t,\cdot,\xi,\eta)=X_g(t,\cdot\,\eta)_{\#}\left[f^0(\cdot,\xi,\eta)\exp\left(At - \int_0^tV_2[g](s,X_g(s,x,\eta),\eta))ds\right)\right],
}{integro}
where $X_g$ denotes the flow solving the equation of  characteristics 
\[
\left\{
\begin{array}{l} 
\displaystyle \frac{dX_g}{dt}(t,x,\eta)=V_1(t,X_g(t,x,\eta),\eta),\\
X_g(0,x,\eta)=x.
\end{array}
\right.
\]
Since the initial data satisfies $f^0 \geq 0$, it follows from  (\ref{integro})  that $f \geq 0$ almost everywhere.
In case $\nu>0$, the heat kernel $G_\nu(t,x)$ can be employed to obtain a  solution through the standard mild formulation
\begin{equation}\label{mildformulation}
  \begin{split}
    &f(t,x,\xi,\eta)=G_\nu(t,.)\star_x f^0 (x,\xi,\eta)\\
    &\qquad+\int_0^t\int_{\R^d} G_\nu(t-s,x-y)\,\left[-\divv_y\left(f(s,y,\xi,\eta)\,V_1[g](s,y,\eta)\right) + f(s,y,\xi,\eta)(A-V_2[g](s,y,\eta))\right]\,dy\,ds.
\end{split}
  \end{equation}
  To show that $f$ remains non-negative also in the case $\nu > 0$ we test the equation with the negative part of $f$, denoted by $f_{-}$. Then utilizing nonnegativity of $g,K_2$ and  integration by parts yields
\[
\frac{1}{2}\partial_{t}\int_{\R^d} |f_{-}|^2(t,x,\xi,\eta)dx \leq \frac{1}{2}\norm{\divv V_1[g]}_{L^{\infty}((0,t_{*})\times (0,1)\times \R^d)}\int_{\R^d} |f_{-}|^2(t,x,\xi,\eta)dx +  A \int_{\R^d} |f_{-}|^2(t,x,\xi,\eta)dx.
\] 
Applying firstly (\ref{v3prim}) and subsequently Gr\"onwall inequality we obtain for almost all $\xi,\eta \in (0,1), t \in (0,t_{*})$ 
\[
\int_{\R^d} |f_{-}|^2(t,x,\xi,\eta)dx \leq \exp(2A+\norm{ K_1}_{L^{1}(\R^d)}\norm{g}_{E})\int_{\R^d} |f^0_{-}|^2(x,\xi,\eta)dx
\]
and since $f^0 \geq 0$ we conclude that $f \geq 0$ almost everywhere also in a case $\nu > 0$.
Now we will show that assuming $g \in E_{\Upsilon}$ also the solution $f \in E_{\Upsilon}$ for appropriately chosen $\Upsilon$ and $t_{*}$ small enough.

In order to show that the solution belongs to $E_\Upsilon$, we apply similar reasoning as in \cite{JPS2025} and \cite{nasza}. Consider any $v(t,x)\in L^\infty((0,\ t_*);\;W^{1,\infty}(\R^d))$, any weak solution $u$ in $L^1_{loc}$ to
\begin{equation}
\left\{\begin{array}{l}
\partial_t u+\divv(u\,v)=\nu\,\Delta u+R(t,x),\\
u(0,x)=u^0(x).
\end{array}\right.
\label{advectiondiffusion}\end{equation}
Then  $u\in L^\infty((0,t_*); L^1(\R^d))$ if only $R\in L^1((0,t_*)\times \R^d)$, and we have 
\begin{equation}
\|u(t,\cdot)\|_{L^1(\R^d)}\leq \|u^0\|_{L^1(\R^d)}+\int_0^t \|R(s,\cdot)\|_{L^1(\R^d)}\,ds.\label{propL1}
\end{equation}
Similarly,  $R\in L^1((0,t_*);L^{\infty}(\R^d))$, yields $u\in L^\infty((0,t_*)\times \R^d)$ and
\begin{equation}
\begin{split}
  \|u(t,\cdot)\|_{L^\infty(\R^d)}\leq &\|u^0\|_{L^\infty(\R^d)}\,\exp\left(t\, \|\divv v\|_{L^\infty((0,t_*)\times \R^d)}\right)\\
  &+\int_0^t \|R(s,\cdot)\|_{L^\infty(\R^d)}\,\,\exp\left((t-s)\, \|\divv v\|_{L^\infty((0,t_*)\times \R^d)}\right)\,ds.\label{propLinfty}
\end{split}
\end{equation}
Furthermore, better regularity of $u^0$, $v$ and $R$, implies better regularity of $u$. Indeed, if $u^0\in W^{1,1}\cap W^{1,\infty}(\R^d)$, $v$ belongs to $(W^{2,\infty}\cap W^{1,1})(\R^d)$ with $\divv v \in W^{1,\infty}(\R^d)$ and $R \in (W^{1,1}\cap W^{1,\infty})(\R^d)$ then
\begin{equation}
  \begin{split}
    &\|\nabla u(t,\cdot)\|_{L^1(\R^d)}\leq \|\nabla u^0\|_{L^1(\R^d)}\,\exp\left(t\,\| v\|_{L^\infty((0,t_{*});W^{1,\infty} (\R^d))}\right)\\
        &+\int_0^t \exp\left((t-s)\,\| v\|_{L^\infty((0,t_*); W^{1,\infty}(\R^d))}\right)\left(\|\divv v(s,\cdot)\|_{W^{1,\infty}(\R^d)}\| u(s,\cdot)\|_{L^1(\R^d)}+ \norm{\nabla R(s,\cdot)}_{L^{1}(\R^d)} \right)\,ds.
\end{split}\label{propW11}
  \end{equation}
  and
\begin{align}\label{gradinfty}
 \|\nabla u(t,\cdot)\|_{L^\infty(\R^d)}&\leq \|\nabla u^0\|_{L^\infty(\R^d)}\,\exp\left(t \left( \|\divv v\|_{L^\infty((0,t_{*})\times \R^d)} +\norm{v}_{L^{\infty}((0,t_*);W^{1,\infty}(\R^d))} \right)\right)\\ \nonumber
 &+\int_0^t \exp\left((t-s)\,\left(\| v\|_{L^\infty((0,t_*); W^{1,\infty}(\R^d))}+\|\divv v\|_{L^\infty((0,t_{*})\times \R^d)}\right)\right)\alpha(s) ds, 
\end{align}
where
\[
\alpha(s) := \|\divv v(s,\cdot)\|_{W^{1,\infty}(\R^d)}\|u(s,\cdot)\|_{L^{\infty}(\R^d)} + \norm{\nabla R(s,\cdot)}_{L^{\infty} (\R^d)}\,.
\]
For a detailed derivation of these classical estimates we refer to  \cite[Appendix]{nasza}.
Note that for a given $g\in E$ equation (\ref{independ2-linear}) may be written in the form \eqref{advectiondiffusion} where $\xi,\eta$ are only parameters and $R$ is linear w.r.t. $f$, i.e. 
\[
v_{\eta}(t,x)=V_1[g](t,x,\eta), \hd R_{\xi,\eta}(t,x) = f(t,x,\xi,\eta) (A -V_2[g](t,x,\eta) ). 
\]
Taking advantage of the specific structure of our equation and the previously established non-negativity of $f$, we can obtain simpler estimates for both the $L^1$ and $L^
 \infty$ - norms of the solution. Indeed, since we assume $g,K_2 \geq 0$ we obtain that $R_{\xi,\eta} \leq Af$ and integrating (\ref{independ2-linear}) we arrive at 
 \eqq{
 \norm{f(t,\cdot,\xi,\eta)}_{L^{1}(\R^d)} \leq e^{At_{*}}\norm{f^0(\cdot,\xi,\eta)}_{L^{1}(\R^d)}.
}{l1bound0}
Applying to both sides $\W$-norm we obtain 
\eqq{
\norm{f}_{L^{\infty}((0,t_*); \W(\mathcal{M}(\R^d)))} \leq e^{At_{*}}\norm{f^0}_{\W(\mathcal{M}(\R^d))}.
}{l1bound}
Testing (\ref{independ2-linear}) with $f^{p-1}$  we obtain
\[
\partial_t\ird f^{p}(t,x,\xi,\eta)dx \leq (p-1)\ird f^{p}(t,x,\xi,\eta)dx \norm{\divv V_1[g]}_{L^{\infty}((0,t_{*})\times (0,1)\times \R^d)} + A \ird f^{p}(t,x,\xi,\eta)dx. 
\]
Applying (\ref{v5}), Gr\"onwall inequality and subsequently  passing to the limit with $p\rightarrow \infty$ we infer
\eqq{
\norm{f(t,\cdot,\xi,\eta)}_{L^{\infty}(\R^d)}
\leq \exp\left(t_{*}\left(\norm{\divv K_{1}}_{L^{\infty}(\R^d)}\norm{ g}_{L^{\infty}((0,t_{*});\W(\mathcal{M}(\R^d)))} +A\right)\right)\norm{f^0(\cdot,\xi,\eta)}_{L^{\infty}( \R^d)}
}{finfty0}
and hence
\eqq{
\norm{f}_{L^{\infty}((0,t_{*});\W(L^{\infty}(\R^d))}
\leq \exp \left(t_{*}\left(\norm{\divv K_{1}}_{L^{\infty}(\R^d)}\norm{ g}_{L^{\infty}((0,t_{*});\W(\mathcal{M}(\R^d)))} +A\right)\right)\norm{f^0}_{\W(L^{\infty}( \R^d))}.
}{finfty}
It remains to estimate gradient terms. 
Let us estimate the gradient terms of $R$. By (\ref{v1}) and (\ref{v4}) 
\begin{align}\label{Rgrad1}
  \norm{\nabla R_{\xi,\eta}(s,\cdot)}_{L^{1}(\R^d))} & \leq \norm{\nabla f(s,\cdot,\xi,\eta)}_{L^{1}(\R^d)}\left(A+\norm{K_{2}}_{L^{\infty}(\R^d)}\norm{g(s,\cdot,\cdot,\cdot)}_{\W(\mathcal{M}(\R^d))}\right)\\ \nonumber  
  &+\norm{f(s,\cdot,\xi,\eta)}_{L^{\infty}(\R^d)}   \norm{K_{2}}_{L^{1}(\R^d)}\norm{\nabla g(s,\cdot,\cdot,\cdot)}_{\W(\mathcal{M}(\R^d))}
\end{align}
and similarly 

\begin{align}\label{Rgradinfty}
\norm{\nabla R_{\xi,\eta}(s,\cdot)}_{L^{\infty}(\R^d))} &\leq \norm{\nabla f(s,\cdot,\xi,\eta)}_{L^{\infty}(\R^d))}\left(A+\norm{K_{2}}_{L^{\infty}(\R^d)}\norm{g(s,\cdot,\cdot,\cdot)}_{\W(\mathcal{M}(\R^d))}\right) \\ \nonumber
& +\norm{f(s,\cdot,\xi,\eta)}_{L^{\infty}(\R^d)}\norm{K_{2}}_{L^{1}(\R^d)}\norm{g(s,\cdot,\cdot,\cdot)}_{\W(L^{\infty}(\R^d))}.
\end{align}
Thus, using
(\ref{propW11}) together with the estimates (\ref{v1}), (\ref{v2}), (\ref{v5}) and  (\ref{Rgrad1}) gives
\[
  \begin{split}
    &\|\nabla f(t,\cdot,\xi,\eta)\|_{ L^1(\R^d)}\leq \|\nabla f^0(\cdot,\xi,\eta)\|_{L^1(\R^d))}\,\exp\left(t \norm{K_1}_{W^{1,1}(\R^d)}\norm{g}_{L^{\infty}((0,t_*);\W(L^{\infty}(\R^d)))}\right)\\
        &+\int_0^t \exp\left((t-s)\,\norm{K_1}_{W^{1,1}(\R^d)}\norm{g}_{L^{\infty}((0,t_*);\W(L^{\infty}(\R^d)))}\right)\\
        &\times \left( \norm{\nabla f(s,\cdot,\xi,\eta)}_{L^{1}(\R^d))}\left(A+\norm{K_{2}}_{L^{\infty}(\R^d)}\norm{g(s,\cdot,\cdot,\cdot)}_{\W(\mathcal{M}(\R^d))} \right)+ \beta(s,\xi,\eta) \right)\,ds,
\end{split}
\]
where we denote
\begin{align*}
    \beta(s,\xi,\eta) &:= \norm{\divv K_1}_{L^{\infty}(\R^d)}\left[\norm{g(s,\cdot,\cdot,\cdot)}_{\W(\mathcal{M}(\R^d))}\|+\norm{\nabla g(s,\cdot,\cdot,\cdot)}_{\W(\mathcal{M}(\R^d))}\|\right] \|f(s,\cdot,\xi,\eta)\|_{L^1(\R^d)}\\
    &+\norm{f(s,\cdot,\xi,\eta)}_{L^{\infty}( \R^d)}\norm{\nabla g(s,\cdot,\cdot,\cdot)}_{\W(\mathcal{M}(\R^d))}   \norm{K_{2}}_{L^{1}(\R^d)}.
\end{align*}
Thanks to (\ref{l1bound0}) and (\ref{finfty0}) we may estimate 
\begin{align*}
    \norm{\beta(s)}_{\W} &\leq \left(C_1e^{At_{*}}\norm{f^0}_{\W(\mathcal{M}(\R^d))} + C_2e^{t_{*}C_{3}\norm{g}_{L^{\infty}((0,t_{*});\W(\mathcal{M}(\R^d)))}}\norm{f^0}_{\W(L^{\infty}(\R^d))} \right)\\
    &\times \left[\norm{g(s,\cdot,\cdot,\cdot)}_{\W(\mathcal{M}(\R^d))}\|+\norm{\nabla g(s,\cdot,\cdot,\cdot)}_{\W(\mathcal{M}(\R^d))}\|\right] ,
\end{align*}
where $C_{1},C_{2},C_{3}$ are positive constants dependent only on norms of $K_{1}$, $K_2$ and $A$.
Furthermore, by Gr\"onwall lemma we get
\begin{align*}
&\|\nabla f(t,\cdot,\xi,\eta)\|_{L^1(\R^d)} \leq \left[ \|\nabla f^0(\cdot,\xi,\eta)\|_{L^1(\R^d)} + \int_0^{t}\beta(s,\xi,\eta)ds\right]\\ \nonumber
&\times \exp\left(\left(A+(\norm{K_{1}}_{W^{1,1}(\R^d)}+\norm{K_{2}}_{L^{\infty}(\R^d)})\norm{g}_{L^{\infty}((0,t_{*});\W(\mathcal{M}(\R^d)\cap L^{\infty}(\R^d)))}\right)t\right), 
\end{align*}
which leads to
\begin{align}\label{gradl1}
&\|\nabla f(t,\cdot,\cdot,\cdot)\|_{\W(\mathcal{M}(\R^d))} \leq \left[ \|\nabla f^0\|_{\W(\mathcal{M}(\R^d))} + \int_0^{t}\norm{\beta(s)}_{\W}ds\right]\\ \nonumber
&\times \exp\left(\left(A+(\norm{K_{1}}_{W^{1,1}(\R^d)}+\norm{K_{2}}_{L^{\infty}(\R^d)})\norm{g}_{L^{\infty}((0,t_{*});\W(\mathcal{M}(\R^d)\cap L^{\infty}(\R^d)))}\right)t\right).
\end{align}

With the $L^{\infty}$ - norm of the gradient we deal as follows. Using (\ref{gradinfty}) together with (\ref{Rgradinfty}) and (\ref{v2}), (\ref{v5}) we arrive at
\[
  \begin{split}
    &\|\nabla f(t,\cdot,\xi,\eta)\|_{L^{\infty}( \R^d)}\leq \|\nabla f^0(\cdot,\xi,\eta)\|_{L^{\infty}(\R^d )}\,\exp\left(2t \norm{K_1}_{W^{1,1}(\R^d)}\norm{g}_{L^{\infty}((0,t_{*});\W(L^{\infty}(\R^d)))}\right)\\
        &+\int_0^t \exp\left(2(t-s)\,\norm{K_1}_{W^{1,1}(\R^d)}\norm{g}_{L^{\infty}((0,t_{*});\W(L^{\infty}(\R^d)))}\right)\\&\times \left( \norm{\nabla f(s,\cdot,\xi,\eta)}_{L^{\infty}(\R^d)}\left(A+\norm{K_{2}}_{L^{\infty}(\R^d)}\norm{g(s,\cdot,\cdot,\cdot)}_{\W(\mathcal{M}(\R^d))} \right)+ \gamma(s) \right)\,ds,
\end{split}
\]
where
\begin{align*}
    \gamma(s,\xi,\eta) &:= \norm{\divv K_1}_{L^{\infty}(\R^d)}\left[\norm{g(s,\cdot,\cdot,\cdot)}_{\W(\mathcal{M}(\R^d))}\|+\norm{\nabla g(s,\cdot,\cdot,\cdot)}_{\W(\mathcal{M}(\R^d))}\|\right] \|f(s,\cdot,\xi,\eta)\|_{L^\infty(\R^d)}\\
    &+\norm{f(s,\cdot,\xi,\eta)}_{L^{\infty}( \R^d)}\norm{\nabla g(s,\cdot,\cdot,\cdot)}_{\W(L^{\infty}(\R^d))}   \norm{K_{2}}_{L^{1}(\R^d)}.
\end{align*}

From (\ref{finfty0}) we infer 
\[
\gamma(s) \leq C_4e^{ t_{*}C_5\norm{g}_{L^{\infty}((0,t_{*});\W(\mathcal{M}(\R^d)))}}\left[\norm{g(s,\cdot,\cdot,\cdot)}_{\W(\mathcal{M}(\R^d))}\|+\norm{\nabla g(s,\cdot,\cdot,\cdot)}_{\W((\mathcal{M}\cap L^{\infty})(\R^d))}\| \right]\norm{f^0(\cdot,\xi,\eta)}_{L^{\infty}(\R^d)},
\]
where $C_{4},C_{5}$ are positive constants dependent only on norms of $K_{1}, K_2$ and $A$.
We utilize again Gr\"onwall estimate to the result
\begin{align*}
&\|\nabla f(t,\cdot,\xi,\eta)\|_{L^\infty(\R^d)} \leq \left[ \|\nabla f^0(\cdot,\xi,\eta)\|_{L^\infty(\R^d)} + \int_0^{t}\gamma(s,\xi,\eta)ds\right]\\ \nonumber
&\times \exp\left(2t\left(A+(\norm{K_{1}}_{W^{1,1}(\R^d)}+\norm{K_{2}}_{L^{\infty}(\R^d)})\norm{g}_{L^{\infty}((0,t_{*});\W(\mathcal{M}(\R^d)\cap L^{\infty}(\R^d)))}\right)\right), 
\end{align*}
Thus,
\begin{align}\label{gradlinfty}
       \|\nabla f\|_{L^{\infty}((0,t_*);\W(L^{\infty}(\R^d)))} &\leq \left(\|\nabla f^0\|_{\W(L^{\infty}(\R^d))} + \int_0^{t}\norm{\gamma(s)}_{\W}ds\right) \\ \nonumber
&\times \exp\left(2t\left(A+(\norm{K_{1}}_{W^{1,1}(\R^d)}+\norm{K_{2}}_{L^{\infty}(\R^d)})\norm{g}_{L^{\infty}((0,t_{*});\W(\mathcal{M}(\R^d)\cap L^{\infty}(\R^d)))}\right)\right).
\end{align}

Let us now fix $f^0 \in \W((W^{1,1}\cap W^{1,\infty})(\R^d))$. We extend $f^0$ periodically in $\xi$ and convolve it with periodic mollifying kernel in $\xi$ variable. In this way we obtain nonnegative  sequence $f_{\ve}^0  \in L^{\infty}((0,1)^2;(W^{1,1}\cap W^{1,\infty})(\R^d))$ which converges to $f^0$ weakly$^*$ in $\W((W^{1,1}\cap W^{1,\infty})(\R^d))$ with the bounds given by Lemma~\ref{hmj}. Note that due to higher regularity of $f$ in our case, we may apply the bound in Lemma \ref{hmj} also to $\nabla f^0$. Thus, we obtain
\begin{align*}
&\norm{f^0_{\ve}}_{\W(\mathcal{M}(\R^d))} \leq \norm{f^0}_{\W(\mathcal{M}(\R^d))}, \hd \norm{f^0_{\ve}}_{\W(L^{\infty}(\R^d))} \leq \norm{f^0}_{\W(L^{\infty}(\R^d))}\\
&\norm{\nabla f^0_{\ve}}_{\W(\mathcal{M}(\R^d))} \leq \norm{\nabla f^0}_{\W(\mathcal{M}(\R^d))}, \hd \norm{\nabla f^0_{\ve}}_{\W(L^{\infty}(\R^d))} \leq \norm{\nabla f^0}_{\W(L^{\infty}(\R^d))}.
\end{align*}
For a fixed $g \in E_{\Upsilon}$ denote by $f^\ve$ the solution to linear problem (\ref{independ2-linear}) with initial condition $f^0_\ve$. Then, using the estimates (\ref{l1bound}), (\ref{finfty}), (\ref{gradl1}), (\ref{gradlinfty}) and the bound above, we conclude that $f^\ve$ is uniformly bounded in $E$. Thus, it has a subsequence (still denoted by $\ve$) which converges weakly$^*$ in $E$ to some $f$. Clearly, the limit $f$ satisfies (\ref{independ2-linear}) in a weak sense, since for any test function $\vf$ from the definition of weak solutions we have
\[
\vf V_{i}[g] \in L^{1}((0,t_{*});L^{1}_{\eta}C_{\xi}((W^{1,1}\cap W^{1,\infty})(\R^d))), \hd \hd i = 1,2.
\] 
Recalling the Definition \ref{weakspacedef}  we may pass to the limit in all terms obtaining that $f$ is actually a weak solution to (\ref{independ2-linear}). Moreover, passing to the weak$^*$ limit in (\ref{integro}) (respectively (\ref{mildformulation})) for $f^\ve$, we obtain that  $f(t)$ is weak$^*$ continuous  and $f|_{t=0} = f^0$. Uniqueness of the solution in $L^{\infty}((0,t_{*});\W(\mathcal{M}(\R^d)))$ follows from the linearity of the equation and estimate (\ref{l1bound}).
In order to finish the proof of Lemma \ref{linearlemma}, we note that the estimates (\ref{l1bound}), (\ref{finfty}), (\ref{gradl1}) and (\ref{gradlinfty}) show that if $\Upsilon> \norm{f^0}_{\W((W^{1,1}\cap W^{1,\infty})(\R^d))}e^{A}$, we obtain that for $g \in E_{\Upsilon}$ the solution $f \in E_{\Upsilon}$ for sufficiently small $t_{*}$ which depends on $\Upsilon$, $A$ and norms of $K_{1}$ and $K_{2}$. 
\end{proof}

\begin{proof}[Proof of the propagation of regularity of $f^j_i(0)$ from Proposition \ref{fn}]
 Let us show that $W^{1,1}\cap W^{1,\infty}$- regularity of $f^j_i(0)$ propagates in time.  We observe that the equation (\ref{semic}),   may be written as $\partial_t f_i^j + \divv (f_i^j v^1[f^i]) = f_i^j v^2[f^i]$, $i,j=1,\dots N$, where
\[
v^1[f^i](t,x) :=  \ird K_{1}(x-y)\sum_{k=1}^Nf_k^i(t,dy), \hd \hd  v^2[f^i](t,x):= A - \ird K_{2}(x-y)\sum_{k=1}^Nf_k^i(t,dy).
\]
Then, using the fact that $v^1$ is Lipschitz and $f_i^j, K_2 \geq 0$ we have
\[
\norm{f^j_i(t)}_{L^{1}(\R^d)}\leq e^{At}\norm{f^j_i(0)}_{L^{1}(\R^d)}.
\]
Furthermore, 
\[
\norm{v^1[f^i]}_{L^1(\R^d)} \leq e^{At}\norm{K_1}_{L^{1}(\R^d)}\sum_{k=1}^N\norm{f^i_k(0)}_{L^{1}(\R^d)},  \norm{v^1[f^i]}_{L^\infty(\R^d)} \leq e^{At}\norm{K_1}_{L^{\infty}(\R^d)}\sum_{k=1}^N\norm{f^i_k(0)}_{L^{1}(\R^d)}
\]
and
\[
\norm{\nabla v^1[f^i]}_{L^1(\R^d)} \leq e^{At}\norm{\nabla K_1}_{L^{1}(\R^d)}\sum_{k=1}^N\norm{f^i_k(0)}_{L^{1}(\R^d)},   \norm{\nabla v^1[f^i]}_{L^\infty(\R^d)} \leq e^{At}\norm{\nabla K_1}_{L^{\infty}(\R^d)}\sum_{k=1}^N\norm{f^i_k(0)}_{L^{1}(\R^d)}.
\]
Clearly, we obtain the analogous estimates for $v^2$. Thus, using classical $L^{\infty}$ bound for transport problems together with nonnegativity of $K_2$ and $f^j_i$ we obtain 
\[
\norm{f_i^j(t)}_{L^{\infty}(\R^d)}\leq \exp\left(t\left( e^{At}\norm{\divv K_1}_{L^{\infty}(\R^d)}\sum_{k=1}^N\norm{f^i_k(0)}_{L^{1}(\R^d)} + A\right)\right)\norm{f^j_i(0)}_{L^{\infty}(\R^d)}.
\]
Note that here we allow nonuniform in $N$ estimates. 
We may obtain a priori bounds in $W^{1,1}$ and $W^{1,\infty}$   using the estimates (\ref{propW11}) and (\ref{gradinfty}), with $v=v^1[f^i]$ and $R =f^j_i v^2[f^i]$.
Since the terms involving $\nabla R$ involve also $\nabla f^j_i$ and $\norm{\divv v^1[f^i]}_{W^{1,\infty}(\R^d)}$ involves also $\nabla f^i$, we get the estimates of the form
\begin{align*}
&\norm{\nabla f^j_i(t)}_{L^{1}(\R^d)}\leq C\left(\norm{\nabla f^j_i(0)}_{L^{1}(R^d)} + \int_{0}^t \norm{\divv v^1[f^i](s,\cdot)}_{W^{1,\infty}(\R^d)} + \norm{\nabla f^j_i(s)}_{L^{1}(\R^d)}ds\right),\\ 
&\norm{\nabla f^j_i(t)}_{L^{\infty}(\R^d)}\leq C\left(\norm{\nabla f^j_i(0)}_{L^{\infty}(R^d)} + \int_{0}^t \norm{\divv v^1[f^i](s,\cdot)}_{W^{1,\infty}(\R^d)} + \norm{\nabla f^j_i(s)}_{L^{\infty}(\R^d)}ds\right),
\end{align*}
where we included in the constant $C$ the constants dependent on $t_{*}$ as well as $L^{\infty}((0,t_{*});(W^{1,1}\cap W^{1,\infty})(\R^d))$ norms of $v^1[f^i],v^2[f^i]$ and $(L^1\cap L^{\infty})(\R^d)$ norms of $f^j_i(0)$. Applying the estimate 
\[
\norm{\divv v^1[f^i](s,\cdot)}_{W^{1,\infty}(\R^d)} \leq  \norm{\divv K_1}_{L^{\infty}(\R^d)} \sum_{k=1}^N \norm{\nabla f^i_k(s,\cdot)}_{L^{1}(\R^d)}
\]
together with the Gr\"onwall inequality applied to the function  $u(t)=\sum_{i,j=1}^N\norm{f^j_i(t)}_{L^{1}(\R^d)}$, we obtain the desired bound in $W^{1,1}(\R^d)$ (note that here we do not search for $N$ - independent bound). Then similarly we obtain the result in $W^{1,\infty}(\R^d)$.  To justify the formal  calculations above we may discuss the approximate system $\partial_t (f^j_i)^\ve + \divv ((f_i^j)^\ve v^{1,\ve}[(f^i)^\ve]= (f_i^j)^\ve v^{2,\ve}[(f^i)^\ve]$, where $v^{1,\ve}$ and $v^{2,\ve}$ are created from $v^1$ and $v^2$ by replacing kernels $K_1$ and $K_2$ by their convolution with mollifying kernel. Furthermore, we smooth out the initial condition by defining $(f^j_i)^{\ve}(0)$ as a convolution of $f^j_i(0)$ with mollifying kernel. Then, since $(f^j_i)^{\ve}(0), v^{1,\ve}[(f^i)^\ve]$ and $v^{2,\ve}[(f^i)^\ve]$ are $C^{\infty}(\R^d)$ so is $(f^j_i)^\ve$ and the above a priori estimates are rigorous for $(f^j_i)^\ve$ and, due to the properties of convolution with mollifier, does not depend upon $\ve$. Testing the equation against $\vf \in H^{1}(\R^d)$ and controlling the $\ve$ - independent bound of $(f^j_i)^{\ve}v^{1,\ve}[(f^i)^\ve]$ and $(f^j_i)^{\ve}v^{2,\ve}[(f^i)^\ve]$ in $L^{2}(\R^d)$ we obtain the $\ve$ - independent bound for $\partial_t (f^j_i)^{\ve}$ in $L^{\infty}((0,t_{*});H^{-1}(\R^d))$. Thus, by Aubin-Lions lemma $(f^j_i)^{\ve}$ is relatively compact in $C([0,t_{*}];L^{2}(B_s))$ for every $s>0$. Together with weak compactness, this allows us to pass to the limit in the weak formulation. By weak compactness the limiting $(\chi_i^j)$ also satisfies the uniform bounds in $(W^{1,1}\cap W^{1,\infty})(\R^d)$. Finally, we conclude that the limiting $(\chi^j_i)$ coincides with  $(f^j_i)$ applying the $L^{\infty}((0,t_{*});(\mathcal{M}_{+}(\R^d))^N)$ uniqueness of measure solutions  to system of transport equations with non-local flow  \cite[Theorem 1.1]{Crippa}.

\end{proof}
\begin{remark}
One may provide an alternative proof of the propagation of regularity of $f^j_i$ in time without referring to the equation (\ref{semic}). The proof is  based only on the flow representation $\xb_i(t) = h_i(t,X_i^0)$ and $\wjib(t)=\wjib^0\exp(\int_0^t\tilde{a}_i(s,h_i(s,X^0_i))ds)$ from the proof of Lemma \ref{exi}. See Remark 4 in the Appendix of \cite{nasza}.
\end{remark}

\subsection{Technical proofs from Section 3.3}

We begin with presenting the proof of Lemma \ref{ftau}.

\begin{proof}[Proof of Lemma \ref{ftau}]
 For any $T\in \Tree$  we introduce the modified operator 
 \[
 \tilde{\tau}(T,f)(t,x_1,\dots,x_n,\xi) = \int_{[0,1]^{n}} \prod_ {(k,m)\in E(T)} f(t,x_{m-1},\xi_k,\xi_m)d\xi_2\dots d\xi_{|T|}\Big|_{\xi_1=\xi},
 \]
where $n=|T|-1$. Clearly, $\tau(T,f)(t,x_1,\dots,x_n) = \izj  \tilde{\tau}(T,f)(t,x_1,\dots,x_n,\xi) d\xi$. We will inductively prove that for every $T\in \Tree$ there exists $F \in \mathcal{F}$  and permutation $\sigma$ such that $\tilde{\tau}(T,f)(t,x_1,\dots,x_n,\xi) = F(f)(t,x_{\sigma(1)},\dots,x_{\sigma(n)},\xi)$ for every $f$ as in the claim of lemma. 
 Let us begin with the only $T \in \Tree_2$. Then
 \[
 \tilde{\tau}(T,f)(t,x,\xi) = \izj f(t,x,\xi,\eta)d\eta = F_0(f)(t,x,\xi).
 \]
 Let us assume that for every $k \leq m$ for every $T \in \Tree_k$  the claim holds true. Let us take $T \in \Tree_{m+1}$ and assume that the root has a degree $l\leq m$. We index the vertices connected with a root by $i=2,\dots,l+1$. Then, we denote by $T_i$ the subtree rooted at vertex $i$ for $i=2,\dots,l+1$. Applying the induction hypothesis we have that, for $i=2,\dots,l+1$, there exist $F_i \in \mathcal{F}$  such that $\tilde{\tau}(T_i,f) = F_i^\sigma(f)$, where by $F_i^\sigma$ we denote the composition of a transform $F_i$ with the permutation given in the induction hypothesis. Hence,  we arrive at
 \[
 \tilde{\tau}(T,f)(t,x_1,\dots,x_n,\xi) = \int_{[0,1]^l} \prod_{i=2}^{l+1} \tilde{\tau}(T_i,f)(\xi_i)f(t,x_{i-1},\xi,\xi_i)\prod_{i=2}^{l+1}d\xi_i
 =\prod_{i=2}^{l+1}\izj F_i^\sigma(f)(\xi_i)f(t,x_{i-1},\xi,\xi_i)d\xi_i,
 \]
 where for each subtree $T_i$ the spatial indexes of $\tilde{\tau}(T_i,f)$ (dependent on the labeling of vertices in a tree) are distinct  and different from $\{1,\dots,l\}$.  From the construction of the algebra $\mathcal{F}$ we note that for every $i=2,\dots,l+1$ there exists $F_i^{*}\in \mathcal{F}$ such that $\izj F_i^\sigma(f)(\xi_i)f(t,x_{i-1},\xi,\xi_i)d\xi_i = F_i^{*}(f)(\xi)$ and finally there exists  $F\in \mathcal{F}$ such that $\prod_{i=2}^{l+1}F_i^{*}(f) = F(f)$. Thus $\tilde{\tau}(T,f) =F(f)$ up to a possible permutation of the order of spatial variables. Applying the principle of mathematical induction we finish the proof.
 \end{proof}

Below we provide the proof of Lemma \ref{apro}. 

\begin{proof}[Proof of Lemma \ref{apro}]
 We proceed by induction. Let us begin with $F_0$. At first, by Fubini theorem
\[
F_0(f^\ve)(t,x,\xi) = \izj (\rho^\ve * f)(t,x,\xi,\eta)d\eta = \izj  \rho^\ve(\xi-\xi')\izj f(t,x,\xi',d\eta) d\xi' = (\rho^\ve * F_0(f))(t,x,\xi).
\]
Hence, by the properties of mollifier $F_0(f^\ve) \rightarrow F_0(f)$ in $L^{1}((0,1);L^{1}(\R^d))$ for almost all $t \in (0,t_{*})$. Furthermore, by Lemma \ref{hmj} (see the estimate (\ref{etxi}))
\eqq{
\norm{f^\ve(t,\cdot,\cdot,\cdot)}_{L^{\infty}_\eta\mathcal{M}_{\xi}(\mathcal{M}(\R^d))} \leq \norm{f(t,\cdot,\cdot,\cdot)}_{L^{\infty}_\eta\mathcal{M}_{\xi}(\mathcal{M}(\R^d))}.
}{apboundb}
Hence, applying the estimate (\ref{fnestj}) for more regular $f^\ve$ together with (\ref{apboundb}) we obtain for any $n$- rank transform $F \in \mathcal{F}$
\eqq{
\norm{F(f^\ve)(t,\cdot,\cdot)}_{L^{\infty}((0,1);L^{1}(\R^{dn}))}\leq \norm{f^\ve(t,\cdot,\cdot,\cdot)}_{L^{\infty}_\eta(L^{1}_\xi(L^{1}(\R^{d}))}^n \leq  \norm{f(t,\cdot,\cdot,\cdot)}_{L^{\infty}_\eta(\mathcal{M}_{\xi};\mathcal{M}(\R^d))}^n.
}{apbounda}

Let us assume that for any $k < n$ all $k$-rank transforms satisfy the convergence $F(f^\ve) \rightarrow F(f)$ in $L^{1}((0,1);L^{1}(\R^{dk}))$. We will prove that the convergence holds also for  $n$-rank transform $F$. If $F = F^l\cdot F^k$, for some $k+l = n$, then using (\ref{apbounda}) and the induction hypothesis
\begin{align*}
 &\int_0^1 \int_{\R^{dn}}\abs{F(f^\ve)-F(f)}dx_1\dots dx_n d\eta \leq \int_0^1 \int_{\R^{dk}}\abs{F^k(f^\ve)-F^k(f)}dx_1\dots dx_k  \int_{\R^{dl}}\abs{F^l(f^\ve)}dx_{1}\dots dx_{l}d\eta\\
 & +\int_0^1 \int_{\R^{dk}}\abs{F^l(f^\ve)-F^l(f)}dx_1\dots dx_l  \int_{\R^{dl}}\abs{F^k(f)}dx_{1}\dots dx_{k}d\eta\\
 &\leq \norm{F^l(f^\ve)(t,\cdot,\cdot)}_{L^{\infty}((0,1);L^{1}(\R^{dl}))} \norm{F^{k}(f^\ve)-F^k(f)}_{L^{1}((0,1);L^{1}(\R^{dk}))} \\
 &+ \norm{F^k(f)(t,\cdot,\cdot)}_{L^{\infty}((0,1);L^{1}(\R^{dk}))} \norm{F^{l}(f^\ve)-F^l(f)}_{L^{1}((0,1);L^{1}(\R^{dl}))} \longrightarrow 0 \hd \hd \m{ as } \ve \rightarrow 0.
\end{align*}

In case $F(f)(t,x_1,\dots,x_n,\xi) = \izj \tilde{F}(f)(t,x_1,\dots,x_{n-1},\eta)f(t,x_{n},\xi,d\eta)$ for a $n-1$- rank transform $\tilde{F}$ we obtain desired convergence applying Lemma \ref{hmj}.
 Hence, in any case $F(f^\ve) \rightarrow F(f)$ in $L^{1}((0,1);L^{1}(\R^{dn}))$.  We finish the proof applying the mathematical induction principle.
\end{proof}

We finish this section with a proof of Lemma \ref{taustar}.

\begin{proof}[Proof of Lemma \ref{taustar}]
We will proceed by induction on the number of vertices in a tree, but to give more insights in the first step we will discuss also three vertex trees. At first we take the only possible two vertex tree $T_2$. Here, for simplicity we denote $\xi_1 = \xi, \xi_2 = \eta, x_1 = x$. Then in case $m = 1$, $\tau^1_{*}$ actually coincides with $\tilde{\tau}$ from the proof of Lemma \ref{ftau} and the convergence follows from Lemma \ref{apro}. Anyway, let us repeat the short reasoning
\[
\tau_{*}^1(T_2,f^{\ve})(x,\xi) = F_{0}(f^{\ve})(x,\xi) = \izj  \rho^\ve(\xi-\xi')\izj f(x,\xi',d\eta) d\xi' = \rho^\ve * F_0(f)(x,\xi)
\]
and hence $\tau_{*}^1(T_2,f^{\ve}) \rightarrow F_0(f) = \tau^1_{*}(T_2,f)$ in $L^{1}((0,1)\times \R^d)$  with 
\[
 \norm{\tau^1_{*}(T_2,f^{\ve})}_{L^{\infty}((0,1)\times \R^d)} \leq \norm{\tau^1_{*}(T_2,f)}_{L^{\infty}((0,1)\times \R^d)} \leq \norm{f}_{\W(L^{\infty}(\R^d))}.
\]
In the case $m=2$ we have
\[
\tau_{*}^1(T_2,f^{\ve})(x,\eta) =\izj f^{\ve}(x,\xi,\eta)d\xi= \izj  \izj \rho^\ve(\xi-\xi') f(x,\xi',\eta) d\xi'd\xi = \izj f(x,\xi',\eta)d\xi'
\]
and since we have not only convergence but equality the claim follows in this case as well. 
Let us now discuss two three vertex tree: $T_3^1, T_3^2$. Let us begin with the tree with the root and two leaves  $T_3^1$. If $m=1$ (root), then again our $\tau_*$ coincides with $\tilde{\tau}$ from the proof of Lemma \ref{ftau}. Anyway, we may write
\begin{align*}
\tau_{*}^1(T_3^1,f^{\ve})(x_1,x_2,\xi_1) &= \izj f^\ve(x_1,\xi_1,\xi_2)d\xi_2 \izj f^{\ve}(x_2,\xi_1,\xi_3)d\xi_3
\\
&= \left(\rho^{\ve}* \izj f(x_1,\cdot,d\eta)\right)(\xi_1)\left(\rho^{\ve}* \izj f(x_2,\cdot,d\xi_3)\right)(\xi_1).
\end{align*}
Both terms on the RHS converge in $L^{1}((0,1)\times (\R^d))$ in $\xi_1$ and are uniformly bounded  by $\norm{f}_{\W(L^{\infty}(\R^d))}$, hence $\tau^1_{*}(T_3^1,f^{\ve}) \rightarrow \tau^1_{*}(T_3^1,f)$ in $L^{1}((0,1)\times \R^{2d})$ with the estimate
\[
 \norm{\tau^1_{*}(T^1_3,f^{\ve})}_{L^{\infty}((0,1)\times \R^{2d})} \leq \norm{\tau^1_{*}(f)}_{L^{\infty}((0,1)\times \R^{2d})} \leq \norm{f}^2_{\W(L^{\infty}(\R^d))}.
\]
If $m=2$, (a leaf) then we have
\begin{align*}
\tau_{*}^2(T_3^1,f^{\ve})(x_1,x_2,\xi_2) &= \izj f^\ve(x_1,\xi_1,\xi_2) \izj f^{\ve}(x_2,\xi_1,\xi_3)d\xi_3 d\xi_1\\
&= \izj f^\ve(x_1,\xi_1,\xi_2)\left(\rho^{\ve}* \izj f(x_2,\cdot,d\xi_3)\right)(\xi_1)d
\xi_1
\end{align*}
and we obtain the desired convergence by Lemma \ref{hmj} with $g^{\ve}=\rho^{\ve}* \izj f(x_2,\cdot,d\xi_3)$. Furthermore, by Lemma \ref{impoest} and properties of convolution with mollifier 
\[
\norm{\tau^2_{*}(T^1_3,f^{\ve})}_{L^{\infty}((0,1)\times \R^{2d})} \leq \norm{f}^2_{\W(L^{\infty}(\R^d))}.
\]
If $m=3$ the argument is analogous.
Let us now discuss the tree $T^2_3$ with only one leaf  and three cases $m=1,2,3$. In case $m=1$ (a root) we have
\begin{align*}
\tau^1_{*}(T_3^2,f^{\ve})(x_1,x_2,\xi_1) &=  \izj f^\ve(x_1,\xi_1,\xi_2) \izj f^{\ve}(x_2,\xi_2,\xi_3)d\xi_3 d\xi_2
\\
&= \izj f^\ve(x_1,\xi_1,\xi_2)\left(\rho^{\ve}* \izj f(x_2,\cdot,d\xi_3)\right)(\xi_2)d
\xi_2
\end{align*}
and we obtain the convergence in this case again by Lemma \ref{hmj}.
In case $m=2$ (neither a root, nor a leaf) we have
\begin{align*}
\tau^2_{*}(T_3^2,f^{\ve})(x_1,x_2,\xi_2) &=  \izj f^\ve(x_1,\xi_1,\xi_2) d\xi_1 \izj f^{\ve}(x_2,\xi_2,\xi_3)d\xi_3 \\
&=\izj f(x_1,\xi_1,\xi_2) d\xi_1 \left(\rho^{\ve} *\izj f(x_2,\cdot,\xi_3)d\xi_3\right)(\xi_2).
\end{align*}
Since the convolution converges in $L^1$ and the integral $\izj f(x_1,\xi_1,\cdot) d\xi_1$ is bounded we obtain desired $L^1$ - convergence.
Finally, in case $m=3$ (a leaf) we get
\[
\tau^3_{*}(T_3^2,f^{\ve})(x_1,x_2,\xi_3) =  \izj\izj f^\ve(x_1,\xi_1,\xi_2) d\xi_1  f^{\ve}(x_2,\xi_2,\xi_3)d\xi_2 =\izj \izj f(x_1,\xi_1,\xi_2) d\xi_1 f^{\ve}(x_2,\xi_2,\xi_3)d\xi_2 
\]
and arrive at the desired $L^1$ - convergence applying Lemma \ref{hmj} to a constant sequence $g^\ve = \izj f(x_1,\xi_1,\cdot) d\xi_1$.
In all three cases $m=1,2,3$, we obtain the bound 
\[
\norm{\tau^m_{*}(T^2_3,f^{\ve})}_{L^{\infty}((0,1)\times \R^{2d})} \leq \norm{f}^2_{\W(L^{\infty}(\R^{d}))}
\]
applying the properties of convolution with mollifier and in the case $m=3$ also Lemma \ref{impoest}.

Let us now assume that the claim holds for every $T \in \Tree_{\tilde{n}}$ for every $\tilde{n} \leq n$ and we take arbitrary $T \in \Tree_{n+1}$. We choose arbitrary $m \in \{1,\dots,|T|\}$. We denote by $T_m$ a subtree of $T$ rooted at vertex $m$.  If $m=1$ (a root) the claim follows from Lemma \ref{apro}. Otherwise we denote by $l_0$ the vertex that connects subtree $T_m$ with the rest of a tree, denoted by $T\setminus T_m$. Then
\[
\tau^m_{*}(T,f^{\ve})(x_1,\dots,x_{n},\xi_m) = \izj \tau^{l_0}(T \setminus T_m,f^\ve)(\xi_{l_0})f^{\ve}(x_{m-1},\xi_{l_0},\xi_m)dl_0 \tau_{*}^m(T_m,f^\ve)(\xi_m),
\]
where we omit some spatial indexes for clarity and we set $\tau_{*}^m(T_m,f^\ve)(\xi_m) \equiv 1$ in case when $m$ is a leaf. From the induction hypothesis $\tau_{*}^m(T_m,f^\ve) \rightarrow \tau_{*}^m(T_m,f)$ in $L^{1}$ and is uniformly bounded in $\ve$. Analogously $\tau^{l_0}(T \setminus T_m,f^\ve)\rightarrow \tau^{l_0}(T \setminus T_m,f)$ in $L^{1}$. Applying Lemma \ref{hmj} we arrive at $\tau^m_{*}(T,f^{\ve}) \rightarrow \tau^m_{*}(T,f)$ in $L^1((0,1)\times \R^{dn})$ and the estimate $\norm{\tau^m_{*}(T,f)}_{L^{\infty}((0,1)\times \R^{dn})} \leq \norm{f}^n_{\W(L^{\infty}(\R^d))}$ follows from the induction hypothesis and Lemma \ref{impoest}. Thus, we obtain the claim in case $T \in \Tree_{n+1}$ and we finish the proof applying mathematical induction principle. 
\end{proof}

\subsection{ Proof of Proposition \ref{asym} from Section 3.5}
\begin{proof}
We will prove the claim by induction. At first we note that for the only one 1-rank transform $F_0$ we have
\begin{align*}
\izj \norm{F_0(f)(\cdot,\xi)}_{L^{1}(\R^d\setminus B_R)}d\xi &\leq \izj \izj \norm{f(\cdot,\xi,\cdot)}_{\mathcal{M}(\R^d\setminus B_R)}(d\eta) d\xi =
  \izj \izj \norm{f(\cdot,\cdot,\eta)}_{\mathcal{M}(\R^d\setminus B_R)}(d\xi) d\eta \\
&  \leq \sup_{\eta \in (0,1)}\izj \norm{f(\cdot,\cdot,\eta)}_{\mathcal{M}(\R^d\setminus B_R)}(d\xi) =\norm{f}_{L^{\infty}_\eta(\mathcal{M}_\xi(\mathcal{M}(\R^d\setminus B_R)))}.
\end{align*}
Now let us fix $n \in \mathbb{N}$ and assume that for any transform of rank smaller than or equal to $n$ the thesis holds. We will prove that the claim is true also for any transform of rank $n+1$. Let us take arbitrary $F$ of rank $n+1$. Then the two scenarios are possible. Either there exists  $F^l,F^k$ of rank $l$ and $k$
respectively, with $l+k = n+1$, such that $F = F^l\cdot F^k$ or there exists a transform $F^n$ of rank $n$
such that $F(f)(x_1,\dots,x_{n+1},\xi) = \izj F^{n}(x_1,\dots,x_n,\eta)f(x_{n+1},\xi,d\eta)$. In the first case we apply the induction hypothesis to estimate as follows
\begin{align*}
    &\izj \norm{F(\cdot,\xi)}_{L^{1}(\R^{d(n+1)}\setminus B_R^{(n+1)})}d\xi \\& \leq 
\izj \norm{F^l(\cdot,\xi)}_{L^{1}(\R^{dl}\setminus B_R^{l})}\norm{F^k(\cdot,\xi)}_{L^{1}(\R^{dk})}d\xi + \izj \norm{F^l(\cdot,\xi)}_{L^{1}(\R^{dl})}\norm{F^k(\cdot,\xi)}_{L^{1}(\R^{dk}\setminus B_R^k)}d\xi \\
&\leq\sup_{\xi \in (0,1)}\norm{F^k(\cdot,\xi)}_{L^{1}(\R^{dk})} l \norm{f}_{\W(\mathcal{M}(\R^d))}^{l-1}\norm{f}_{L^{\infty}_{\eta}(\mathcal{M}_\xi(\mathcal{M}(\R^d\setminus B_R)))} 
\\
&+
\sup_{\xi \in (0,1)}\norm{F^l(\cdot,\xi)}_{L^{1}(\R^{dl})} k \norm{f}_{\W(\mathcal{M}(\R^d))}^{k-1}\norm{f}_{L^{\infty}_{\eta}(\mathcal{M}_\xi(\mathcal{M}(\R^d\setminus B_R)))} 
\\
&\leq \left(l \norm{f}^{k+l-1}_{\W(\mathcal{M}(\R^d))}
+ k \norm{f}^{k+l-1}_{\W(\mathcal{M}(\R^d))}\right)\norm{f}_{L^{\infty}_{\eta}(\mathcal{M}_\xi(\mathcal{M}(\R^d\setminus B_R)))}  = (n+1)\norm{f}^{n}_{\W(\mathcal{M}(\R^d))}\norm{f}_{L^{\infty}_{\eta}(\mathcal{M}_\xi(\mathcal{M}(\R^d\setminus B_R)))}, 
\end{align*}
where in the last estimate we applied Proposition \ref{fntau}.
In the second case, applying Lemma \ref{impoest} and induction hypothesis we have
\begin{align*}
    &\izj \norm{F(\cdot,\xi)}_{L^{1}(\R^{d(n+1)}\setminus B_R^{(n+1)})}d\xi \\
&\leq \izj \izj \norm{F^n(\cdot,\eta)}_{L^{1}(\R^{dn}\setminus B_R^{n})}\norm{f(\cdot,\xi,\cdot)}_{\mathcal{M}(\R^d)}(d\eta)d\xi + \izj \izj \norm{F^n(\cdot,\eta)}_{L^{1}(\R^{dn})}\norm{f(\cdot,\xi,\cdot)}_{\mathcal{M}(\R^d\setminus B_R)}(d\eta)d\xi \\
&\izj \norm{F^n(\cdot,\eta)}_{L^{1}(\R^{dn}\setminus B_R^{n})}d\eta \norm{f}_{L^{\infty}_\eta(\mathcal{M}_\xi(\mathcal{M}(\R^d))) } +
\norm{F^n}_{L^{\infty}((0,1);L^{1}(\R^{dn}))}\izj \izj \norm{f(\cdot,\xi,\cdot)}_{\mathcal{M}(\R^d\setminus B_R)}(d\eta)d\xi \\
&(n+1) \norm{f}_{\W(\mathcal{M}(\R^d))}^{n+1} \norm{f}_{L^{\infty}_\eta(\mathcal{M}_\xi(\mathcal{M}(\R^d\setminus B_R)))}.  
\end{align*}
Hence we obtain desired estimate for arbitrary transform of rank $n+1$. We finish the proof applying mathematical induction principle.  
\end{proof}


\begin{thebibliography}{99}
%\addcontentsline{toc}{Chapter}{{\bf Bibliography}}
%**********************************************************************************
{\footnotesize

\bibitem{Aover}
Ayi N., \textit{Graph and mean-field limits for interacting particle systems}, in N. D. Alikakos and C. Villani,
editors, Analysis Techniques for Mathematical Physics and Geometry – and Music, volume 2390 of
Lecture Notes in Mathematics. Springer, Cham, 2026.


\bibitem{A2026}
Ayi N., \textit{Mean-field limits for interacting particles on general adaptive dynamical networks}, arXiv preprint: arXiv:2601.03742, (2026).

\bibitem{ADover}
Ayi N., Duteil N. P. \textit{Large-population limits of non-exchangeable particle systems}, Active Particles, Volume 4: Theory, Models, Applications, 79-133, (2024).

\bibitem{AD2021}
Ayi N.,  Duteil N. P., \textit{Mean-field and graph limits for collective dynamics models with time-varying weights}, Journal of Differential Equations, 299, 65-110, (2021).

\nic{

\bibitem{BCN} Bet G., Coppini F.,  Nardi F. R.  \textit{Weakly interacting oscillators on dense random graphs. Journal of Applied Probability}, 61(1), 255-278, (2024).

\bibitem{BCCZ}
Borgs C., Chayes J., Cohn H., Zhao Y.,  \textit{An $L^p$ theory of sparse graph convergence I: Limits, sparse random graph models, and power law distributions}, Transactions of the American Mathematical Society, 372(5), 3019-3062, (2019).

\bibitem{BCCZ2}
Borgs C., Chayes J. T., Cohn H.,  Zhao Y., \textit{An $L^p$ theory of sparse graph convergence II: LD convergence, quotients and right convergence},  Ann. Probab. 46 (2018), no. 1, 337–396.


\bibitem{CM2019} Chiba H., Medvedev G.S., \textit{ The mean field analysis for the Kuramoto model on graphs I. The mean field
equation and transition point formulas}, Discrete Contin. Dyn. Syst. Ser. A 39 (2019), no. 1, 131–155.
21.

\bibitem{CM20192}
Chiba H., Medvedev, G.S., \textit{The mean field analysis of the Kuramoto model on graphs II. Asymptotic stability
of the incoherent state, center manifold reduction, and bifurcations}, Discrete Contin. Dyn. Syst. Ser. A 39 (2019),
no. 7, 3897–3921.
22.
\bibitem{CMM2018} Chiba, H., Medvedev G.S., Muzuhara M.S., \textit{Bifurcations in the Kuramoto model on graphs}, Chaos 28 (2018),073109.

}

\bibitem{BCG2024} Ben-Porat I., Carrillo J. A.,  Galtung, S. T. \textit{Mean field limit for one dimensional opinion dynamics with Coulomb interaction and time dependent weights}. Nonlinear Analysis, 240, 113462, (2024).


\bibitem{fizyka} Berner R., Gross T.,  Kuehn C., Kurths J., Yanchuk S., \textit{ Adaptive dynamical networks}, Physics Reports, 1031:1–59, 2023. 

\bibitem{Poyato2026}Cabrera-Nyst J.,  Poyato D. \textit{ Mean field limit of non-exchangeable interacting diffusions on co-evolutionary networks} arXiv preprint: arXiv:2606.21556, (2026).

\bibitem{Crippa}
Crippa, G.,  Lécureux-Mercier, M., \textit{  Existence and uniqueness of measure solutions for a system of continuity equations with non-local flow.} Nonlinear Differential Equations and Applications NoDEA, 20(3), 523-537, (2013).


\bibitem{DU1977}Diestel J.,  Uhl J. J., \textit{ Vector Measures}, Mathematical Surveys and Monographs, Vol. 15, American Mathematical Society, Providence (1977).

\bibitem{diPerna}  DiPerna R. J.,  Lions P.-L., \textit{ Ordinary differential equations, transport theory and Sobolev
spaces}, Inventiones Mathematicae 98 (1989), no. 3, 511–547.

\bibitem{Dudley} Dudley, R. M., \textit{ The speed of mean Glivenko-Cantelli convergence}, The Annals of Mathematical Statistics, 40(1), 40-50, (1969).

%\bibitem{GC2022} Gkogkas M. A., Kuehn C., \textit{ Graphop mean-field limits for Kuramoto-type models}, SIAM
%Journal on Applied Dynamical Systems, 21(1):248–283, 2022.

\bibitem{D2022}Duteil N.P., \textit{Mean-field limit of collective dynamics with time-varying weights}, Netw. Heterog. Media 17 (2) 129–161, (2022).


\bibitem{GKX2023} Gkogkas M. A., Kuehn C., Xu. C., \textit{ Continuum limits for adaptive network dynamics}, 
Communication in Mathematical Sciences, 21:83–106, (2023).

\bibitem{GKX2025}Gkogkas M. A., Kuehn C.,  Xu C., \textit{  Mean field limits of co-evolutionary signed heterogeneous networks}, European Journal of Applied Mathematics, 37(3), 643-686, (2026).

\bibitem{nasza} Gwiazda P., Ryszewska K., \textit{ A note on application of mean-field limit to non-exchangeable
non-conservative systems}, arXiv preprint:  	arXiv:2607.20014, (2026).

\bibitem{JPS2025} Jabin P. E., Poyato D.,  Soler J., \textit{ Mean‐field limit of non‐exchangeable systems}, Communications on Pure and Applied Mathematics, 78(4), 651-741, (2025).

%\bibitem{JZ2025} Jabin P. E.,  Zhou D. \textit{ The mean-field limit of sparse networks of integrate-and-fire neurons}, Annales de l'Institut Henri Poincaré C, 43(2), 273-343, (2025).

%\bibitem{KVM2018} Kaliuzhnyi-Verbovetskyi D.,  Medvedev G. S., \textit{ The mean field equation for the Kuramoto model on graph sequences with non-Lipschitz limit},  SIAM Journal on Mathematical Analysis, 50(3), 2441-2465, (2018).

%\bibitem{K1970} Kružkov, S. N., \textit{First order quasilinear equations in several independent variables},  Mathematics of the USSR-Sbornik, 10(2), 217, (1970).

\bibitem{Kuehn}  Kuehn C., Xu C., \textit{ Vlasov equations on digraph measures}, J. Differ. Equ. 339 (2022), 261–349.

%\bibitem{Lovasz} Lovász L., \textit{Large networks and graph limits}, (Vol. 60). American Mathematical Soc, (2012).

%\bibitem{LovaszSzegedy} Lovász L., Szegedy B.,  \textit{ Limits of dense graph sequences}, Journal of Combinatorial Theory, Series B, 96(6), 933-957, (2006).

\bibitem{MPD2019} McQuade S., Piccoli B., Pouradier Duteil N., \textit{ Social dynamics models with time-varying influence}, Mathematical Models and Methods in Applied Sciences, 29(04), 681-716, (2019).

\bibitem{Perthame} Perthame B. \textit{ Transport equations in biology}, Basel: Birkhäuser Basel, (2007).

\bibitem{Rudin}
Rudin, W. \textit{Functional analysis }, 2nd ed. International Series in Pure and Applied Mathematics. McGraw-Hill, Inc., New York, (1991). 

\bibitem{Singer} Singer, I. (1957). Linear functionals on the space of continuous mappings of a compact Hausdorff space into a Banach space. Rev. Math. Pures Appl., 2, 301-315.

\bibitem{Throm} Throm S., \textit{ Continuum limit for interacting systems on adaptive networks}, European
Journal of Applied Mathematics, 37(3), 706-720 (2026).

\bibitem{Throm2} Throm S., \textit{ Mean field limit for interacting systems on co-evolving networks}, arXiv preprint arXiv:2507.21312, (2025).

\bibitem{Zhou} Zhou D., \textit{ Non-exchangeable mean-field theory for adaptive weights: propagation of dissociatedness and graphon sampling lemma},  arXiv preprint: arXiv:2506.13587, (2025).











}
\end{thebibliography}
\end{document}